\RequirePackage[l2tabu, orthodox]{nag}
\RequirePackage{fixltx2e}
\documentclass[10pt,a4paper,reqno,oneside,final]{smfart}

\usepackage{leftidx}
\usepackage[titletoc]{appendix}
\usepackage{tikz}\usetikzlibrary{decorations.markings,matrix,arrows}
\usepackage{tikz-cd}
\usepackage{amsfonts}
\usepackage{amsthm}
\usepackage[T1]{fontenc}
\usepackage[mathscr]{eucal}
\usepackage{setspace}
\usepackage{amssymb,latexsym,amsmath,amscd}
\usepackage{graphicx}
\usepackage{showkeys}
\usepackage{layout}
\usepackage{enumerate}
\usepackage{indentfirst}
\usepackage[varg]{pxfonts}
\usepackage[stretch=10]{microtype}
\usepackage{booktabs}
\usepackage{xspace}
\usepackage{eufrak}
\usepackage{calrsfs}
\usepackage{filecontents}
\usepackage{mathtools}
\usepackage{titlesec}
\usepackage{fixltx2e}
\usepackage{todonotes}
\usepackage[backref = page, colorlinks  =  true,
     citecolor    = gray, pdfborder={0 0 0}]{hyperref}
\usepackage{scalerel,stackengine}
\usepackage{nameref}
\usepackage{cleveref}
\usepackage{calrsfs}
\usepackage{enumitem}

\titleformat{\paragraph}
{\normalfont\normalsize\bfseries}{\theparagraph}{1em}{}
\titlespacing*{\paragraph}
{0pt}{3.25ex plus 1ex minus .2ex}{1.5ex plus .2ex}

\makeatletter
\newcommand{\neutralize}[1]{\expandafter\let\csname c@#1\endcsname\count@}
\makeatother

\theoremstyle{plain}
\newtheorem{thm}{Theorem}[section]
\newtheorem*{thm*}{Theorem}

\newtheorem{lem}[thm]{Lemma}
\newtheorem*{claim}{Claim}
\newtheorem{pro}[thm]{Proposition}
\newtheorem{pro-def}[thm]{Proposition-Definition}
\newtheorem{lem-def}[thm]{Lemma-Definition}
\newtheorem{cor-def}[thm]{Corollary-Definition}
\newtheorem{cor}[thm]{Corollary}

\newtheorem{Hyp}[thm]{Hypotheses}

\theoremstyle{definition}

\newtheorem{Def}[thm]{Definition}
\newtheorem{rem}[thm]{Remark}

\theoremstyle{remark}

\newcommand{\ssec}{\subsection}
\newcommand{\sssec}{\subsubsection}

\newcommand\comm[1]{\textcolor{red}{#1}}

\newcommand{\ol}{\overline}
\newcommand{\ti}[1]{\tilde{#1}}

\newcommand{\vast}{\bBigg@{4}}
\newcommand{\Vast}{\bBigg@{5}}
\newcommand{\wt}{\widetilde}
\newcommand{\rwh}{\reallywidehat}
\newcommand{\wh}{\widehat}
\newcommand\ddfrac[2]{\frac{\displaystyle #1}{\displaystyle #2}}

\stackMath
\newcommand\reallywidehat[1]{%
\savestack{\tmpbox}{\stretchto{%
  \scaleto{%
    \scalerel*[\widthof{\ensuremath{#1}}]{\kern-.6pt\bigwedge\kern-.6pt}%
    {\rule[-\textheight/2]{1ex}{\textheight}}
  }{\textheight}%
}{0.5ex}}%
\stackon[1pt]{#1}{\tmpbox}%
}
\definecolor{armygreen}{rgb}{0.29, 0.33, 0.13}
\definecolor{ao(english)}{rgb}{0.5, 0.2, 0.0}

\newcommand{\bC}{\mathbf{C}}

\newcommand{\bH}{\mathbf{H}}

\newcommand{\bP}{\mathbf{P}}
\newcommand{\bQ}{\mathbf{Q}}
\newcommand{\bR}{\mathbf{R}}

\newcommand{\bW}{\mathbf{W}}
\newcommand{\bZ}{\mathbf{Z}}

\newcommand{\bm}{\mathbf{m}}

\newcommand{\cC}{\mathcal{C}}

\newcommand{\cE}{\mathcal{E}}
\newcommand{\cF}{\mathcal{F}}
\newcommand{\cG}{\mathcal{G}}
\newcommand{\cH}{\mathcal{H}}
\newcommand{\cI}{\mathcal{I}}
\newcommand{\cJ}{\mathcal{J}}
\newcommand{\cL}{\mathcal{L}}

\newcommand{\cO}{\mathcal{O}}

\newcommand{\cV}{\mathcal{V}}
\newcommand{\cW}{\mathcal{W}}
\newcommand{\cX}{\mathcal{X}}
\newcommand{\cY}{\mathcal{Y}}

\newcommand{\sX}{\mathscr{X}}

\newcommand{\fX}{\mathfrak{X}}

\newcommand{\gD}{\Delta}

\newcommand{\gL}{\Lambda}
\newcommand{\gO}{\Omega}
\newcommand{\gS}{\Sigma}
\newcommand{\gT}{\Theta}

\newcommand{\ga}{\alpha}
\newcommand{\gb}{\beta}
\newcommand{\gd}{\delta}
\newcommand{\gep}{\varepsilon}

\newcommand{\gk}{\kappa}
\newcommand{\gl}{\lambda}

\newcommand{\gs}{\sigma}
\newcommand{\gt}{\theta}

\newcommand{\Aut}{\mathrm{Aut}}

\newcommand{\Div}{\mathrm{Div}}

\newcommand{\Ext}{\mathrm{Ext}}

\newcommand{\Gal}{\mathrm{Gal}}

\newcommand{\Gr}{\mathrm{Gr}}

\newcommand{\Hom}{\mathrm{Hom}}

\newcommand{\Id}{\mathrm{Id}}

\newcommand{\Ima}{\mathrm{Im}}

\newcommand{\mmin}{\mathrm{min}}

\newcommand{\pr}{\mathrm{pr}}

\newcommand{\sm}{\mathrm{sm}}

\newcommand{\Spec}{\mathrm{Spec \ }}

\newcommand{\tors}{\mathrm{tors}}

\newcommand{\tr}{\mathrm{tr}}

\newcommand{\ie}{\emph{i.e.} }
\newcommand{\eg}{\emph{e.g.} }

\newcommand{\bss}{\backslash}
\newcommand{\cnec}{\mathrel{:=}}

\newcommand{\st}{\star}

\renewcommand{\(}{\left(}
\renewcommand{\)}{\right)}

\newcommand{\dto}{\dashrightarrow}

\newcommand{\hto}{\hookrightarrow}
\newcommand{\xto}[1]{\xrightarrow{ #1 }}

\makeatletter
\let\orgdescriptionlabel\descriptionlabel
\renewcommand*{\descriptionlabel}[1]{%
  \let\orglabel\label
  \let\label\@gobble
  \phantomsection
  \edef\@currentlabel{#1}%
  \let\label\orglabel
  \orgdescriptionlabel{#1}%
}
\makeatother
\tikzset{node distance=2cm, auto}

\numberwithin{equation}{section}

\title{Algebraic approximations of compact K\"ahler manifolds of algebraic codimension 1} 

\author{Hsueh-Yung Lin}

\email{hsuehyung.lin.math@gmail.com}

\begin{document}

\begin{altabstract}
For every compact K\"ahler manifold $X$ of algebraic dimension $a(X) = \dim X - 1$, we prove that $X$ has arbitrarily small deformations to some projective manifolds. 
\end{altabstract}

\maketitle

\section{Introduction}
\ssec{The Kodaira problem}\hfill

Let $X$ be a compact K\"ahler manifold. An \emph{algebraic approximation of $X$} is a deformation $\cX \to \gD$ of $X$ such that up to shrinking $\gD$, the subset of $\gD$ parameterizing projective manifolds is dense. The so-called Kodaira problem asks whether a compact K\"ahler manifold in general admits an algebraic approximation. For compact K\"ahler surfaces, the Kodaira problem has a positive answer~\cite[Theorem 16.1]{KodairaSurfaceII}; see also~\cite{Schrackdefo, CaoJApproxalg, GrafDefKod0, ClaudonToridefequiv, HYLbimkod1, ClaudonHorpi1} for other positive results. However starting from dimension 4, there exist compact K\"ahler manifolds which cannot deform to any projective manifold at all (even homotopically)~\cite{Voisincs}. Despite these works, the existence of algebraic approximations is still unknown for most compact K\"ahler manifolds.

Recall that the algebraic dimension $a(X)$ of a compact complex manifold $X$ is the transcendental degree (over $\bC$) of its field of meromorphic functions. For compact K\"ahler manifolds, those satisfying $a(X) = \dim X$ are exactly the projective ones by Moishezon's criterion~\cite{MoishezonKMP}. Among non-projective compact K\"ahler manifolds, $a(X) = \dim X - 1$ is the highest possible algebraic dimension. The aim of this text is to prove that  such manifolds have algebraic approximations.

\begin{thm}\label{thm-main0}
Every compact K\"ahler manifold $X$ of algebraic dimension $a(X) = \dim X - 1$ admits an algebraic approximation.
\end{thm}

An \emph{elliptic fibration} is a holomorphic surjective map $f : X \to B$ whose general fiber is an elliptic curve. The algebraic reduction of a compact K\"ahler manifold $X$ with $a(X) = \dim X - 1$ is bimeromorphic to an elliptic fibration over a projective variety~\cite[Theorem 12.4]{UenoClassAlgVar}. Conversely, a compact K\"ahler manifold $X$ bimeromorphic to an elliptic fibration $X' \to B$ over a projective variety $B$ has algebraic dimension 
$$a(X) \ge a(B) = \dim B = \dim X - 1.$$ 
Thus we have the following variant of Theorem~\ref{thm-main0}, which we will prove instead in this text.

\begin{thm}[Variant of Theorem~\ref{thm-main0}]\label{thm-mainfibellip}
Let $X$ be a compact K\"ahler manifold. If $X$ is bimeromorphic to the total space of an elliptic fibration over a projective variety, then $X$ has an algebraic approximation.
\end{thm}

In~\cite{ClaudonHorpi1}, it has already been proven that every compact K\"ahler manifold $X$ as in Theorem~\ref{thm-mainfibellip} is bimeromorphic to an elliptic fibration admitting an algebraic approximation. 

\begin{rem}
Theorem~\ref{thm-mainfibellip} fails in general if the base of the elliptic fibration is not assumed to be projective: the product of an elliptic curve with Voisin's example constructed in~\cite[Section 2]{Voisincs} is a counter-example~\cite[Proposition 1]{Voisincs}.
\end{rem}

The notion of algebraic approximations can be defined more generally for any compact complex manifold $X$: in the definition, we require instead that Moishezon manifolds (\ie compact complex manifolds $X'$ with $a(X') = \dim X'$) are dense in the deformation $\cX \to \gD$ of $X$ (see Definition~\ref{def-appalg}). Since small deformations of K\"ahler manifolds remain K\"ahler,  this definition of algebraic approximation coincides with the one introduced at the beginning of the text for compact K\"ahler manifolds by Moishezon's criterion. Recall that the Fujiki class $\cC$ is the collection of compact complex varieties which are bimeromorphic to compact K\"ahler manifolds.  As an immediate corollary of Theorem~\ref{thm-main0}, we have the following more general statement.

\begin{cor}\label{cor-main0}
Every compact complex variety $X$ in the Fujiki class $\cC$ with rational singularities and algebraic dimension $a(X) = \dim X - 1$ admits an algebraic approximation.
\end{cor}
\begin{proof}
Let $\ti{X} \to X$ be a bimeromorphic morphism from a compact K\"ahler manifold $\ti{X}$. Since $X$ has at worst rational singularities, every deformation of $\ti{X}$ induces a deformation of $\ti{X} \to X$~\cite[Theorem 2.1]{RanStabMap}. As $ a(\ti{X}) = \dim \ti{X} - 1$, Corollary~\ref{cor-main0} follows from Theorem~\ref{thm-main0}.
\end{proof}

\begin{rem}
Small deformations of varieties in the Fujiki class $\cC$ might no longer be in the Fujiki class $\cC$~\cite{CampanaCnotstable, LebrunPoonFujikiC} and in the definition of algebraic approximations, it is not required that nearby fibers of the deformation of $X$ remain in the Fujiki class $\cC$. 
\end{rem}

In~\cite{HYLkod3}, we will apply Theorem~\ref{thm-main0} to prove the existence of algebraic approximations for compact K\"ahler threefolds, which was one of our motivations for this work. More precisely, note that since $\gk(X) \le a(X)$ where $\gk(X)$ is the Kodaira dimension, Theorem~\ref{thm-main0} holds in particular for compact K\"ahler manifolds $X$ with $\gk(X) = \dim X - 1$. We will apply Theorem~\ref{thm-main0}  to cover the case of threefolds with $\gk = 2$.

The rest of the introduction is devoted to an overview of the proof of Theorem~\ref{thm-mainfibellip}. 

\ssec{Algebraic approximations and bimeromorphic models}\label{ssec-choixbim}\hfill

Given a compact K\"ahler manifold $X$ as in Theorem~\ref{thm-mainfibellip}, we will prove Theorem~\ref{thm-mainfibellip} by constructing a bimeromorphic model $X \dto X'$ of $X$ together with an algebraic approximation of $X'$ which induces a deformation of $X \dto X'$.

In general, based on Ancona, Tomassini, and Bingener's work on formal modifications~\cite{BingenerModformelles}, we will observe that if $X'$ is a bimeromorphic model of a given compact complex variety $X$ and $Y \subset X'$ a subvariety containing the locus that are modified under the bimeromorphic map $X' \dto X$, then a deformation $\Pi : \cX' \to \gD$ of $X'$ induces a deformation of $X \dto X'$ provided $\Pi$ induces a trivial deformation of the completion $\hat{Y}$ of $X'$ along $Y$.\footnote{Usually the completion of $X$ along a subvariety $Y$ is denoted by $\hat{X}$. In this text, since we want to keep track of the subvariety along which the completion is done, we will often use the notation $\hat{Y}$ instead of $\hat{X}$.} In particular if $X$ is a compact K\"ahler manifold, then an algebraic approximation of $X'$ preserving $\hat{Y}$ induces an algebraic approximation of $X$. The exact statement that we will prove and apply is Proposition~\ref{pro-Formalgapp}, allowing us to transform the problem of finding an algebraic approximation of $X$ into the problem of finding a bimeromorphic model $X'$ of $X$ together with an algebraic approximation $\cX' \to \gD$ of $X'$ preserving $\hat{Y}$ for a certain subvariety $Y \subset X'$. 

Let us return to the case where $X$ is a compact complex manifold as in Theorem~\ref{thm-mainfibellip}. By Hironaka's desingularization, $X$ is bimeromorphic to the total space of an elliptic fibration $f' : X' \to B$ over a projective manifold $B$ such that the discriminant locus $\gD \subset B$ is a simple normal crossing divisor. We can find a finite Galois cover $r : \ti{B} \to B$ such that the base change $\ti{f}' : \ti{X}' = X' \times_B \ti{B} \to \ti{B}$ has local meromorphic sections at every point of $\ti{B}$ and the local monodromies of $\bH = (R^1\ti{f}'_*\bZ)_{| \ti{B} \bss \ti{\gD}}$ around $\ti{\gD} = r^{-1}(\gD)$ are unipotent. In order to simplify the discussion, we shall first assume that the finite cover $\ti{B} \to B$ is trivial so that $f'$ already satisfies the later properties. Tacitly in the following discussion, these are the properties that we need in most of the arguments involved in the proof of Theorem~\ref{thm-mainfibellip}.

  As we mentioned above, we will prove Theorem~\ref{thm-mainfibellip} by constructing a bimeromorphic model $X \dto \sX$ of $X$ together with an algebraic approximation of $\sX$ preserving $\hat{Y}$, where   $Y \subset \sX$ is a subvariety containing the locus that are modified under the bimeromorphic map $\sX \dto X$. Up to a finite Galois cover, the bimeromorphic model that we consider will be the total space of some particular elliptic fibration $f : \sX \to B$ called \emph{tautological model}, and the algebraic approximation of $\sX$ will be realized by a subfamily of the \emph{tautological family} associated to $f$. The next paragraph is an informal discussion on how these notions are introduced. For the details, the reader is referred to Section~\ref{sec-locW}.  

\ssec{Elliptic fibrations and tautological families}\label{ssec-fibell}\hfill

Let $B$ be a compact complex manifold and $\gD \subset B$ a normal crossing divisor. To each weight one variation of Hodge structures $\bH$ of rank 2 over $B^\st \cnec B \bss \gD$, we can associate a unique minimal Weierstra{\ss} fibration $p : W\to B$, which is an elliptic fibration satisfying $(R^1p_*\bZ)_{| B^\st} \simeq \bH$~\cite[Theorem 2.5]{NakayamaWeier}. A minimal Weierstra{\ss} fibration can be considered as a standard compactification of the Jacobian fibration $J \to B^\st$ associated to $\bH$. 

Let $\cE(B,\gD,\bH)$ be the set of bimeromorphic classes of elliptic fibrations $f' : X' \to B$ such that $f'$ is smooth over $B^\st$ with $(R^1f'_*\bZ)_{| B^\st} \simeq \bH$ and  admits local meromorphic sections over every point of $B$. Like in the smooth case where smooth elliptic fibrations are torsors under their Jacobian fibrations, elements of $\cE(B,\gD,\bH)$ can roughly be regarded as "meromorphic $W$-torsors" where $p:W \to B$ is the minimal Weierstra{\ss} elliptic fibration associated to $\bH$. So if $\cJ_{\bH/B}$ denotes the sheaf of abelian groups of germs of meromorphic sections of $p$, then we have an injective map 
$$\eta : \cE(B,\gD,\bH) \hto H^1(B,\cJ_{\bH/B}).$$

In general, not every element in $H^1(B,\cJ_{\bH/B})$ corresponds to (the bimeromorphic class of) an elliptic fibration. However, consider the subsheaf $\cJ^W_{\bH/B} \subset \cJ_{\bH/B}$ of germs of local holomorphic sections of $p : W \to B$ whose image is contained in the smooth locus of $p$, then each element $\eta \in H^1(B,\cJ^W_{\bH/B})$ corresponds to some elliptic fibration $p^\eta : W^\eta \to B$, called the \emph{locally Weierstra{\ss} fibration twisted by $\eta$}.  Torsion points of $H^1(B,\cJ^W_{\bH/B})$ correspond to elliptic fibrations that are projective. The sheaf $\cJ^W_{\bH/B}$ lies in the short exact sequence
\begin{equation}\label{suitex-introJW}
\begin{tikzcd}[cramped, row sep = 0, column sep = 40]
0  \arrow[r] & j_* \bH \arrow[r] &\cL_{\bH/B} \arrow[r,"\exp"] & \cJ^W_{\bH/B} \arrow[r] & 0. 
\end{tikzcd}
\end{equation}
where $j : B^\st \hto B$ is the inclusion and $\cL_{\bH/B} = R^1p_*\cO_W$. Each $\eta \in H^1(B,\cJ^W_{\bH/B})$ comes equipped with a family of elliptic fibrations
$$ \Pi : \cW \to B \times V \to V \cnec H^1(B,\cL_{\bH/B})$$
parameterized by $V$ that is tautological in the sense that the elliptic fibration parameterized by $t \in V$ corresponds to $\eta + \exp(t) \in H^1(B,\cJ^W_{\bH/B})$. 

More generally, let $f' : X' \to B$ be an elliptic fibration such that $X'$ is a compact K\"ahler manifold. Assume that the bimeromorphic class of $f'$ lies in $\cE(B,\gD,\bH)$, then we can construct a bimeromorphic model $f : \sX \to B$ of $f'$ together with a family of elliptic fibrations
$$ \Pi : \cX \to B \times V \to V$$
which contains $f$ as the central fiber and is tautological in the sense that the bimeromorphic class of the elliptic fibration parameterized by $t \in V$ corresponds to $\eta(f) + \exp'(t)$ where $\exp'$ is the composition of 
$$\exp : H^1(B,\cL_{\bH/B}) \to H^1(B,\cJ^W_{\bH/B})$$
 with $H^1(B,\cJ^W_{\bH/B}) \to H^1(B,\cJ_{\bH/B})$ (see Proposition-Definition~\ref{pro-locWeimodHdg}). We call $f : \sX \to B$ a \emph{tautological model} of $f'$ and $\Pi$ the \emph{tautological family} associated to $f$. Such a family has already been constructed in~\cite{ClaudonHorpi1} and proven to be an algebraic approximation of $f$.  Up to some finite Galois cover, this is how the bimeromorphic version of Theorem~\ref{thm-mainfibellip} was proven in~\cite{ClaudonHorpi1}. 

However, for such a bimeromorphic model $f$, there may be no algebraic approximation of $\sX$ contained in $\Pi : \cW \to V$ as a subfamily \emph{preserving the modification $X \dto \sX$} (for instance, this will be the case as long as the subvariety $Y \subset \sX$ along which $\sX$ is modified is non-algebraic). To prove Theorem~\ref{thm-mainfibellip}, we need to refine our choice of bimeromorphic model. Let $Y_\mmin \subset \sX$ be the minimal subvariety which is modified under the induced bimeromorphic map $\sX \dto X$ and let $Y \cnec f^{-1}(f(Y_\mmin))$. Up to replacing $f'$ with a more carefully constructed bimeromorphic model (see Step 2, 3, and 4 in the proof of Lemma~\ref{lem-red}), we may assume that the projection $Y \to f(Y)$ has a multi-section. It is for the pair $(\sX,Y)$ that we will prove that the tautological family $\Pi$ associated to $f : \sX \to B$ contains a subfamily which is an algebraic approximation of $\sX$ preserving the completion $\hat{Y}$ of $\sX$ along $Y$. Eventually it can be reduced to the case where $f$ is a locally Weierstra{\ss} fibration twisted by some $\eta \in H^1(B,\cJ^W_{\bH/B})$ and we will find algebraic approximations in the associated tautological family using Hodge theory, as we will discuss in the next paragraph. 

\ssec{The Hodge theory of Weierstra{\ss} fibrations}\hfill

With the same notations introduced in~\ref{ssec-fibell}, a new ingredient in this work allowing us to go further than~\cite{ClaudonHorpi1} is the observation that $H^1(B,j_*\bH)$ can be endowed with a natural pure $\bZ$-Hodge structure of weight 2 such that the map $H^1(B,j_*\bH) \to H^1(B,\cL_{\bH/B})$ induced by~\eqref{suitex-introJW} is the projection to the $(0,2)$-part (see Lemma~\ref{lem-morStHdgminW}). As $j_*\bH \to \cL_{\bH/B}$ is isomorphic to $R^1p_*\bZ \to R^1p_*\cO_W$ (Lemma~\ref{lem-isophi}) where $p : W \to B$ is the minimal Weierstra{\ss} fibration associated to $\bH$,  when $p$ is smooth this result is a particular case of~\cite[Theorem 2.9]{ZuckerHdgL2}. The Hodge-theoretic interpretation of $H^1(B,j_*\bH) \to H^1(B,\cL_{\bH/B})$ allows us to show without difficulties that for every $\eta \in H^1(B,\cJ^W_{\bH/B})$, the subset of points $t \in V = H^1(B,\cL_{\bH/B})$ such that $\eta + \exp(t)$ is contained in the torsion part of $H^1(B,\cJ^W_{\bH/B})$ is dense in $V$. This also gives an alternative way (in comparison to the proof of~\cite[Theorem 3.25]{ClaudonHorpi1}) closer to the spirit of~\cite{ClaudonToridefequiv, HYLbimkod1} to show that the tautological family $\Pi$ is an algebraic approximation.

Recall that we want to prove that the tautological family $ \Pi : \cW \to B \times V \to V$ associated to $\eta \in H^1(B,\cJ^W_{\bH/B})$ contains a subfamily which is an algebraic approximation of $ p^\eta : W^\eta \to B$ preserving the completion $\hat{Y}$. First of all, it is easy to see that if $\imath : Z = p^\eta(Y) \hto B$ is the inclusion, then the subfamily of $\Pi$ parameterized by
$$ V_Z \cnec \ker \( \imath^* :H^1(B,\cL_{\bH/B}) \to H^1(Z,\imath^*\cL_{\bH/B})\),$$ 
 is a deformation of $p^\eta$ preserving $Y \to Z$ (Lemma~\ref{lem-presWtauto}). However, we do not know whether $\imath^*$ is the $(0,2)$-part of a morphism of pure Hodge structures of weight 2, so cannot (use Hodge theory to) conclude  that the subfamily parameterized by $V_Z$ is an algebraic approximation. So instead of $\imath : Z \hto B$, we consider a log-desingularization $\ti{\imath} : \ti{Z} \to Z \hto B$ of the pair $(\ol{Z \bss \gD}, \ol{Z \bss \gD} \cap \gD)$.  By doing so, the pullback 
$$\ti{\imath}^* :  H^1(B,\cL_{\bH/B}) \to H^1(\ti{Z},\ti{\imath}^*\cL_{\bH/B})$$ 
will be the $(0,2)$-part of the morphism of pure Hodge structures
$$\ti{\imath}^* : H^1(B,j_*\bH) \to H^1({\ti{Z}},j_{\ti{Z}*}\bH_{\ti{Z}})$$
of weight 2. We will then show that
$$V_{\ti{Z}} \cnec \ker\(\ti{\imath}^* : H^1(B, \cL_{\bH/B}) \to H^1(\ti{Z}, \imath^*\cL_{\bH/B})\) \subset V_Z,$$
so the subfamily of $\Pi$ parameterized by $V_{\ti{Z}}$ also preserves $Y \to Z$ (Proposition~\ref{pro-desingpre}). Based on the assumption that $Y \to Z$ has a multi-section, we will also show that this subfamily is an algebraic approximation of $p^\eta$  (Proposition~\ref{pro-densepair0}).
 The proofs of both Proposition~\ref{pro-desingpre} and Proposition~\ref{pro-densepair0} use Hodge theory in an essential way. Another property that we can prove by Hodge theory is that $V_{\ti{Z}}$ contains a lattice $\gL$ such that for every $t \in V_{\ti{Z}}$ and $\gl \in \gL$, the elliptic fibrations parameterized by $t$ and $t + \gl$ are isomorphic (Lemma~\ref{lem-latticeisom}). We summarize these results as follows for the sake of reference. 

\begin{pro}[Proposition~\ref{pro-densepair0} + Lemma~\ref{lem-latticeisom} + Proposition~\ref{pro-desingpre} without the $G$-action]\label{pro-introdensepres}
If $Y \to Z$ has a multi-section, then the tautological family $\Pi$ contains a subfamily 
$$\Pi' : \cW_{\ti{Z}} \xto{q} B \times V_{\ti{Z}} \to V_{\ti{Z}}$$ 
 parameterized by a linear subspace $V_{\ti{Z}} \subset H^1(B,\cL_{\bH/B})$ which is an algebraic approximation of $p^\eta$ preserving $Y \to Z$. Moreover, there exists a lattice $\gL \subset V_{\ti{Z}}$ such that $t$ and $t + \gl$ parameterize isomorphic elliptic fibrations for each $t \in V_{\ti{Z}}$ and $\gl \in \gL$.
\end{pro}

\begin{rem}
The assumption in Proposition~\ref{pro-introdensepres} that $Y \to Z$ has a multi-section is also an obvious necessary condition for $p^\eta$ to have an algebraic approximation preserving $Y \to Z$. Indeed, if such an algebraic approximation exists, then $Y \to Z$ is algebraic, thus admits a multi-section.
\end{rem}

\ssec{From subvarieties to the completions and the end of the proof}\hfill

The geometric properties of the family $\Pi'$ described in Proposition~\ref{pro-introdensepres} allow us to apply Lemma~\ref{lem-prescompact} to deduce that $\Pi'$ preserves in fact the completion $\hat{Y} \to \hat{Z}$ of $p^\eta$ along $Y \to Z$. The argument goes roughly as follows: It suffices to prove by induction that the induced deformation $\cY_n \to Z_n \times V_{\ti{Z}} \to V_{\ti{Z}}$ of the $n$-th infinitesimal neighborhood $Y_n \to Z_n$ of $Y \to Z$ is trivial. Due to the flatness of $q$, we will see that each fiber $\cY_{n,t}$ of $\cY_n \to V_{\ti{Z}}$ is a square-zero extension of $Y_{n-1}$ by a sheaf $\cF$  independent of $t$. Thus we have a holomorphic map $\phi: V_{\ti{Z}} \to \Ext^1(L^\bullet_{Y_{n-1}},\cF)$, which is constant if and only if $\cY_n \to V_{\ti{Z}}$ is a trivial deformation where $\Ext^1(L^\bullet_{Y_{n-1}},\cF)$ is the vector space parameterizing square-zero extension of $Y_{n-1}$ by $\cF$. We will see that the last statement of Proposition~\ref{pro-introdensepres} together with the induction hypothesis implies that $\phi$ factorizes through $V_{\ti{Z}} / \gL$, which is a complex torus. Therefore $\phi$ is indeed constant. 

This is an overview of the proof of Theorem~\ref{thm-mainfibellip} under the assumption that the Galois cover of $r : \ti{B} \to B$ in~\ref{ssec-choixbim} is trivial. If $r$ is non-trivial with Galois group $G$, we replace $f'$ in~\ref{ssec-choixbim} with $\ti{f}' : \ti{X}' = X' \times_B \ti{B} \to \ti{B}$ and run the whole argument sketched above for $\ti{f}'$ in the $G$-equivariant setting. The algebraic approximation of the chosen bimeromorphic model will be realized by the finite quotient of a subfamily of the tautological family.

\begin{rem} 
Instead of constructing an algebraic approximation of $X$ preserving $\hat{Y}$, one could have tried to find directly an algebraic approximation $\cX \to \gD$ of $X$ such that a neighborhood $U \subset X$ of $Y$ deforms trivially in $\cX \to \gD$, which is the way we prove the existence of algebraic approximations of compact K\"ahler threefolds with $\gk \le 1$ in~\cite{HYLkod3}. However for the elliptic fibrations that we consider in this text, it is difficult (and maybe impossible) to find such algebraic approximations in the tautological families. This is why we choose to verify the weaker property that the completion $\hat{Y}$ is preserved along $\cX \to \gD$ and use Proposition~\ref{pro-Formalgapp} to conclude.
\end{rem}

This text is organized as follows. We will first introduce some basic terminology including various definitions of deformations of complex spaces and maps in Section~\ref{sec-prelim}. In Section~\ref{sec-locW}, we will recall and improve some results of the theory of elliptic fibrations developed in~\cite{NakayamaWeier, NakayamaLoc, NakayamaGlob, ClaudonHorpi1}. It is also in Section~\ref{sec-locW} that we will construct the tautological families (Proposition-Definition~\ref{pro-locWeimodHdg}). Section~\ref{sec-HdgTaut} is the Hodge-theoretic part of the text. We will prove Proposition~\ref{pro-densepair0}, Proposition~\ref{pro-desingpre}, and other related results therein. In Section~\ref{sec-preuve}, we will prove Proposition~\ref{pro-Formalgapp} and finally the main result Theorem~\ref{thm-mainfibellip}.

\section{Preliminaries}\label{sec-prelim}

\ssec{Notation and terminology}
\hfill

\sssec{Fibrations and deformations}\hfill

In this text, A \emph{fibration} is a proper holomorphic surjective map $f : X \to B$ between complex spaces with connected fibers. The fiber of $f$ over $b \in B$ will frequently be denoted by $X_b$. A \emph{multi-section} of $f$ is a subvariety $S \subset X$ such that $f_{|S}$ is surjective and generically finite. A deformation of a complex space $X$ is a flat surjective morphism $\Pi : \cX \to \gD$ containing $X$ as a fiber. Let $f : X \to B$ be a holomorphic map. A \emph{deformation of $f$} is a composition $\Pi : \cX \xto{q} B \times \gD \xto{\pr_2} \gD$ such that $\Pi$ is flat  and $q_{|\cX_o} : \cX_o \to B$ equals $f$ for some $o \in \gD$. Note that in our definition, a deformation of $f$ always preserves the base $B$. 

Let $G$ be a group and $X$ a complex space endowed with a $G$-action. We say that a deformation $\Pi : \cX \to \gD$ of $X$ \emph{preserves the $G$-action} if there exists a fiberwise $G$-action on $\cX \to \gD$ extending the given $G$-action on $X$. Similarly, let $f : X \to B$ be a $G$-equivariant map. We say that a deformation $\Pi : \cX \xto{q} B \times \gD \xto{\pr_2} \gD$ of $f$ \emph{preserves the $G$-action} if there exists a fiberwise $G$-action on $\cX \to \gD$ extending the $G$-action on $X$ such that $q$ is $G$-equivariant.

\begin{Def}\label{def-preserver}
\hfill
\begin{enumerate}[label = \roman{enumi})]
\item \label{def-preserverSV} Let $\Pi : \cX \to \gD$ be a deformation of a complex variety $X$ and let $Y \subset X$ be an analytic subset. We say that $\Pi$ \emph{preserves $Y$} if there exists $\cY \subset \cX$ such that $\cY$ is isomorphic to $Y \times \gD$ over $\gD$. 
\item \label{def-preserverSVF} We say that $\Pi$ \emph{preserves the completion $\hat{Y}$} of $X$ along $Y$ if there exist $\cY \subset \cX$ and an isomorphism $\hat{\cY} \simeq \hat{Y} \times \gD$ over $\gD$ where $\hat{\cY}$ is the completion of $\cX$ along $\cY$.
\item \label{def-preserverF} Let $\Pi : \cX \xto{q} B \times \gD \to \gD$ be a deformation of $f : X \to B$ and let $Z$ be an analytic subset of $B$. We say that $\Pi$ \emph{preserves $Y \cnec f^{-1}(Z) \to Z$} if $q^{-1}(Z \times \gD)$ is isomorphic to $Y \times \gD$ over $B \times \gD$. We define in a similar way deformations of $f$ that preserve the completion $\hat{Y} \to \hat{Z}$ of $f$ along $Y \to Z$.
\item  \label{def-preserverG} Let $G$ be a group acting on a complex variety $X$ and $Y \subset X$ a $G$-stable analytic subset. Let $\Pi : \cX \to \gD$ be a deformation of $X$ preserving the $G$-action. We say that $\Pi$ \emph{preserves $G$-equivariantly $Y$} if there exists a $G$-stable subvariety $\cY \subset \cX$ together with a $G$-equivariant isomorphism $\cY \simeq Y \times \gD$ over $\gD$. The $G$-equivariant versions of ii) and iii) are defined similarly.
\end{enumerate}
\end{Def}

\begin{Def}\label{def-loctriv}
\hfill
\begin{enumerate}[label = \roman{enumi})]
\item A deformation $\Pi : \cX \xto{q} B \times \gD \to \gD$ of $f : X \to B$ is said to be \emph{locally trivial (over $B$)} if there exists an open cover $\{U_i\}$ of $B$ such that $q^{-1}(U_i \times \gD) \simeq f^{-1}(U_i) \times \gD$ over $U_i \times \gD$. 
\item Let $G$ be a group and $f : X \to B$ a $G$-equivariant map. We say that $\Pi$ is \emph{$G$-equivariantly locally trivial} if $\Pi$ preserves the $G$-action and  the isomorphisms $q^{-1}(U_i \times \gD) \simeq f^{-1}(U_i) \times \gD$ above are $G$-equivariant for some $G$-invariant open cover $\{U_i\}$ of $B$.
\end{enumerate}
\end{Def}

Obviously, the quotient of a $G$-equivariantly locally trivial deformation is a locally trivial deformation.

\begin{lem}[{\cite[Lemma 2.2]{HYLbimkod1}}]\label{lem-Gquotloctriv}
If $\Pi : \cX \xto{q} B \times \gD \to \gD$ is a $G$-equivariantly locally trivial deformation of a $G$-equivariant fibration $f : X \to B$ for some finite group $G$, then the quotient $\cX/G \to (B/G) \times \gD \to \gD$ is a locally trivial deformation of $X/G \to B/G$.
\end{lem}

Now we come to the notion of algebraic approximations. Recall that a compact complex variety $X$ is called Moishezon if its algebraic dimension $a(X)$ equals $\dim X$.

\begin{Def}[Algebraic approximation]\label{def-appalg}
Let $X$ be a compact complex variety. An \emph{algebraic approximation} of $X$ is a deformation $\pi : \cX \to \gD$ of $X$ such that the subset of points in $\gD$ parameterizing Moishezon varieties is dense for the Euclidean topology.  
\end{Def}

\sssec{Desingularizations and Stein factorizations}\label{sssec-desingSt}
\hfill

Let $X$ be a compact complex variety. In this text, a \emph{desingularization} of $X$ is a bimeromorphic projective morphism $\nu : \ti{X} \to X$ such that $\ti{X}$ is smooth and if $X_{\sm} \subset X$ denotes the smooth locus of $X$, then $\nu_{| \nu^{-1}(X_{\sm})}$ is biholomorphic onto $X_{\sm}$. By Hironaka's theorem, desingularizations of a compact complex variety $X$ always exist. If, moreover, $X$ is endowed with a group action by some group $G$, then as a consequence of the existence of functorial desingularizations~\cite[Theorem 3.45]{KollarLRS}, we can always assume that  $\nu : \ti{X} \to X$ is $G$-equivariant. 

Let $B$ be a complex projective variety and $\gD \subset B$ a subvariety. A \emph{log-desingularization} of the pair $(B,\gD)$ is a birational morphism $\nu : \ti{B} \to B$ from a smooth projective variety $\ti{B}$ such that $\nu^{-1}(\gD)$ is a simple normal crossing divisor. Once again due to Hironaka, log-desingularizations always exist. If $B$ is endowed with a $G$-action for some finite group $G$ such that $\gD$ is $G$-stable, then we can always assume that  $\nu :  \ti{B} \to B$ is $G$-equivariant~\cite{AbramovichWangGequivres}. 

For any dominant generically finite meromorphic map $f : X \dto Y$ between two normal complex spaces, a \emph{Stein factorization} of $f$ is defined to be the Stein factorization $\ti{X} \to Z \to Y$ of $\ti{f} : \ti{X} \to Y$ where $\ti{f}$ is a resolution of $f$ such that $\ti{X}$ is normal. Note that the finite map $Z \to Y$ in the Stein factorization does not depend on the choice of $\ti{f}$: indeed, if $\ti{f}' : \ti{X}' \to Y$ is another resolution of $f$ such that $\ti{X}'$ is normal, then we can find another normal complex space $\ti{X}''$ and bimeromorphic surjective maps $\nu : \ti{X}'' \to \ti{X}$ and $\nu' : \ti{X}'' \to \ti{X}'$. As both $\ti{X}'' \to \ti{X}$ and $\ti{X}'' \to \ti{X}'$ have connected fibers by (the complex analytic version of) Zariski's main theorem, the finite maps in the Stein factorizations of $\ti{f}$ and $\ti{f}'$ both coincide with the finite map in the Stein factorization of $\ti{f} \circ \nu = \ti{f}' \circ \nu' :  \ti{X}'' \to Y$.

\ssec{Formal modifications}\hfill

Let $X$ be a complex space and $Y \subset X$ a complex subspace. A \emph{modification} of $X$ along $Y$ is a surjective bimeromorphic morphism of complex spaces $\nu : X' \to X$ such that $\nu$ maps $X' \bss \nu^{-1}(Y)$ biholomorphically onto $X \bss Y$. The completion $\hat{Y'} \to \hat{Y}$ of $\nu$ along $Y' \cnec \nu^{-1}(Y) \to Y$ is a morphism of formal complex spaces and there exists a generalization of such morphisms (of formal complex spaces) called \emph{formal modifications}. We refer to~\cite[p.39]{ModAnLMS} for a definition of formal modifications, which we will not need in this text. The only examples of formal modifications that we will encounter are contained in the following proposition.

\begin{pro}[{\cite[p.40, Proposition 4 and 5]{ModAnLMS}}]\label{pro-modforex}
\hfill
\begin{enumerate}[label = \roman{enumi})]
\item The formal completion of a modification is a formal modification. More precisely, if $\nu : X' \to X$ is a modification along $Y \subset X$ and $Y' = \nu^{-1}(Y)$, then the induced map $\hat{\nu} : \hat{Y'} \to \hat{Y}$ is a formal modification. 
\item If $\hat{\nu} : \fX' \to \fX$ is a formal modification, then for every complex space $Y$, $(\hat{\nu} \times \Id_Y) : \fX' \times Y \to \fX \times Y$ is also a formal modification.
\end{enumerate}
\end{pro}

Now we state the theorem of Ancona, Tomassini, and Bingener about the existence and uniqueness of the extension of formal modifications mentioned in the introduction, which serves as one of the key ingredients of the proof of Theorem~\ref{thm-mainfibellip}. 

\begin{thm}[Ancona-Tomassini-Bingener{~\cite[Theorem 8.1 and 9.1]{BingenerModformelles}}]\label{thm-ATB}
 \hfill

Let $\hat{\nu} : \fX' \to \fX$ be a formal modification.
\begin{enumerate}[label = \roman{enumi})]
\item If $\fX'$ is isomorphic to the formal completion of a complex space $X'$ along a complex subspace $Y' \subset X'$, then there exist a complex space $X$ and a modification $\nu : X' \to X$ along $Y \subset X$ such that $Y' = \nu^{-1}(Y)$ and  $\hat{\nu}$ is isomorphic to the formal completion $ \hat{Y'} \to \hat{Y}$ of $\nu$ along $Y' \to Y$. Moreover, if $Y' \subset X'$ is given, then the modification $\nu$ is unique (up to isomorphisms).
\item If $\fX$ is isomorphic to the formal completion of a complex space $X$ along a complex subspace $Y \subset X$, then there exist a complex space $X'$ and a modification $\nu : X' \to X$ along $Y \subset X$ such that if $Y' = \nu^{-1}(Y)$, then  $\hat{\nu}$ is isomorphic to the formal completion $ \hat{Y'} \to \hat{Y}$ of $\nu$ along $Y' \to Y$. Moreover, if $Y \subset X$ is given, then the modification $\nu$ is unique (up to isomorphisms).
\end{enumerate}
\end{thm}

\ssec{$G$-equivariant cohomology}\hfill

Let $X$ be a paracompact topological space and $\cF$ a sheaf of abelian groups on $X$. Let $G$ be a finite group acting on $X$ and on $\cF$ in a compatible way. The functor $\cF \mapsto \Gamma(X,\cF)^G$ is left exact and  the associated derived functor $\cF \mapsto H^k_G(X,\cF)$ is called the \emph{$G$-equivariant cohomology} of $\cF$. 

Following~\cite[Section 13]{KodairaSurfaceII}, we recall how elements of $H^k_G(X,\cF)$ are represented by $k$-cocycles and spell out the cocycle condition for $k=1$. The reader is referred to \emph{loc. cit.} for the details. Let $\{U_i\}_{i \in I}$ be a $G$-invariant good cover such that either $g(U_i) \cap U_i = \emptyset$ or $g(U_i) = U_i$. Define the $G$ action on $I$ by $g^{-1}(U_i) = U_{gi}$. Let $A^{p,q}_G(\cF)$ be the abelian group of collections of sections
$$\left\{s^{g_1,\ldots,g_p}_{i_1\cdots i_q} \in \cF(U_{i_1\cdots i_q}) \right\}_{g_1,\ldots,g_p \in G; i_1,\ldots, i_q \in I}$$
and let $$A^k_G(\cF) \cnec \bigoplus_{p+q = k} A^{p,q}_G(\cF).$$
Define 
$$D^k \cnec d^k + (-1)^q\gd^k : A^k_G(\cF) \to A^{k+1}_G(\cF)$$
where $d^k$ is the sum (over $p+q = k$) of the the coboundary maps $A^{p,q}_G(\cF) \to A^{p,q+1}_G(\cF)$ of $q$-cochains computing the \v{C}ech cohomology and $\gd^k$  the sum (over $p+q = k$) of the the coboundary maps  $A^{p,q}_G(\cF) \to A^{p+1,q}_G(\cF)$ of $p$-cochains computing the group cohomology of $G$-modules. Elements of $\ker(D^k)$ (resp. $\Ima(D^{k-1})$) are called $k$-cocycles (resp. $k$-coboundaries) and we have
$$H^k_G(X,\cF) = \ker(D^k)/\Ima(D^{k-1}).$$

In this text, we will only be interested in the first cohomology group $H^1_G(X,\cF)$. Each element of $H^1_G(X,\cF)$ is represented by a 1-cocycle $\left\{(\eta_{ij})_{i,j \in I},(\eta_i^g)_{i \in I, g \in G}\right\}$ and the concrete meaning of the cocycle condition is the following: for every $i,j,k \in I$ and $g,h \in G$, 
$$
\begin{array}{ll}
& \eta_{ij} + \eta_{jk} + \eta_{ki} = 0, \\
& g \cdot \eta_{(gi)(gj)} - \eta_{ij} = \eta_j^g - \eta_i^g, \\
& g \cdot \eta_{gi}^h -\eta_i^{gh} + \eta_i^g.
\end{array}
$$

\ssec{Campana's criterion}\hfill

Let $X$ be a complex variety. We say that $X$ is \emph{algebraically connected} if a general pair of points $x,y \in X$ is contained in a connected (but not necessarily irreducible) and complete curve in $X$. For a compact complex variety $X$ in the Fujiki class $\cC$ (i.e. bimeromorphic to a compact K\"ahler manifold), Campana shows that $X$ is Moishezon if and only if $X$ is algebraically connected~\cite[Corollaire, p.212]{CampanaCored}. Since we will mainly deal with fibrations by curves over a projective base, here we state a variant of Campana's criterion in this situation.

\begin{thm}[Campana]\label{thm-multsecMoi}
Let $f : X \to B$ be a fibration whose general fiber is a curve. Assume that $X$ is in the Fujiki class $\cC$ and $B$ is projective, then $X$ is Moishezon if and only if $f$ has a multi-section.
\end{thm}

\section{Locally Weierstra{\ss} fibrations and tautological families}\label{sec-locW}

The theory of elliptic fibrations had first been developed by Kodaira for fibrations over a curve~\cite{KodairaSurfaceII}, then by N. Nakayama in any dimension~\cite{NakayamaWeier, NakayamaLoc, NakayamaGlob}. We will start with a concise summary of the part of the theory of elliptic fibrations developed in~\cite{NakayamaWeier, NakayamaLoc, NakayamaGlob, ClaudonHorpi1} that we will use in this text. Then we will construct tautological models and the associated tautological families in~\ref{ssec-tautmod} and~\ref{ssec-famtaut}.

\ssec{Weierstra{\ss} fibrations}\label{ssec-fibWtaut}\hfill

\sssec{}

In the following we recall the definition of Weierstra{\ss} fibrations. These fibrations can be considered as standard compactifications of Jacobian fibrations associated to smooth elliptic fibrations: we will see that for every elliptic fibration $f : X \to B$ (over a complex manifold $B$), there exists a unique minimal Weierstra{\ss} fibration $p : W \to B$ whose restriction to $p^{-1}(B^\st)$ is isomorphic to the Jacobian fibration  $J \to B^\st$ associated to the smooth part $X^\st \to B^\st$ of $f$.

\begin{Def}\label{def-Wei}
Let $\cL$ be a line bundle over a complex space $B$ and let 
$$Z \in H^0\(\bP, \cO_\bP(1)\), \ X \in H^0\(\bP, \cO_\bP(1) \otimes p^*\cL^{(-2)}\), \ \text{and }  Y \in H^0\(\bP, \cO_\bP(1) \otimes p^*\cL^{(-3)}\)$$
be the three coordinate sections of the projectivization $\bP \cnec \bP(\cO \oplus \cL^2 \oplus \cL^3)$.
\begin{enumerate}[label = \roman{enumi})]
\item A \emph{Weierstra{\ss} fibration} is a fibration $p : W \cnec W(\cL,\ga,\gb) \to B$ defined by the projection onto $B$ of the hypersurface in $\bP$ defined by a nonzero section of the form 
$$Y^2Z - X^3 - \ga XZ^2 - \gb Z^3 \in H^0\(\bP, \cO_\bP(3) \otimes p^*\cL^{(-6)}\)$$
for some sections $\ga \in H^0(B,\cL^{(-4)})$ and $\gb \in H^0(B,\cL^{(-6)})$ such that a general fiber of $p$ is smooth (or equivalently, $4 \ga^3 + 27 \gb^2 \ne 0$ in $H^0(B,\cL^{(-12)})$). 
\item Without assuming that a general fiber of $p: W \to B$ is smooth, $p$ is called a \emph{singular} Weierstra{\ss} fibration.
\item \label{def-Wei3} The section defined by $\gS \cnec \{X = Z = 0\} \subset W$ is called the \emph{zero-section} of $p$.
\item\label{def-Wei-minglob} A Weierstra{\ss} fibration defined above is called \emph{minimal} if there is no prime divisor $D \subset B$ such that $\Div(\ga) \ge 4D$ and $\Div(\gb) \ge 6D$. 
\item\label{def-Wei-min} A (singular) (minimal) \emph{locally Weierstra{\ss} fibration} is a fibration $f : X \to B$ such that there exists an open cover $\{U_i\}$ of $B$ such that each $f^{-1}(U_i) \to U_i$ is a (singular) (minimal) Weierstra{\ss} fibration.
\end{enumerate}
\end{Def}

Obviously, any base change $W \times_B Z \to Z$ of a (singular) Weierstra{\ss} fibration $W \to B$ is still a (singular) Weierstra{\ss} fibration. Locally a (singular) Weierstra{\ss} fibration $p:W \to B$ is a pullback of the standard Weierstra{\ss} fibration
$$\left\{Y^2Z =  X^3 + a XZ^2 + b Z^3  \right\} \subset \bP^2 \times \bC^2$$
parameterized by $(a,b) \in \bC^2$. So $p$ is \emph{flat} with \emph{irreducible and reduced fibers} which are either elliptic curves, nodal cubic curves, or cuspidal cubic curves. 

When $B$ is normal, the total space $W$ of a Weierstra{\ss} fibration  is also normal~\cite[1.2.(1)]{NakayamaWeier}. The notion of \emph{minimal} Weierstra{\ss} fibrations first appears in~\cite{NakayamaWeier} and the main interests of this notion reside in the properties that, under the assumption that $B$ is normal and the discriminant locus $\gD = \Div(4\ga^3 + 27 \gb^2)$ is a normal crossing divisor, a Weierstra{\ss} fibration $p : W \to B$ is minimal if and only if $W$ has at worst rational singularities~\cite[Corollary 2.4]{NakayamaWeier}; also, when $B$ is smooth,  each bimeromorphic class of  Weierstra{\ss} fibrations contains a unique minimal Weierstra{\ss} fibration~\cite[Theorem 2.5]{NakayamaWeier}.

\sssec{}

Let $p:W\to B$ be a (singular) Weierstra{\ss} fibration and let $W^\# \subset W$ denote the smooth locus of $p$. The fibration $W^\# \to B$ can be considered as an analytic group variety over $B$ with connected fibers whose zero-section is  $\gS$ (see Definition~\ref{def-Wei}\comm{.}\ref{def-Wei3}). The relative tangent space at $\gS$ is isomorphic to $\cL$~\cite[Lemma 5.1.1.(8)]{NakayamaGlob}, so 
the exponential map induces a surjective morphism
\begin{equation}\label{expoW}
\begin{tikzcd}[cramped, row sep = 0, column sep = 20]
\exp : \cL \arrow[r, two heads] & \cJ^W 
\end{tikzcd}
\end{equation}
where $\cJ^W$ is the sheaf of germs of holomorphic sections of $p$ whose image is \emph{contained in  $W^\#$}.  We also have 
\begin{equation}\label{iso-LRp1OW}
\cL \simeq R^1p_*\cO_W
\end{equation}
\cite[1.2.(6)]{NakayamaWeier} since fibers of $p$ are curves of arithmetic genus 1.

Let $B$ be a complex manifold and $p^\st : J \to B^\star$ a Jacobian fibration over a Zariski open $B^\star \subset B$. By~\cite[Theorem 2.5]{NakayamaWeier}, there exists a unique minimal  Weierstra{\ss} fibration $p: W = W(\cL,\ga,\gb) \to B$ such that $p^{-1}(B^\st)$ is isomorphic to $J$ over $B^\st$. Since the weight one variation of Hodge structures $\bH \cnec R^1p^\st_*\bZ$ of rank 2 determines $J \to B^\star$, we also say that $p : W \to B$ is the minimal Weierstra{\ss} fibration \emph{associated to $\bH$}. If $f : X \to B$ is an elliptic fibration, then the minimal Weierstra{\ss} fibration \emph{associated to $f$} is defined to be the minimal Weierstra{\ss} fibration associated to $(R^1f_*\bZ)_{|B^\st}$ where $B^\st \subset B$ is a Zariski open over which $f$ is smooth.  If $G$ is a group acting on $B$ and on $\bH$ in a compatible way, then by~\cite[Corollary 2.6]{NakayamaWeier}, the $G$-action extends to a $G$-action on $W$ such that $p$ is $G$-equivariant.

In the case where $p : W \to B$ is the minimal Weierstra{\ss} fibration associated to $\bH$, we use the notations
$$\cL_{\bH/B} \cnec \cL \ \ \ \ \ \text{         and        } \ \ \ \  \cJ^W_{\bH/B} \cnec \cJ^W.$$ 
By~\cite[Proposition 2.10]{NakayamaWeier}, the exponential map lies in the short exact sequence
\begin{equation}\label{exseq-fibjac}
\begin{tikzcd}[cramped, row sep = 0, column sep = 40]
0  \arrow[r] & j_* \bH \arrow[r, "\varphi"] &\cL_{\bH/B} \arrow[r,"\exp"] & \cJ^W_{\bH/B} \arrow[r] & 0. 
\end{tikzcd}
\end{equation}
where $j : B^\st \hto B$ is the inclusion. The restriction of~\eqref{exseq-fibjac} to $B^\st$ is the short exact sequence
\begin{equation}\label{suitex-jaclisse}
\begin{tikzcd}[cramped, row sep = 0, column sep = 40]
0  \arrow[r] &  \bH \arrow[r] & \cE \cnec \cH/F^1\cH \arrow[r] & \cJ \arrow[r] & 0. 
\end{tikzcd}
\end{equation}
where $F^1\cH$ is the first piece of the Hodge filtration of $\cH \cnec \bH \otimes \cO_{B^\st}$ and $\cJ$ is the sheaf of germs of holomorphic sections of $J \to B^\star$.

\sssec{}\label{sssec-Gequiv}
The action of the group variety $W^\#$ over $B$ on $W^\#$ extends to a $W^\#$-action on $W$~\cite[Lemma 5.1.1.(7)]{NakayamaGlob}, so each \v{C}ech 1-cocycle $\eta$ of the sheaf $\cJ^W$ defines a (singular) locally Weierstra{\ss} fibration $p^\eta : W^\eta \to B$ (\emph{cf.}~\cite[Construction 3.14]{ClaudonHorpi1}). If $\eta'$ is another 1-cocycle representing the same class in $H^1(B, \cJ^W)$, then $W^{\eta'} \simeq W^\eta$ over $B$. We say that $p^\eta : W^\eta \to B$ is the locally Weierstra{\ss} fibration \emph{associated to $\eta$} (or \emph{twisted} by $\eta$).

More generally, let $G$ be a finite group and $p : W \to B$ a $G$-equivariant Weierstra{\ss} fibration such that the zero-section $\gS \subset W$ is $G$-stable. The sheaf $\cJ^W$ is endowed with a natural $G$-action and to each element  $\eta_G \in H^1_G(B,\cJ^W)$, we can associate a $G$-equivariant locally Weierstra{\ss} fibration $p^\eta : W^\eta \to B$~\cite[Section 3.E]{ClaudonHorpi1}. For later use, we shall briefly recall the construction. 

Given an element $\eta_G \in H_G^1(B,\cJ^W)$. Let $\{U_i\}_{i \in I}$ be a $G$-invariant good open cover of $B$ and let $G$ act on $I$ such that $g^{-1}(U_i) = U_{gi}$. The element $\eta_G$ can be represented by a 1-cocycle $\left\{(\eta_{ij})_{i,j \in I},(\eta_i^g)_{i \in I, g \in G}\right\}$ where $\{\eta_{ij}\}$ is a 1-cocycle representing the image $\eta$ of $\eta_G$ in $H^1(B,\cJ^W)$ and $\eta_i^g$ is a local section of $\cJ^W$ defined over $U_i$. Let $p^{\eta} : W^\eta \to B$ be the locally Weierstra{\ss} fibration twisted by $\eta$. Fix biholomorphic maps 
$$\eta_i : W^\eta_i \cnec (p^{\eta})^{-1}(U_i) \to W_i \cnec p^{-1}(U_i)$$ 
such that $\eta_i \circ \eta_j^{-1} = \tr(\eta_{ij})$ where $\tr(\eta_{ij})$ denotes the translation by the holomorphic section $\eta_{ij}$. For each $g \in G$, the automorphism $\psi_g : W^\eta \to W^\eta$  that defines the $G$-action on $p^\eta : W^\eta \to B$ associated to $\eta_G$ is constructed by patching together the isomorphisms 
$$\psi_g^i \cnec \eta_i^{-1} \circ \tr(\eta_i^g) \circ g \circ  \eta_{gi} : W^\eta_{gi} \to W^\eta_i.$$

Conversely, given a $G$-equivariant minimal locally Weierstra{\ss} fibration $p^{\eta} : W^\eta \to B$ twisted by $\eta \in H^1(B,\cJ^W)$, at least when $B$ is smooth we can also reconstruct the element $\eta_G \in H_G^1(B,\cJ^W)$ starting with which $p^{\eta}$ is constructed. First of all, the $G$-action on $p^\eta$ induces a $G$-action on $\bH \cnec (R^1p^\eta_*\bZ)_{|B^\st}$ where $B^\st \subset B$ is a Zariski open over which $p^\eta$ is smooth. By~\cite[Corollary 2.6]{NakayamaWeier}, the $G$-action on $\bH$ extends to a $G$-action on $W$ such that $p$ is $G$-equivariant, and it is for this $G$-action we define the $G$-equivariant cohomology group $H_G^1(B,\cJ^W)$. Now let $\{U_i\}$ be a good open cover of $B$ and let $\eta_i : W^\eta_i \to W_i$ be biholomorphic maps such that $\eta_i \circ \eta_j^{-1} = \tr(\eta_{ij})$ for some 1-cocycle $\{\eta_{ij}\}$ representing $\eta \in H^1(B,\cJ^W)$. Let $\psi_g : W^\eta \to W^\eta$ be the action of $g \in G$ on $W^\eta$. We define 
$$\tr(\eta_i^g) \cnec \eta_i \circ \psi_g \circ  \eta_{gi}^{-1} \circ g^{-1}.$$
The element $\eta_G \in H_G^1(B,\cJ^W)$ represented by the 1-cocycle $\left\{(\eta_{ij}),(\eta_i^g)\right\}$ is the element we look for.

\sssec{}

Let $p : W \to B$ be a Weierstra{\ss} fibration over a compact complex manifold $B$. For each $\eta \in H^1(B,\cJ^W)$ there exists a family of elliptic fibrations 
$$\Pi : \cW \xto{q} B \times V \to V \cnec H^1(B,\cL)$$
such that the elliptic fibration parameterized by $t \in V$ corresponds to $\eta + \exp(t) \in H^1(B,\cJ^W)$ where 
$$\exp : H^1(B,\cL) \to H^1(B,\cJ^W)$$ is the map induced by $\exp : \cL \to \cJ^W$. Indeed, let $\pr_1 : B \times V \to B$ be the first projection and let 
$$\xi \in V \otimes V^\vee \subset  V \otimes H^0(V,\cO_V) \simeq H^1(B \times V,\pr_1^*\cL)$$ 
be the element corresponding to the identity map $\Id_V$. Let 
$$\gl_\eta \cnec \pr_1^*\eta + \wt{\exp}(\xi) \in  H^1(B \times V,\cJ^{W \times V})$$
where 
$$\pr_1^* : H^1(B,\cJ^W) \to H^1(B \times V,\cJ^{W \times V})$$ 
is the map induced by pulling back sections $\cJ^W \to {\pr_1}_*\cJ^{W \times V}$ and 
$$\wt{\exp} : H^1(B \times V,\pr_1^*\cL) \to H^1(B \times V,\cJ^{W \times V})$$
the map induced by $\exp : \pr_1^*\cL \to \cJ^{W \times V}$.
Then the locally Weierstra{\ss} fibration $q : \cW \to B \times V$ twisted by $\gl_\eta$ will define such a family $\Pi$. The family $\Pi$ is called the \emph{tautological family associated to $\eta$}. Such a family can also be constructed when $p : W \to B$ is a singular Weierstra{\ss} fibration.

If, in addition, the fibration $p$ is $G$-equivariant for some finite group $G$ such that the zero section $\gS \subset W$ is preserved under the $G$-action and assume that $\eta$ is the image of a class $\eta_G \in H^1_G(B,\cJ^W)$, then $\cW^G \cnec \Pi^{-1}(V^G)$ can be endowed with a $G$-action such that the restriction $q_{|\cW^G}$  is the $G$-equivariant locally Weierstra{\ss} fibration twisted by
$$\gl_{\eta_G} \cnec \pr_1^*\eta_G  + \wt{\exp}_G(\xi_G) \in H^1_G(B \times V^G,\cJ^{W^G \times V^G})$$
where $\xi_G \in V^G \otimes  H^0(V^G,\cO_{V^G}) \simeq H^1_G(B \times V,\pr_1^*\cL)$ is the element corresponding to the identity map $\Id_{V^G}$ and 
$$\wt{\exp}_G :H^1_G(B \times V,\pr_1^*\cL) \to H^1_G(B \times V,\cJ^{W \times V})$$ 
is again the map induced by $\exp : \pr_1^*\cL \to \cJ^{W \times V}$. So each point $t \in V^G$ parameterizes the $G$-equivariant locally Weierstra{\ss} fibration twisted by $\eta_G + \exp_G(t) \in H^1_G(B,\cJ^W)$ in the tautological family $\Pi$. 

For every $G$-stable analytic subset $Z \subset B$, the next lemma gives a standard way to produce a subspace $V^G_Z$ of $V$ along which the deformation of $p^{\eta}$ in the tautological family preserves $G$-equivariantly the fibration $(p^{\eta})^{-1}(Z) \to Z$ contained in $p^\eta$.

\begin{lem}\label{lem-presWtauto}
Let $Z \subset B$ be a $G$-stable subvariety and let
$$V^G_Z = \ker\(\imath^* : H^1(B,\cL) \to H^1(Z,\imath^*\cL)\)^G$$ 
where $\imath : Z \hto B$ is the inclusion. Then the subfamily of the tautological family associated to $ \eta_G \in H^1_G(B,\cJ^W)$ parameterized by $V^G_Z$ preserves the $G$-action and $G$-equivariantly the fibration $W_Z^{\eta} \cnec (p^{\eta})^{-1}(Z) \to Z$. 
\end{lem}

\begin{proof}
We already saw that the subfamily parameterized by $V^G_Z$ preserves the $G$-action. Let $\Psi : Z \times V^G_Z \hto B \times V^G$ be the product of $\imath$ with $V^G_Z \hto V$. The restriction 
$$\cY \cnec q^{-1}(Z \times V^G_Z) \to Z \times V^G_Z$$ 
of $q$ to $\cY$ is the $G$-equivariant locally Weierstra{\ss} elliptic fibration twisted by $\Psi^*\gl_{\eta_G} \in H^1_G\(Z \times V^G_Z,\cJ^{W_Z \times V^G_Z}\)$ where $W_Z \cnec p^{-1}(Z)$ and 
$$\Psi^* : H^1_G\(B \times V^G,\cJ^{W \times V^G}\) \to  H^1_G\(Z \times V^G_Z,\cJ^{W_Z \times V^G_Z}\)$$ 
is the map induced by pulling back sections $\cJ^{W \times V^G} \to \Psi_*\cJ^{W_Z \times V^G_Z}$. By definition of $\xi_G$ and $V^G_Z$, we have 
$$\Psi^*\xi_G = 0 \in H^1(Z,\imath^*\cL)^G \otimes H^0(V^G_Z,\cO_{V^G_Z}).$$ 
So $\Psi^*\gl_{\eta_G} =\pr_1^*\imath^*{\eta_G}$ where $\pr_1 : Z \times V^G_Z \to Z$ is the first projection and  
$$\pr_1^* :  H^1_G(Z,\cJ^{W_Z}) \to H^1_G\(Z \times V^G_Z,\cJ^{W_Z \times V^G_Z}\)$$ 
is the map induced by $\cJ^{W_Z} \to {\pr_1}_*\cJ^{W_Z \times V^G_Z}$. It follows that  $\cY \to Z \times V^G_Z$ is $G$-equivariantly isomorphic to $W_Z^{\eta} \times V^G_Z \to Z \times V^G_Z$, which proves the first statement of Lemma~\ref{lem-presWtauto}.
\end{proof}

Finally, when $p : W \to B$ is the minimal Weierstra{\ss} fibration associated to $\bH$, it follows from~\eqref{exseq-fibjac} that $t$ and $t' \in H^1(B,\cL_{\bH/B})$ parameterize isomorphic elliptic fibrations (over $B$) in $\Pi$ if and only if $t - t'$ lies in the image of $ H^1(B,j_*\bH)$.

\sssec{}
Let $p : W \to B$ be a Weierstra{\ss} fibration such that $W$ is normal. As the codimension of $W \bss W^\#$ in $W$ is at least 2, for every $m \in \bZ$, the multiplication-by-$m$ map $W^\# \to W^\#$ extends to a meromorphic map $\bm : W \dto W$ (which can be non-holomorphic, see Remark~\ref{rem-strict}). Gluing these maps locally, we obtain for each $\eta \in H^1(B,\cJ^W)$ a map $\bm : W^\eta \dto W^{m\eta}$, which is generically finite if $m \ne 0$. An immediate consequence of the existence of $\bm$ is the following cohomological criterion for the existence of a multi-section of $p^\eta : W^\eta \to B$.

\begin{lem}\label{lem-Tormultsec}
Assume that $W$ is normal. If $\eta \in H^1(B,\cJ^W)$ is torsion, then $p^\eta : W^\eta \to B$ has a multi-section. 
\end{lem}
\begin{proof}
Assume that $m$ is the order of $\eta$, then $\bm : W^\eta \dto W$ defines a generically finite map onto $W$. Thus $\bm^{-1}(\gS)$ is a multi-section of $W^\eta$ where we recall that $\gS \subset W$ is the zero-section of $p : W \to B$.
\end{proof}

The reader is referred to Lemma~\ref{lem-multisecTor} for a converse of Lemma~\ref{lem-Tormultsec}.

\ssec{Locally Weierstra{\ss} fibrations over a smooth variety with a normal crossing discriminant divisor}\label{ssec-locWNC}
\hfill

\sssec{}
Let $p : W \to B$ be a minimal Weierstra{\ss} fibration over a complex manifold and let $\gD \subset B$ be its discriminant locus. In this paragraph, we assume that $\gD$ is a normal crossing divisor. Under this assumption, $W$, and more generally the total space of the fibration $W^\eta \to B$ twisted by $\eta \in H^1(B,\cJ_{\bH/B}^W)$ have at worst rational singularities~\cite[Corollary 2.4]{NakayamaWeier}. In this situation, the map $\varphi$ in~\eqref{exseq-fibjac} can be described as follows.

\begin{lem}[{\cite[Lemma 3.15]{ClaudonHorpi1}}]\label{lem-isophi}
Let $p : W \to B$ be the minimal Weierstra{\ss} fibration associated to a weight one variation of Hodge structures $\bH$ of rank 2. Assume that $B$ is smooth and $\gD $ is a normal crossing divisor. Then $\varphi : j_* \bH \to \cL_{\bH/B}$ is isomorphic to $R^1p_*\bZ \to R^1p_*\cO_W$ induced by  $\bZ \hto \cO_W$. 
 \end{lem}
 
 If the local monodromies of $\bH$ around $\gD$ are unipotent, then the construction of the minimal Weierstra{\ss} fibration associated to $\bH$ is functorial under pullback:
 
 \begin{lem}[{\cite[Remark on p.549]{NakayamaGlob}}]\label{lem-functW} 
 Let $p : W \cnec W(\cL_{\bH/B} , \ga, \gb) \to B$ be the minimal Weierstra{\ss} fibration associated to $\bH$. Assume that $B$ is smooth and the discriminant locus $\gD$ of $p$ is a normal crossing divisor. Let $\psi : Z \to B$ be a holomorphic map from a complex manifold such that $\gD' \cnec \psi^{-1}(\gD)$ is also a normal crossing divisor. Let $Z^\st = Z \bss \gD'$ and $\bH' = \psi_{|Z^\st}^{-1}\bH$. If  the local monodromies of $\bH$ around $\gD$ are unipotent, then $\cL_{\bH'/Z} \simeq \psi^*\cL_{\bH/B}$  and the base change $W_Z \cnec W \times_B Z \to Z$ is the minimal Weierstra{\ss} fibration associated to $\bH'$.
 \end{lem}

  \begin{proof}
This proof is pointed out to us by N. Nakayama. For simplicity, let $\cL = \cL_{\bH/B}$ and $\cL' = \cL_{\bH' / Z}$. As the local monodromies of $\bH$ around $\gD$ are unipotent and $W$ has at worst rational singularities, by~\cite[Theorem 2.6]{KollarHDIII} the sheaf $R^1p_*\cO_W$ is isomorphic to $\Gr_F^0 {\bar{\cH}}$, the zeroth graded piece of the Hodge filtration on the canonical extension $\bar{\cH}$ of the VHS that $\bH$ underlies. So $\cL \simeq \Gr_F^0 {\bar{\cH}}$ by~\eqref{iso-LRp1OW}. Since the canonical extension is functorial, if $\bar{\cH}'$ denotes the canonical extension of the VHS that $\bH'$ underlies, then 
$$\cL'   \simeq \Gr_F^0 {\bar{\cH}'} \simeq \psi^*\Gr_F^0 {\bar{\cH}} \simeq \psi^*\cL.$$
 
 Let $p' : W' \cnec W'(\cL' , \ga', \gb') \to Z$ be the minimal Weierstra{\ss} fibration associated to $\bH'$. 
 Since $W_Z \to Z$ and $W' \to Z$ are both Weierstra{\ss} fibrations extending the Jacobian fibration associated to $\bH'$, by~\cite[Lemma 1.4]{NakayamaWeier} there exists $\gep \in H^0(Z^\st, \psi^*\cL \otimes (\cL')^\vee ) \simeq H^0(Z^\st,\cO_{Z^\st}) $ such that $(\psi^*\ga)_{|Z^\st} = \gep^4 \ga'_{|Z^\st}$ and $(\psi^*\gb)_{|Z^\st} = \gep^6 \gb'_{|Z^\st}$. As $\gep^4$ is the restriction to $Z^\st$ of the meromorphic function $(\psi^*\ga)/{\ga'}$, $\gep$ extends to a meromorphic function on $Z$ (still denoted by $\gep$). So $\psi^*\ga = \gep^4 \ga'$ and $\psi^*\gb = \gep^6 \gb'$. Since there is no prime divisor $D$ of $Z$ such that $\Div(\ga') \ge 4D$ (resp. $\Div(\ga) \ge 4D$),  $\gep$ has no pole (resp. no zero). Therefore $W_Z \simeq W'$ over $Z$, which proves Lemma~\ref{lem-functW}.
 \end{proof}

\sssec{}
Let $f : X \cnec W^\eta \to B$ be the locally Weierstra{\ss} fibration twisted by $\eta \in H^1(B, \cJ^W_{\bH/B})$. We still assume that $B$ is smooth and the discriminant locus $\gD$ is a normal crossing divisor. In addition to~\eqref{exseq-fibjac}, the sheaf $\cJ^W_{\bH/B}$ sits inside another short exact sequence
\begin{equation}\label{exseq-JWZ}
\begin{tikzcd}[cramped, row sep = 0, column sep = 40]
0  \arrow[r] & \cJ^W_{\bH/B} \arrow[r] & R^1f_*\cO^\times_X \arrow[r] & \bZ \arrow[r] & 0,
\end{tikzcd}
\end{equation}
where $R^1f_*\cO^\times_X \to \bZ$ maps locally a line bundle to the degree of its restriction to a general fiber of $f$. The class $\eta \in H^1(B, \cJ^W_{\bH/B})$ coincides with the one associated to the extension~\eqref{exseq-JWZ}~\cite[short exact sequence (9)]{ClaudonHorpi1}.

Now let $f:X \to B$ be an elliptic fibration such that both $X$ and $B$ are smooth, the discriminant locus $\gD \subset B$ is a normal crossing divisor, and $f$ has local meromorphic sections at every point of $B$. For such an elliptic fibration, there exists a short exact sequence similar to~\eqref{exseq-JWZ}:
Let  $\bH \cnec (R^1f_*\bZ)_{|B \bss \gD} $ and let $p : W \to B$ be the minimal Weierstra{\ss} fibration associated to $\bH$. Regarding the zero-section $\gS \subset W$ as the neutral element, let $\cJ_{\bH/B}$ be the sheaf of abelian groups of germs of meromorphic sections of $p$. By~\cite[Lemma 5.4.8]{NakayamaGlob}, there exists a short exact sequence
\begin{equation}\label{suitex-JHB}
\begin{tikzcd}[cramped, row sep = 0, column sep = 40]
0  \arrow[r] & \cJ_{\bH/B} \arrow[r] & R^1f_*\cO^\times_X /\cV_X  \arrow[r] & \bZ \arrow[r] & 0 
\end{tikzcd}
\end{equation}
where $\cV_X \cnec \ker\(R^1f_*\cO_X^\times \to j_*j^*R^1f_*\cO_X^\times\)$ and the third arrow comes from the map $R^1f_*\cO^\times_X \to \bZ$ defined by the degree of a line bundle restricted to a general fiber as in~\eqref{exseq-JWZ}. 

By~\cite[Proposition 5.5.1]{NakayamaGlob}, there exists an inclusion 
$$\eta : \cE(B,\gD,\bH) \hto H^1(B,\cJ_{\bH/B})$$ 
where $\cE(B,\gD,\bH)$ is the set of bimeromorphic classes of elliptic fibrations $f: X \to B$ admitting local meromorphic sections at every point of $B$ and whose restriction over $B \bss \gD$ is bimeromorphic to a smooth fibration with $(R^1f_*\bZ)_{|B \bss \gD} \simeq \bH$. The image $\eta(f) \in H^1(B,\cJ_{\bH/B})$ of  $f$  coincides with the element associated to the extension~\eqref{suitex-JHB}.  According to the above, there exists a map $H^1(B, \cJ^W_{\bH/B}) \to \cE(B,\gD,\bH)$ which associates $\eta \in H^1(B, \cJ^W_{\bH/B})$ to the bimeromorphic class of $p^\eta : W^\eta \to B$ and the composition 
$$\imath^W : H^1(B, \cJ^W_{\bH/B}) \to \cE(B,\gD,\bH) \hto  H^1(B,\cJ_{\bH/B})$$ 
equals the map induced by the inclusion  $\cJ^W_{\bH/B} \subset \cJ_{\bH/B}$. 

\begin{rem}\label{rem-strict}
There exist examples of Weierstra{\ss} fibrations $W \to B$ (with smooth discriminant locus) constructed by N. Nakayama such that the inclusion $\cJ^W_{\bH/B} \subset \cJ_{\bH/B}$ is strict~\cite{NakayamaCex}. This also allows us to explain why the multiplication-by-$m$ $\bm : W \dto W$ introduced in \S 3.1.5 can be non-holomorphic. Let $p : W \to B$ be a Weierstra{\ss} fibration such that both $B$ and the discriminant locus $\gD$ of $p$ are smooth and $\cJ^W_{\bH/B} \subsetneq \cJ_{\bH/B}$. The quotient $\cJ_{\bH/B}/\cJ^W_{\bH/B}$ is torsion~\cite[Corollary 5.4.11]{NakayamaGlob}. Therefore if $Y \subset W$ is a local meromorphic section, \ie the closure of a local section $Y^\circ \subset W^\circ = p^{-1}(B \bss \gD)$ of the  smooth fibration $p^{-1}(B\bss \gD) \to B \bss \gD$, then for some integer $m > 0$, the closure of $\bm(Y^\circ)$ in $W$ is a holomorphic section of $p$ contained in the smooth locus $W^\#$ of $p$. Since $\cJ^W_{\bH/B} \subsetneq \cJ_{\bH/B}$, we can choose $Y$ such that $Y$ is not a holomorphic section contained in $W^\#$ (so that $m > 1$). Assume that $\bm : W \to W$ is holomorphic. As $Y$ is irreducible, $\bm(Y)$ coincides with the closure of $\bm(Y^\circ)$ in $W$. Since $\bm^{-1}(W^\#) = W^\#$, necessarily $Y \subset W^\#$, so $Y \to B$ is finite. As $B$ is smooth and $Y \to B$ is generically injective, $Y \to B$ is biholomorphic. It follows that $Y$ is already a holomorphic section contained in $W^\#$, which contradicts the hypothesis that $\bm$ is holomorphic.
\end{rem}

We can also generalize the above discussion to the $G$-equivariant setting~\cite[Section 3.E]{ClaudonHorpi1}. Let $G$ be a finite group acting on $B$ and on $\bH$ in a compatible way. Let $\cE_G(B,\gD,\bH)$ denote the set of bimeromorphic classes of $G$-equivariant elliptic fibrations $f \in \cE(B,\gD,\bH)$ such that $ (R^1f_*\bZ)_{|B \bss \gD}$ is $G$-equivariantly isomorphic to $\bH$.\footnote{We could have defined $\cE_G(B,\gD,\bH)$ to be a larger set by allowing the $G$-action on the total space of $f$ to be only \emph{meromorphic} (but still holomorphic over $B \bss \gD$). However in this text, we will only consider $G$-actions that are holomorphic.}  To each $G$-equivariant elliptic fibration $f \in \cE_G(B,\gD,\bH)$, we can associate an element $\eta_G(f) \in H_G^1(B,\cJ_{\bH/B})$ in an injective manner using a similar construction to the one sketched in~\ref{sssec-Gequiv}. According to the above, there exists a map $H_G^1(B, \cJ^W_{\bH/B}) \to \cE_G(B,\gD,\bH)$ that associates $\eta_G \in H_G^1(B, \cJ^W_{\bH/B})$ to the bimeromorphic class of the $G$-equivariant elliptic fibration $p^\eta : W^\eta \to B$ and the composition 
$$\imath_G^W : H_G^1(B, \cJ^W_{\bH/B}) \to \cE_G(B,\gD,\bH) \hto  H^1_G(B,\cJ_{\bH/B})$$ 
equals the map induced by the inclusion $\cJ^W_{\bH/B} \subset \cJ_{\bH/B}$.

\ssec{K\"ahler or projective elliptic fibrations}\label{ssec-hypkah}\hfill

Let $f:X \to B$ be an elliptic fibration over a complex manifold $B$ with a normal crossing discriminant divisor. The aim of this paragraph is to recall further properties of $f$ under the additional assumption that $X$ is in the Fujiki class $\cC$.

Given a minimal $\eta$-twisted locally Weierstra{\ss} fibration $p^{\eta} : W^\eta \to B$ over a compact K\"ahler manifold, we have the following cohomological characterization for the total space $W^{\eta}$ to be bimeromorphically K\"ahler.

\begin{pro}[{\cite[Theorem 3.20 and Proposition 3.23]{ClaudonHorpi1}}]\label{pro-condK}
Let $p:W\to B$ be a minimal Weierstra{\ss} fibration over a compact K\"ahler manifold and assume that the discriminant locus is a normal crossing divisor. Let $\eta \in H^1(B,\cJ^W_{\bH/B})$.  The following assertions are equivalent.
\begin{enumerate}[label = \roman{enumi})]
\item  The total space $W^{\eta}$ is in the Fujiki class $\cC$.
\item The "Chern class" $c(\eta)$ of $\eta$ is a torsion element, where $c : H^1(B, \cJ^W_{\bH/B}) \to H^2(B,j_*\bH)$ is defined to be the connecting morphism  induced by~\eqref{exseq-fibjac}.
\end{enumerate}
\end{pro}

When $W^\eta$ is in the Fujiki class $\cC$ and $B$ is projective, we can prove the following converse of Lemma~\ref{lem-Tormultsec}.
\begin{lem}\label{lem-multisecTor}
Let $p: W \to B$ be a minimal Weierstra{\ss} fibration and $f  : W^\eta \to B$ the locally Weierstra{\ss} fibration twisted by $\eta \in H^1(B, \cJ^W_{\bH/B})$. Assume that $W^\eta$ is in the Fujiki class $\cC$ and $B$ is projective. If $f$ has a multi-section, then $\eta$ is a torsion element.
\end{lem}

\begin{proof}
Since $W^\eta$ is in the Fujiki class $\cC$, by Proposition~\ref{pro-condK} $c(\eta)$ is torsion. So up to replacing $\eta$ with a larger multiple, we may assume that $\eta = \exp(t)$ for some $t \in H^1(B,\cL_{\bH/B})$.
Since the multiplication by $m$ is generically finite, the elliptic fibration $f: W^{\exp(t)} \to B$ still has a multi-section.
 For any $m \in \bZ_{>0}$, a multi-section of $f$ can be pulled back to a multi-section of $f' = f^{\exp(t/m)}: W^{\exp(t/m)} \to B$ under the generically finite map $\bm : W^{\exp(t/m)} \dto W^{\exp(t)}$. Assume that $m$ is sufficiently large so that $W^{\exp(t/m)}$ is closed enough to $W$ in the tautological family associated to $p : W \to B$.  As $B$ is projective, $W$ is also projective. So $W^{\exp(t/m)}$, being a small deformation of $W$ with at worst rational singularities (because $p : W\to B$ is minimal),  is K\"ahler~\cite[Proposition 5]{NamikawaExt2forms}.  Again as $B$ is projective,  $W^{\exp(t/m)}$ is Moishezon by Theorem~\ref{thm-multsecMoi}. It follows that $W^{\exp(t/m)}$ is projective~\cite[Theorem 1.6]{NAMIKAWA2002} and an ample line bundle on $W^{\exp(t/m)}$ gives rise to an element $s \in H^0(B, R^1f'_*\cO^\times_{W^{\exp(t/m)}})$ whose restriction to a general fiber of $f'$ has nonzero degree. So the image of $s$ in $H^0(B, \bZ)$ under the map induced by~\eqref{exseq-JWZ} is not zero. Therefore ${\exp(t/m)}$ is torsion, so $\eta$ is torsion as well.
\end{proof}

At the end of~\ref{ssec-locWNC}, we introduced the map $\imath^W : H^1(B, \cJ^W_{\bH/B}) \to H^1(B,\cJ_{\bH/B})$.
 This map is not surjective in general. However when $X$ is in the Fujiki class $\cC$, some multiple of the class $\eta(f) \in H^1(B,\cJ_{\bH/B})$ always comes from $H^1(B, \cJ^W_{\bH/B})$.

\begin{lem}[{\cite[Lemma 3.19]{ClaudonHorpi1}}]\label{lem-Ktors}
Let $f : X \to B$ be an elliptic fibration over a compact complex manifold $B$.  Assume that the discriminant locus is a normal crossing divisor and $f$ has local meromorphic sections over every point of $B$. Let  $W \to B$ be the minimal Weierstra{\ss} fibration associated to $f$. Assume that $X$ is  in the Fujiki class $\cC$, then there exists $m \in \bZ \bss \{0\}$ such that $$m \cdot \eta(f) \in \Ima\(\imath^W : H^1(B,\cJ_{\bH/B}^W) \to H^1(B,\cJ_{\bH/B})\).$$
\end{lem}

For $G$-equivariant elliptic fibrations $f : X \to B$, if the conclusion of Lemma~\ref{lem-Ktors} holds, then it also holds $G$-equivariantly:

\begin{lem}[{\cite[Lemma 3.24]{ClaudonHorpi1}}]\label{lem-torsimGtors}
Let $p : W \to B$ be a $G$-equivariant Weierstra{\ss} fibration for some finite group $G$ such that the zero-section $\gS \subset W$ is $G$-stable. Let $\eta_G \in H_G^1(B,\cJ_{\bH/B})$ and let $\eta$ denote its image in $H^1(B,\cJ_{\bH/B})$. Assume that there exist $m \in \bZ_{>0}$ and $\eta' \in H^1(B,\cJ_{\bH/B}^W)$ such that $m\eta =\imath^W(\eta') $, then up to replacing $m$ with a larger multiple of it,  $m \eta_G$ can be lifted to an element in $H^1_G(B,\cJ_{\bH/B}^W)$.
\end{lem}

\ssec{Tautological models}\label{ssec-tautmod} \hfill

Let $B$ be a complex manifold and $B^\st \subset B$ a Zariski open such that $B \bss B^\st$ is a normal crossing divisor. Let $\bH$ be a weight one variation of Hodge structures of rank 2 over $B^\st$ and let $G$ be a finite group acting on both $B$ and $\bH$ in a compatible way.  We denote by $p: W \to B$ be the minimal $G$-equivariant Weierstra{\ss} fibration associated to $\bH$. Given $\eta_G \in H^1_G(B,\cJ_{\bH/B})$ and assume that $m\eta_G$ can be lifted to an element $\eta_G' \in H^1_G(B,\cJ^W_{\bH/B})$ for some $m \in \bZ \bss \{0\}$. Let $\{U_i\}$ be a $G$-invariant good open cover of $B$ and $\left\{(\eta_{ij}),(\eta_i^g)\right\}$ a 1-cocycle representing $\eta_G \in H^1_G(B,\cJ_{\bH/B})$. In this paragraph, we will construct a $G$-equivariant elliptic fibration 
$$f : \sX \cnec \sX\(m; (\eta_{ij}),(\eta^g_{i}) ; \eta'_G\) \to B$$ representing $\eta_G$ together with a $G$-equivariant finite \emph{holomorphic} map
 $$\bm : \sX \to W^{\eta'}$$ 
 over $B$, where $W^{\eta'} \to B$ is the $G$-equivariant locally Weierstra{\ss} fibration associated to $\eta'_G \in H_G^1(B,\cJ^W_{\bH/B})$. The map $\bm$ generalizes the multiplication-by-$m$ map defined before.
  We will call this $G$-equivariant elliptic fibration $f$ a \emph{tautological model} associated to $\eta_G$. The reason we call such an elliptic fibration a tautological model is that in the next section, we will show that each tautological model $f:\sX \to B$ comes equipped with a \emph{tautological family} of elliptic fibrations parameterized by $H^1(B,\cL_{\bH/B})$ generalizing the one defined before for twisted locally Weierstra{\ss} fibrations.

 Let us 

An argument similar to that of~\cite[Proposition 5.5.4]{NakayamaGlob} will allow us to construct $f : \sX \to B$.  For each $g \in G$, let $\phi_g : W \to W$ be the biholomorphic map defining the $G$-action on $W$. Let $\left\{(\eta'_{ij}),(\eta'^g_{i})\right\}$ be a 1-cocycle representing $\eta'_G \in  H^1_G(B,\cJ^W_{\bH/B})$. Let $W_i = p^{-1}(U_i)$ and $W_{ij} = p^{-1}(U_{ij})$ where $U_{ij} = U_i \cap U_j$. As $m  \eta_G = \imath^W_G (\eta'_G)$, there exist meromorphic sections $\gs_i$ of $W_i \to U_i$ such that 
\begin{equation}\label{comm-mod1cob}
\begin{tikzcd}[cramped, row sep = 25, column sep = 35]
W_{ij} \arrow[r, dashed, "\tr(\eta_{ij})"] \ar[d, dashed, swap, "\bm_{ij}"] & W_{ji}  \ar[d, dashed, "\bm_{ji}"] \\
W_{ij} \arrow[r, dashed, " \tr(m\eta_{ij})"] \arrow[d, dashed, swap, " \tr(\gs_i)"]& W_{ji} \arrow[d, dashed, " \tr(\gs_j)"]   \\ 
W_{ij} \arrow[r, "\sim"]\ar[r, "\tr(\eta'_{ij})", swap] &  W_{ji}   
\end{tikzcd}
\ \ \ \quad \quad \text{ and } \ \ \ \quad \quad
\begin{tikzcd}[cramped, row sep = 25, column sep = 55]
 W_{gi} \arrow[r, dashed, "\psi^g_i \cnec \tr(\eta_i^g) \circ \phi_g"] \ar[d, dashed, swap, "\bm_{gi}"] & W_{i}  \ar[d, dashed, "\bm_{i}"] \\
 W_{gi}  \ar[r, dashed, "\tr(m\eta_i^g) \circ \phi_g"]  \arrow[d, dashed, swap, " \tr(\gs_{gi})"]& W_{i} \arrow[d, dashed, " \tr(\gs_{i})"]  \\ 
W_{gi} \arrow[r, "\sim"]\ar[r, "\tr(\eta_i'^g) \circ \phi_g", swap] &  W_{i}  
\end{tikzcd}
\end{equation} 
are commutative where the notation $\tr(\gs)$ denotes the translation by the meromorphic section $\gs$ and  $\bm_{i}$ (resp. $\bm_{ij}$) the restriction to $W_i$ (resp. $W_{ij}$) of the multiplication-by-$m$ map $W \dto W$. Let $\mu_i : \sX_i \to W_i$ be the finite map in any Stein factorization of $\tr( \gs_i )\circ \bm_i$ (see \S~\ref{sssec-desingSt}). Then there exist bimeromorphic maps $h_i : \sX_i \dto W_i$ over $U_i$ such that
\begin{equation}\label{comm-SteinWij}
\begin{tikzcd}[cramped, row sep = 25, column sep = 75]
\sX_{ij} \arrow[r, dashed, "h_{ij} \cnec h_j^{-1}\circ \tr(\eta_{ij}) \circ h_i"] \ar[d, swap, "\mu_{i}"]  & \sX_{ji}  \ar[d, "\mu_{j}"]  \\ 
 W_{ij} \ar[r, "\tr(\eta'_{ij})"]   &  W_{ji}    
\end{tikzcd}
\ \ \ \quad \quad \text{ and } \ \ \ \quad \quad
\begin{tikzcd}[cramped, row sep = 25, column sep = 75]
\sX_{gi} \arrow[r, dashed, "h_i^g \cnec h_i^{-1}\circ \psi_i^g \circ h_{gi}"] \ar[d, swap, "\mu_{gi}"]  & \sX_{i}  \ar[d, "\mu_{i}"] \\ 
W_{gi} \ar[r,"\tr(\eta_i'^g) \circ \phi_g"] &  W_{i}   
\end{tikzcd}
\end{equation}
are commutative where $\sX_{ij} \cnec {\mu_i}^{-1}(W_{ij})$ and $\sX_{ji} \cnec {\mu_j}^{-1}(W_{ji})$. As $\mu_i$ and $\mu_j$ are finite and the varieties in~\eqref{comm-SteinWij} are normal, the maps $h_{ij}$ and $h_i^g$ are biholomorphic. Thus we obtain an elliptic fibration $f: \sX \to B$ by gluing the $\sX_i \to U_i$ using the 1-cocycle $\{h_{ij}\}$ of biholomorphic maps. The maps $h_{i}^g$ can also be glued to a biholomorphic map $\psi_g : \sX \to \sX$ using the cocycle condition of $\left\{(\eta_{ij}),(\eta^g_{i})\right\}$ and $g \mapsto \psi_g$ defines a $G$-action on $\sX$ such that $f : \sX \to B$ is $G$-equivariant. By construction, $f$ is a $G$-equivariant elliptic fibration representing $\eta_G$. We can also glue the $\mu_i$ and obtain a $G$-equivariant finite map $\bm : \sX \to W^{\eta'}$, which generalizes the multiplication-by-$m$ maps defined previously. 
 
\begin{Def}
Let $\eta_G \in H^1_G(B,\cJ_{\bH/B})$ be an element such that $m\eta_G$ can be lifted to an element $\eta'_G \in H^1_G(B,\cJ^W_{\bH/B})$ for some integer $m \ne 0$. 
The $G$-equivariant elliptic fibration 
$$f:\sX = \sX\(m ; (\eta_{ij}),(\eta^g_{i}) ; \eta'_G\) \to B$$ 
representing $\eta_G$ constructed above is called a \emph{(G-equivariant) tautological model (associated to $\eta_G$)}.
\end{Def}
If the $G$-action is trivial (so that $H^1_G(B,\cJ_{\bH/B}) = H^1(B,\cJ_{\bH/B})$), we will use the notation $\sX\(m ; (\eta_{ij}) ; \eta'\)$ instead of $\sX\(m ; (\eta_{ij}),(0) ; \eta'\)$.

\begin{rem}\label{rem-modcobordW}
Up to isomorphism, the construction of $f:\sX  \to B$ and $\bm$ depends on $m$, $\eta'_G$, and the 1-cocycle $\left\{(\eta_{ij}),(\eta^g_{i})\right\}$ representing $\eta_G$. But it is easy to see from the construction that $f$ and $\bm$ do not depend on the 1-cocycle $\left\{ (\eta'_{ij}), (\eta_{i}'^g) \right\}$ representing $\eta'_G$. 
\end{rem}

\begin{rem}
Since the Stein factorization is functorial under flat pullbacks, the flat pullback of a tautological model is still a tautological model.
\end{rem}

The next lemma shows that a bimeromorphic class of elliptic fibrations $[\varphi] \in \cE_G(B,\gD,\bH)$ contains a tautological model if the total space of $\varphi$ is in the Fujiki class $\cC$. 
 
 \begin{lem-def}\label{lem-locWeimodHdg}
Let $\varphi : X \to B$ be a $G$-equivariant elliptic fibration over a compact complex manifold for some finite group $G$. Assume that $X$ is in the Fujiki class $\cC$, $\varphi$ is smooth over the complement of a normal crossing divisor in $B$, and $\varphi$ has local meromorphic sections over every point of $B$. Then the element $\eta_G(\varphi)  \in H_G^1(B,\cJ_{\bH/B})$ associated to $\varphi$ has a tautological model, called a \emph{($G$-equivariant) tautological model} of $\varphi$. 
\end{lem-def}

\begin{proof}
Let $\eta$ be the image of $\eta_G(\varphi)$ in $H^1(B,\cJ_{\bH/B})$. Since $X$ is in the Fujiki class $\cC$, by Lemma~\ref{lem-Ktors} some nonzero multiple of $\eta$ can be lifted to an element in $H^1(B,\cJ^{W}_{\bH/B})$. So by Lemma~\ref{lem-torsimGtors}, some nonzero multiple of $\eta_G$ can be lifted to an element $\eta'_G \in H_G^1(B,\cJ^{W}_{\bH/B})$. Thus $\eta_G(\varphi)$ has a tautological model.
\end{proof}

We finish this paragraph with some geometric properties of the tautological models.
 
\begin{lem}\label{lem-dicstaut} 
The total space of $f: \sX \to B$ is normal and the discriminant locus $\gD_f \subset B$ of $f$ coincides with the discriminant locus $\gD$ of the minimal Weierstra{\ss} fibration $p : W \to B$ associated to $\bH$. 
\end{lem}
\begin{proof}
As $\mu_i : \sX_i \to W_i$ is the finite map in a Stein factorization, each $\sX_i$ is normal. So $\sX$ is normal. 

Since the restriction of $p$ to $W \bss p^{-1}(\gD_f)$ is the minimal Weierstra{\ss} fibration associated to the restriction of $f$ to $\sX \bss f^{-1}(\gD_f)$, which is a smooth fibration, we have $\gD \subset \gD_f$. For the other inclusion, since a meromorphic section of a smooth elliptic fibration over a smooth variety is holomorphic~\cite[Lemma 1.3.5]{NakayamaLoc}, the bimeromorphic maps in~\eqref{comm-mod1cob} are biholomorphic outside of $p^{-1}(\gD)$. So $\sX_i \to U_i$ is isomorphic to $W_i \to U_i$ over $U_i \bss \gD$. In particular $\sX \to B$ is smooth over $B \bss \gD$. 
\end{proof}

\begin{lem}\label{lem-isomlisse}
If $f : \sX \to B$ is a tautological model associated to a $G$-equivariant elliptic fibration $\varphi : X \to B$ by virtue of Lemma-Definition~\ref{lem-locWeimodHdg}, then there exists a $G$-equivariant bimeromorphic map $\sX \dto X$ over $B$ which is biholomorphic over $B \bss \gD_\varphi$ where $\gD_\varphi$ is the discriminant locus of $\varphi$.
\end{lem}
\begin{proof}
Since both $f$ and $\varphi$ represent the same class $\eta_G(\varphi) \in H_G^1(B,\cJ_{\bH/B})$, we have a $G$-equivariant bimeromorphic map $\sX \dto X$ over $B$. Let $p : W \to B$ be the minimal Weierstra{\ss} fibration associated to $\varphi$ and  $\gD$  the discriminant locus of $p$. Since $\gD \subset \gD_\varphi$, by Lemma~\ref{lem-dicstaut} $f$ is smooth over $B \bss \gD_\varphi \subset B \bss \gD = B \bss \gD_f $. As $\varphi$ is also smooth over $B \bss \gD_\varphi$, by~\cite[Lemma 5.3.3]{NakayamaGlob} $\sX \dto X$ is biholomorphic over $B \bss \gD_\varphi$.
\end{proof}

\ssec{The tautological family associated to a tautological model}\label{ssec-famtaut}\hfill

We continue to use the notations introduced in the last paragraph. The goal of this paragraph is to construct a family of elliptic fibrations parameterized by $V \cnec H^1(B,\cL_{\bH/B})$ which contains $f : \sX \to B$ as the central fiber and which is tautological in the sense that the elliptic fibration parameterized by $t \in V$ represents $\eta + \exp'(t) \in H^1(B,\cJ_{\bH/B})$, where
$$\exp' : H^1(B,\cL_{\bH/B}) \to H^1(B,\cJ_{\bH/B})$$ 
is the composition of $\exp : H^1(B,\cL_{\bH/B}) \to H^1(B,\cJ^W_{\bH/B})$ with $\imath^W : H^1(B,\cJ^W_{\bH/B}) \to H^1(B,\cJ_{\bH/B})$. Moreover, if $t \in V^G$, then $t$ parameterizes the $G$-equivariant elliptic fibration representing $\eta_G + \exp'_G(t) \in H_G^1(B,\cJ_{\bH/B})$, where 
$$\exp'_G : H^1(B,\cL_{\bH/B})^G \to H^1_G(B,\cJ_{\bH/B})$$ 
is defined similarly. When $\dim B = 1$, tautological families have already been constructed by Kodaira in~\cite{KodairaSurfaceII}. In higher dimension, such families have also appeared in~\cite[Lemma 7.4.3]{NakayamaGlob} and~\cite[Lemma 4.1]{ClaudonHorpi1}.

\begin{pro-def}\label{pro-locWeimodHdg}
Let $B$ be a compact complex manifold and $B^\st \subset B$ a Zariski open such that $B \bss B^\st$ is a normal crossing divisor. Let $\bH$ be a weight one variation of Hodge structures of rank 2 over $B^\st$ and let $G$ be a finite group acting on both $B$ and $\bH$ in a compatible way. Let $p:W \to B$ denote the $G$-equivariant minimal Weierstra{\ss} fibration associated to $\bH$ and  $\gD \subset B$  the discriminant divisor of $p$.
 Given $\eta_G \in H^1_G(B,\cJ_{\bH/B})$ and assume that  $m\eta_G$ can be lifted to an element $\eta'_G \in H^1_G(B,\cJ^W_{\bH/B})$ for some $m \in \bZ \bss \{0\}$. Consider the tautological model $f:\sX = \sX\(m; (\eta_{ij}),(\eta^g_{i});\eta'_G\) \to B$  of $\eta_G$ where $\left\{(\eta_{ij}),(\eta^g_{i})\right\}$ is a 1-cocycle representing $\eta_G$. Then there exists a family of elliptic fibrations
$$\Pi : \cX \xto{q} B \times V \to V \cnec H^1(B,\cL_{\bH/B})$$ 
satisfying the following properties.
\begin{enumerate}[label = \roman{enumi})]
\item  \label{pro-locWeimodHdgi}  (Tautologicality)  The central fiber of $\Pi$ is $f : \sX \to B$ and
 the elliptic fibration parameterized by $t \in V$ represents the element $\eta + \exp'(t) \in H^1(B,\cJ_{\bH/B})$ where $\eta \in H^1(B,\cJ_{\bH/B})$ is the image of $\eta_G$. 
 \item  \label{pro-locWeimodHdgii}  ($G$-equivariance) There exists a $G$-action on $\cX^G \cnec q^{-1}(B \times V^G)$ such that the subfamily
 $$\Pi^G : \cX^G \xto{q^G \cnec q_{|\cX^G}} B \times V^G \to V^G$$
  of $\Pi$ parameterized by $V^G$ is $G$-equivariantly locally trivial over $B$ and the $G$-equivariant elliptic fibration parameterized by $t \in V^G$ corresponds to $\eta_G + \exp'_G(t) \in H_G^1(B,\cJ_{\bH/B})$.
\item (Multiplication-by-$m$) Let $\Pi' : \cW \to B \times V \to V$ be the tautological family associated to the image $\eta' \in H^1(B,\cJ^W_{\bH/B})$ of $\eta'_G$. There exists a map $\bm : \cX \to \cW$ over $B \times V$ whose restriction to each fiber over $V$ is the multiplication-by-$m$ defined in~\ref{ssec-tautmod}. Over $V^G$, the map $\bm$ is $G$-equivariant.
\end{enumerate}
The family $\Pi$ is called the \emph{tautological family} associated to $f : \sX \to B$.
\end{pro-def}

\begin{proof}
First we define 1-cocycle classes $\gl$, $\gl_G$, $\gt$, and $\gt_G$ that we will use to construct the family $\Pi$ and the $G$-action over $V^G$. Let  
\begin{equation*}
\begin{aligned}
\xi & \in V \otimes H^0(V,\cO_{V}) \simeq  H^1(B \times V,\cL_{\pr_1^{-1}\bH/B \times V}), \\
\xi_G & \in V^G \otimes H^0(V^G,\cO_{V^G})\simeq  H^1(B \times V^G,\cL_{\pr_1^{-1}\bH/B \times V^G})
\end{aligned}
\end{equation*}
 be the elements  corresponding to the identity maps $\Id_V $ and $\Id_{V^G}$. Let $\pr_1 : B \times V \to B$ be the first projection and define 
\begin{equation*}
\begin{aligned}
\gl & \cnec \pr_1^*\eta  + \wt{\exp}'(\xi) \in  H^1(B \times V,\cJ_{\pr_1^{-1}\bH/B \times V}), \\
\gl_G & \cnec \pr_1^*\eta_G + \wt{\exp}'_G(\xi_G) \in  H^1_G(B \times V^G,\cJ_{\pr_1^{-1}\bH/B \times V^G}), \\
 \gt & \cnec \pr_1^*\eta' +\wt{\exp}(m\xi) \in H^1(B \times V,\cJ^{W \times V}_{\pr_1^{-1}\bH/B \times V}), \\
\gt_G & \cnec \pr_1^*\eta'_G + \wt{\exp}_G(m\xi_G) \in  H^1_G(B \times V^G,\cJ^{W\times V^G}_{\pr_1^{-1}\bH/B \times V^G}).
\end{aligned}
\end{equation*}
where $\pr_1^* :  H^1(B,\cJ_{\bH/B})\to H^1(B \times V,\cJ_{\pr_1^{-1}\bH/B \times V})$ in the definition of $\gl$ is the map induced by $\cJ_{\bH/B} \to {\pr_1}_*\cJ_{\pr_1^{-1}\bH/B \times V}$ (the definitions of the other $\pr_1^*$ are similar), $\wt{\exp}$ is the map induced by $\exp : \cL_{\pr_1^{-1}\bH/B \times V} \to \cJ^{W\times V}_{\pr_1^{-1}\bH/B \times V}$, and $\wt{\exp}'$ is the composition of $\wt{\exp}$ with $H^1(B \times V,\cJ^{W \times V}_{\pr_1^{-1}\bH/B \times V}) \to H^1(B \times V,\cJ_{\pr_1^{-1}\bH/B \times V})$ (similar for $\wt{\exp}_G$ and $\wt{\exp}'_G$).

We specify the \v{C}ech 1-cocycles representing $\gl$ and $\gl_G$ as follows. Let $\{U_i\}_{i \in I}$ be a $G$-invariant good open cover of $B$ and let $U_{ij} \cnec U_i \cap U_j$. Let $\left\{(\eta_{ij}),(\eta^g_{i})\right\}$ be a 1-cocycle representing $\eta_G \in H^1_G(B,\cJ_{\bH/B})$ and let 
$$\ti{\eta}_{ij} \in \cJ_{\pr_1^{-1}\bH/B \times V}(U_{ij} \times V) \ \text{ and } \ \ti{\eta}_i^g \in \cJ_{\pr_1^{-1}\bH/B \times V^G}(U_i \times V^G)$$ 
be the pullbacks of $\eta_{ij}$ and $\eta^g_i$ under $U_{ij} \times V \to U_{ij}$ and $U_i \times V^G \to U_i$ respectively. As $V$ and $V^G$ are contractible, $\{U_i \times V\}$ and $\{U_i \times V^G\}$ are good covers of $B \times V$ and $B \times V^G$. 
Fix a \emph{$G$-invariant} 1-cocycle $(\xi_{ij})$ representing $\xi$. The 1-cocycles $\{\gl_{ij}\}$ and $\left\{(\gl'_{ij}),(\gl^g_{i})\right\}$ we choose to represent $\gl$ and $\gl_G$ are defined by
\begin{equation}\label{def-1cocycles}
\begin{split}
\gl_{ij}  & \cnec \ti{\eta}_{ij} + \exp(U_{ij} \times V)(\xi_{ij}) \in \cJ_{\pr_1^{-1}\bH/B \times V}(U_{ij} \times V), \\
\gl'_{ij} & \cnec {\gl_{ij}}_{| U_{ij} \times V^G } \in \cJ_{\pr_1^{-1}\bH/B \times V^G}(U_{ij} \times V^G), \\
\gl^g_{i} &  \cnec \ti{\eta}_i^g \in \cJ_{\pr_1^{-1}\bH/B \times V^G}(U_i \times V^G), 
\end{split}
\end{equation}
where $\exp : \cL_{\pr_1^{-1}\bH/B \times V} \to \cJ^{W \times V}_{\pr_1^{-1}\bH/B \times V}$ is the exponential map. We define the elliptic fibrations $q : \cX \to B \times V$ and $q^G : \cX^G \to B \times V^G$ to be the tautological models $\sX\(m ; (\gl_{ij}) ; \gt\) \to B \times V$ and $\sX\(m ; (\gl'_{ij}),(\gl_i^g) ; \gt_G\) \to B \times V^G$. In order to prove Proposition-Definition~\ref{pro-locWeimodHdg} for $q$ and $q^G$, we need to look into the construction of these tautological models. This will also be useful when we prove Proposition~\ref{pro-locWeimodHdgiii} in Section~\ref{sec-HdgTaut}.

Let $\left\{(\eta'_{ij}),(\eta'^g_{i})\right\}$ be a 1-cocycle representing $\eta'_G$ and let 
$$\ti{\eta}'_{ij} \in \cJ^{W \times V}_{\pr_1^{-1}\bH/B \times V}(U_{ij} \times V) \ \text{ and } \ \ti{\eta}'^g_i \in \cJ^{W \times V^G}_{\pr_1^{-1}\bH/B \times V^G}(U_i \times V^G)$$ 
be the pullbacks of $\eta'_{ij}$ and $\eta'^g_i$  under $U_{ij} \times V \to U_{ij}$ and $U_i \times V^G \to U_i$ respectively. Define
\begin{equation}\label{def-1cocyclesgt}
\begin{split}
\gt_{ij}  & \cnec \ti{\eta}'_{ij} + \exp(U_{ij} \times V)(m\xi_{ij}) \in \cJ^{W \times V}_{\pr_1^{-1}\bH/B \times V}(U_{ij} \times V), \\
\gt'_{ij} & \cnec {\gt_{ij}}_{| U_{ij} \times V^G } \in \cJ^{W \times V^G}_{\pr_1^{-1}\bH/B \times V^G}(U_{ij} \times V^G), \\
\gt^g_{i} &  \cnec \ti{\eta}'^g_i \in \cJ^{W \times V^G}_{\pr_1^{-1}\bH/B \times V^G}(U_i \times V^G), 
\end{split}
\end{equation}
so that $\gt$ and $\gt_G$ are represented by $\{\gt_{ij}\}$ and $\left\{(\gt'_{ij}),(\gt^g_{i})\right\}$. Since $m\eta_G =\imath_G^W(\eta'_G)$, up to refining $\{U_i\}$ there exist meromorphic sections $\gs_i$ of $W_i \cnec p^{-1}(U_i) \to U_i$ such that the 1-coboundary associated to $\{\gs_i\}$ modifies the 1-cocycle $\left\{(m\eta_{ij}),(m\eta^g_{i})\right\}$ to $\left\{(\eta'_{ij}),(\eta'^g_{i})\right\}$. So by construction, the 1-coboundary associated to $\{\gs_i \times \Id_V\}$ (resp. $\{\gs_i \times \Id_{V^G}\}$) modifies the 1-cocycle $\{m\gl_{ij}\}$ to $\{\gt_{ij}\}$ (resp. $\left\{(m\gl'_{ij}),(m\gl_i^g)\right\}$ to $\left\{(\gt'_{ij}),(\gt^g_{i})\right\}$). We emphasize that the 1-coboundaries modifying these 1-cocycles are of the form $\{\gs_i \times \Id\}$, which we will use to deduce that the bimeromorphic maps $\sX_i \times V \dto W_i \times V$ induced by the Stein factorizations in the construction of the tautological model are of the form $h_i \times  \Id : \sX_i \times V \dto W_i \times V$. 

As in the construction of $f : \sX \to B$, if $\bm_i$ and $\bm_{ij}$ denote the restrictions of the multiplication-by-$m$ map $W \to W$ to $W_i $ and $W_{ij} \cnec p^{-1}(U_{ij})$, then we have commutative diagrams
\begin{equation}
\begin{tikzcd}[cramped, row sep = 25, column sep = 35]
W_{ij} \times V \arrow[r, dashed, "\tr(\gl_{ij})"] \ar[d, dashed, swap, "\bm_{ij} \times \Id"] & W_{ji} \times V \ar[d, dashed, "\bm_{ji} \times \Id"] \\
W_{ij}\times V \arrow[d, dashed, swap, " \tr(\gs_i) \times \Id"]& W_{ji} \times V \arrow[d, dashed, "\tr(\gs_j) \times \Id"]  \\ 
W_{ij} \times V \arrow[r, "\sim"]\ar[r, swap, "\tr(\gt_{ij})"] &  W_{ji} \times V 
\end{tikzcd}
\ \ \ \ \ \
\begin{tikzcd}[cramped, row sep = 25, column sep = 85]
W_{gi}\times V^G \arrow[r, dashed, "\psi^g_i \times \Id \cnec \(\tr(\gl_i^g) \circ \phi_g\) \times \Id  "] \ar[d, dashed, swap, "\bm_{gi} \times \Id"] & W_{i}\times V^G  \ar[d, dashed, "\bm_{i} \times \Id"] \\
W_{gi}\times V^G  \arrow[d, dashed, swap, " \tr(\gs_{gi}) \times \Id"]& W_{i}\times V^G \arrow[d, dashed, " \tr(\gs_{i}) \times \Id"]  \\ 
W_{gi}\times V^G \arrow[r, "\sim"]\ar[r, "\phi_g \times \Id", swap] &  W_{i} \times V^G  
\end{tikzcd}
\end{equation} 
where $\phi_g : W \to W$ denotes the action of $g \in G$ on $W$ induced by the $G$-action on $\bH$. Let $\mu_i : \sX_i \to W_i$ be the finite map in the Stein factorization of $ \tr(\gs_i) \circ \bm_i$ and let $\sX_{ij} \cnec \mu_i^{-1}(W_{ij})$, $\sX_{ji} \cnec \mu_j^{-1}(W_{ji})$.
 Then there exist bimeromorphic maps $h_i : \sX_i \dto W_i$ over $U_i$ such that
\begin{equation}\label{diag-recollementW'}
\begin{tikzcd}[cramped, row sep = 20, column sep = 105]
\sX_{ij}\times V \arrow[r, dashed, "H_{ij} \cnec (h_j \times \Id)^{-1}\circ \tr(\gl_{ij}) \circ (h_i \times \Id)"] \ar[d, swap, "\mu_{i} \times \Id"] & \sX_{ji} \times V \ar[d, "\mu_{j} \times \Id"] \\ 
W_{ij} \times V \ar[r, "\tr(\gt_{ij})"] &  W_{ji} \times V   
\end{tikzcd}
\ \ \ \ \ \
\begin{tikzcd}[cramped, row sep = 20, column sep = 95]
\sX_{gi} \times V^G\arrow[r, dashed, "H_i^g \cnec   \(h_i^{-1}\circ \psi_i^g \circ  h_{gi}\) \times \Id"] \ar[d, swap, "\mu_{gi} \times \Id"]  & \sX_{i} \times V^G \ar[d, "\mu_{i} \times \Id"] \\ 
W_{gi} \times V^G\ar[r,"\phi_g \times \Id"] &  W_{i} \times V^G  
\end{tikzcd}
\end{equation}
are commutative. As both $\sX_{i} $ and $W_{i}$ are normal and $\mu_i $ is finite, the bimeromorphic maps $H_{ij}$ and $H^g_{i}$ are biholomorphic. Thus by gluing the $\mu_i : \sX_{i} \times V \to W_{i} \times V$ together using the 1-cocycle of biholomorphic maps $\{H_{ij}\}$, we obtain a tautological model $q : \cX \to B \times V$ of $\gl$, a locally Weierstra{\ss} fibration $q' : \cW \to B \times V$ twisted by $\gt$, and a finite surjective map $\bm : \cX \to \cW$ over $B \times V$.
Moreover, over $B \times V^G$ the fibration $q$ (resp. $q'$) is the $G$-equivariant tautological model $\cX^G \cnec \sX\(m ; (\gl'_{ij}),(\gl_i^g) ; \gt_G\) \to B \times V^G$ (resp. the $G$-equivariant locally Weierstra{\ss} fibration twisted by $\gt_G$) and the restriction of $\bm$ to $\cX^G$ is $G$-equivariant. This is the construction of the family $\Pi :  \cX \xto{q} B \times V \to V$ together with the $G$-action on $\cX^G$ and the map $\bm : \cX \to \cW$.

Obviously  $\Pi^G$ is $G$-equivariantly locally trivial over $B$. By construction, the central fiber of $\Pi^G$ is the $G$-equivariant tautological model $f$. As the elliptic fibration $\cX_t \to B$ parameterized by $t \in V$ in the family $\Pi$ is the tautological model $\sX\(m ; ({\gl_{ij}}_{|B \times \{t\}}) ; \gt_{|B \times \{t\}} \) \to B$ associated to $\gl_{|B \times \{t\}} = \eta + \exp'(t)$, it represents the element $\eta + \exp'(t)$. Similarly, the elliptic fibration parameterized by $t \in V^G$ represents $\eta_G + \exp'_G(t)$. This proves Proposition-Definition~\ref{pro-locWeimodHdg}. 
\end{proof}

\section{The Hodge theory of minimal Weierstra{\ss} fibrations}\label{sec-HdgTaut}

In this section, we will use Hodge theory to study minimal Weierstra{\ss} fibrations and derive some geometric consequences on the tautological families introduced at the end of Section~\ref{sec-locW}. Consider the following hypotheses on an elliptic fibration $f : X \to B$:
\begin{Hyp}\label{hyp-KahlerSection} \hfill
\begin{enumerate}[label = \roman{enumi})]
\item\label{hyp-KahlerSectioni} The base $B$ is a projective manifold. 
\item\label{hyp-KahlerSectionii} The total space $X$ is in the Fujiki class $\cC$.
\item\label{hyp-KahlerSectioniii} The discriminant locus $\gD \subset B$ is a normal crossing divisor.
\item\label{hyp-KahlerSectioniv} The local monodromies of $\bH \cnec (R^1f_*\bZ)_{|B\bss \gD}$ around $\gD$ are unipotent.
\end{enumerate}
\end{Hyp}
 The main result of this section that will be useful for proving Theorem~\ref{thm-mainfibellip} in Section~\ref{sec-preuve} is the following.  

\begin{pro}\label{pro-appalgpres}
Let $G$ be a finite group and $f : \sX \to B$ a $G$-equivariant tautological model satisfying Hypotheses~\ref{hyp-KahlerSection}. Let $Z \subset B$ be a $G$-stable subvariety. If $f^{-1}(\ol{Z\bss \gD}) \to \ol{Z\bss \gD}$ has a multi-section, then the tautological family associated to $f$ contains a subfamily 
$$\Pi' : \cX' \to B \times V' \to V'$$ 
 parameterized by some linear subspace $V' \subset H^1(B, \cL_{\bH/B})$ such that $\Pi'$ is an algebraic approximation of $\sX$ which is $G$-equivariantly locally trivial over $B$ and preserves $G$-equivariantly the completion $\rwh{f^{-1}(Z)} \to \hat{Z}$ of $f$ along $f^{-1}(Z) \to Z$. 
\end{pro}

We postpone the discussion on Hypotheses~\ref{hyp-KahlerSection} to the beginning of \S~\ref{ssec-densepair0}. With the notations introduced in Section~\ref{sec-locW}, we will first observe in \S~\ref{ssec-SH} that $H^1(B,j_*\bH)$ carries a natural (polarized) pure Hodge structure of weight 2 such that $H^1(B,j_*\bH) \to H^1(B,\cL_{\bH/B})$ induced by~\eqref{exseq-fibjac} is the projection to its $(0,2)$-part (see Lemma~\ref{lem-morStHdgminW}). Then we will apply Lemma~\ref{lem-morStHdgminW} to prove two results using Hodge theory (Proposition~\ref{pro-densepair0} and Proposition~\ref{pro-desingpre}) in \S~\ref{ssec-densepair0} and \S~\ref{ssec-stabres}. With the notations in Proposition~\ref{pro-appalgpres},  the first one is the density of the algebraic members in the subfamily of the tautological family parameterized by $V'$, and the second one states that $f^{-1}(Z) \to Z$ is preserved under the deformation of $f$ along $V'$, both assuming that  $f$ is a $G$-equivariant locally Weierstra{\ss} fibration twisted by some element $\eta_G \in H^1_G(B, \cJ^W_{\bH/B})$. Based on these results, we will prove Proposition~\ref{pro-appalgpres} in \S~\ref{ssec-desingpre}.

\ssec{A pure Hodge structure on $H^1(B,j_*\bH)$}\label{ssec-SH}
\hfill

The following result about the degeneration of Leray spectral sequences holds for any projective fibration over a projective manifold whose fibers are irreducible curves.

\begin{lem}\label{lem-specdegE2}
Let $p : W \to B$ be a projective fibration over a projective manifold $B$ such that fibers of $p$ are irreducible curves. Then the Leray spectral sequences of $p$ computing $H^2(W,\bQ)$ and $H^2(W,\cO_W)$ degenerate at $E_2$. In particular, 
\begin{equation}\label{isom-Leray11Q}
H^1(B,R^1p_*\bQ) \simeq \ker\(H^2(W,\bQ)/p^*H^2(B,\bQ) \xto{\imath^*} H^2(F,\bQ) \simeq \bQ \),
\end{equation}
where $\imath : F \hto W$ is the inclusion of a fiber of $p$ in $W$ and
\begin{equation}\label{isom-Leray11OW}
H^1(B,R^1p_*\cO_W) \simeq H^2(W,\cO_W)/p^*H^2(B,\cO_B).
\end{equation}
For the Leray spectral sequence computing $H^2(W,\bQ)$, without assuming that $B$ is smooth, we have 
\begin{equation}\label{isom-Leray02}
E_\infty^{0,2} = E_2^{0,2} \simeq H^0(B,R^2p_*\bQ) \simeq \bQ.
\end{equation}
\end{lem}

\begin{proof}
We may assume that $B$ is connected.
Let $W'$ be a desingularization of $W$. As the pullback $H^j(B,\bQ) \to H^j(W',\bQ)$ is injective (because $B$ and $W'$ are compact K\"ahler manifolds), the map $p^* : E_2^{j,0} \simeq H^j(B,\bQ) \to H^j(W,\bQ)$ through which it factorizes is also injective, especially for $j = 2$ or $3$. So $E_\infty^{2,0} = E_2^{2,0} \simeq  H^2(B,\bQ)$ and $E_\infty^{1,1} = E_2^{1,1} \simeq H^1(B,R^1p_*\bQ)$. As fibers of $p$ are irreducible curves, $R^2p_*\bQ$ is a local system of stalk $\bQ$. Since $W$ is projective by assumption,  the restriction of a very ample class $[H] \in H^2(W,\bQ)$ to $H^2(F,\bQ)$ is nonzero. So $H^2(W,\bQ) \to H^0(B,R^2p_*\bQ)$ is nonzero. In particular, $R^2p_*\bQ$ has a non-trivial section, so $R^2p_*\bQ \simeq \bQ$. It follows that $H^2(W,\bQ) \to H^0(B,R^2p_*\bQ) \simeq \bQ$, being nonzero, is surjective. Hence~\eqref{isom-Leray02} holds and~\eqref{isom-Leray11Q} follows easily.

For the Leray spectral sequence computing $H^2(X,\cO_X)$, by the same argument the pullback $p^* : E_2^{j,0} \simeq H^j(B,\cO_B) \to H^j(W,\cO_W)$ is injective, so $E_\infty^{2,0} = E_2^{2,0}$ and $E_\infty^{1,1} = E_2^{1,1}$. As $R^2p_*\cO_W = 0$ because fibers of $p$ are curves, we have $E_\infty^{0,2} = E_2^{0,2} = 0$,  which proves the degeneration of the Leray spectral sequence and~\eqref{isom-Leray11OW}.
\end{proof}

Using Lemma~\ref{lem-specdegE2}, we can now define a Hodge structure on $H^1(B,j_*\bH)$.

\begin{lem}\label{lem-morStHdgminW}
Let $p : W \to B$ be a minimal Weierstra{\ss} fibration over a projective manifold $B$. Assume that the discriminant locus $\gD$ is a normal crossing divisor. The morphism $H^1(B,j_*\bH) \otimes \bQ \to H^1(B,\cL_{\bH/B})$ induced by $\phi : j_*\bH \to \cL_{\bH/B}$ introduced in~\eqref{exseq-fibjac} is isomorphic to
\begin{equation}\label{mor-proj02}
\ker\(H^2(W,\bQ)/p^*H^2(B,\bQ) \xto{\imath^*} H^2(F,\bQ) \simeq \bQ \) \to H^2(W,\cO_W)/p^*H^2(B,\cO_B)
\end{equation}
induced by $\bQ \hto \cO_W$ where $\imath : F \hto W$ is the inclusion of a fiber of $p$ in $W$. In particular, there is a polarized pure $\bZ$-Hodge structure of weight 2 on $H^1(B,j_*\bH)$ satisfying the Hodge symmetry and the morphism $H^1(B,j_*\bH) \to H^1(B,\cL_{\bH/B})$ is the projection to the $(0,2)$-part of the Hodge structure.

Finally, let $\psi : Z \to B$ be a holomorphic map such that $Z$ is a projective manifold and $\psi^{-1}(\gD)$ is a normal crossing divisor. Assume that the Weierstra{\ss} fibration $W \times_B Z \to Z$ is minimal, then the pullback 
$$\psi^* : H^1(B,j_*\bH) \to H^1(Z,j'_*\bH')$$ 
is a morphism of Hodge structures where $j' : Z^\st \cnec Z \bss \psi^{-1}(\gD) \hto Z$ is the open immersion and $\bH' \cnec \psi^{-1}_{|Z^\st} \bH$. 
\end{lem}

\begin{proof}
By Lemma~\ref{lem-isophi}, the morphism $H^1(B,j_*\bH) \to H^1(B,\cL_{\bH/B})$ is isomorphic to 
$$H^1(B,R^1p_*\bZ) \to H^1(B,R^1p_*\cO_W)$$ 
and by Lemma~\ref{lem-specdegE2}, $H^1(B,R^1p_*\bQ) \to H^1(B,R^1p_*\cO_W)$ is isomorphic to~\eqref{mor-proj02}, which proves the first statement. 

Since $p : W \to B$ is a minimal Weierstra{\ss} fibration and $\gD$ is a normal crossing divisor, $W$ has at worst rational singularities. So the underlying Hodge structure on $H^2(W,\bZ)$ is pure of weight 2 satisfying the Hodge symmetry and $H^2(W,\bZ) \to H^2(W,\cO_W)$ is the projection to its $(0,2)$-part (\cite[Proof of Proposition 5]{NamikawaExt2forms}; for an explicit statement, see~\cite[Corollary B.2.8 and Proposition B.2.9]{KirschnerPeriodSymC}). As the $(0,2)$-part of $H^2(F,\bQ)$ is trivial,~\eqref{mor-proj02} is the projection of a Hodge structure to its $(0,2)$-part. This defines a polarized pure Hodge structure of weight 2  on $H^1(B,j_*\bH)$ satisfying the Hodge symmetry such that $H^1(B,j_*\bH) \to H^1(B,\cL_{\bH/B})$ is the projection to its $(0,2)$-part.

The pullback $\psi^*$ in the last statement is defined to be the composition 
$$H^1(B,j_*\bH) \xto{\psi^*} H^1(Z,\psi^{-1}(j_*\bH)) = H^1(Z,j'_*\bH').$$ 
From the isomorphism $$H^1(B,j_*\bH)\otimes \bQ \simeq \ker\(H^2(W,\bQ)/p^*H^2(B,\bQ) \to H^2(F,\bQ) \simeq \bQ \)$$ 
and the similar one for $H^1(Z,j'_*\bH') \otimes \bQ$, it follows immediately that the pullback $\psi^* : H^1(B,j_*\bH) \to H^1(Z,j'_*\bH')$ is a morphism of Hodge structures.
\end{proof}

\begin{rem}
In particular, Lemma~\ref{lem-morStHdgminW} implies that $H^1(B,j_*\bH) \otimes \bR \to H^1(B,\cL_{\bH/B})$ is surjective. This gives an alternative proof of the equivalent statement that the tautological family of a minimal Weierstra{\ss} fibration over a compact K\"ahler manifold is an algebraic approximation~\cite[Theorem 3.25 and Remark 3.26]{ClaudonHorpi1}.
\end{rem}

\ssec{Density of algebraic elliptic fibrations in the tautological family}\label{ssec-densepair0} 
\hfill

In most of the statements we prove in the rest of Section~\ref{sec-HdgTaut} about elliptic fibrations, we will assume Hypotheses~\ref{hyp-KahlerSection}. Some of the hypotheses therein are considered for the following reasons.  As we will apply Lemma~\ref{lem-morStHdgminW} to define a Hodge structure on $H^1(B,j_*\bH)$, we need to assume that $B$ is a projective manifold and the discriminant locus $\gD$ a normal crossing divisor. Since the Hodge structure on $H^1(B,j_*\bH)$ is defined using the minimal Weierstra{\ss} fibration associated to $\bH$ and we want this construction to be functorial under reasonable pullback, the corresponding pullback of the minimal Weierstra{\ss} fibration associated to $\bH$ need to be minimal. By Lemma~\ref{lem-functW}, this can be achieved if we assume that the local monodromies of $\bH$ around $\gD$ are unipotent.

Our first application of Lemma~\ref{lem-morStHdgminW} is the following density result.

\begin{pro}\label{pro-densepair0}
Let $f : X = W^\eta \to B$ be a $G$-equivariant locally Weierstra{\ss} fibration twisted by $\eta \in H^1(B,\cJ^W_{\bH/B})$ satisfying Hypotheses~\ref{hyp-KahlerSection} for some finite group $G$. Let  $Z \subset B$ be a $G$-stable subvariety such that none of the irreducible components of $Z$ is contained in $\gD$ and let $\imath : \ti{Z} \xto{\tau} Z \hto B$ be a $G$-equivariant log-desingularization of $(Z,Z \cap \gD)$. Assume that  $Y \cnec f^{-1}(Z) \to Z$ has a multi-section, then the subfamily parameterized by
$$V_{\ti{Z}}^G  \cnec \ker\(\imath^* : H^1(B, \cL_{\bH/B}) \to H^1(\ti{Z}, \imath^*\cL_{\bH/B})\)^G$$
of the tautological family $\Pi : \cW \to B \times V  \to V$ associated to $f$ is an algebraic approximation of $X$. 
\end{pro}

While proving Proposition~\ref{pro-densepair0}, we will also prove the following lemma along the way.

\begin{lem}\label{lem-latticeisom}
In the setting of Proposition~\ref{pro-densepair0}, there exists a lattice $\gL \subset V_{\ti{Z}}^G$ such that $\cW_t$ is isomorphic to $\cW_{t + \gl}$ over $B$ for all $t \in V_{\ti{Z}}^G$ and $\gl \in \gL$.
\end{lem}

\begin{proof}[Proof of Proposition~\ref{pro-densepair0} and Lemma~\ref{lem-latticeisom}]
To prove Proposition~\ref{pro-densepair0}, it suffices by Theorem~\ref{thm-multsecMoi} to prove that the subset of $V_{\ti{Z}}^G$ parameterizing elliptic fibrations admitting a multi-section in the tautological family $\Pi$ is dense in $V_{\ti{Z}}^G$ (for the Euclidean topology).  By Lemma~\ref{lem-Tormultsec}, it suffices to show that $V_{\ti{Z}}^G$ contains a dense subset $V_{\tors}$ such that $\eta + \exp(t) \in H^1(B, \cJ^W_{\bH/B})$ is torsion for all $t \in V_{\tors}$. 

We need the following two lemmas. Let $j_{\ti{Z}} : \ti{Z}^\st \cnec \ti{Z} \bss \imath^{-1}(\gD) \hto \ti{Z}$ be the open embedding and $\bH_{\ti{Z}} \cnec  (\imath^{-1}\bH)_{| \ti{Z}^\st}$. Let 
$$\phi : H^1(B,j_*\bH) \otimes \bQ \to  H^1(B,\cL_{\bH/B})$$ 
denote the map induced by $j_*\bH \to \cL_{\bH/B}$ introduced in~\eqref{exseq-fibjac} and let 
$$\imath^* : H^1(B,j_*\bH)^G \to H^1({\ti{Z}},j_{\ti{Z}*}\bH_{\ti{Z}})$$ 
be the restriction to $H^1(B,j_*\bH)^G$ of the pullback $H^1(B,j_*\bH) \to H^1({\ti{Z}},j_{\ti{Z}*}\bH_{\ti{Z}})$.

\begin{lem}\label{lem-Rgenlisse}
Let $K \cnec \ker\(\imath^* : H^1(B,j_*\bH)^G \to H^1({\ti{Z}},j_{\ti{Z}*}\bH_{\ti{Z}})\)$. Then $\phi(K_\bQ)$ is dense in $V_{\ti{Z}}^G$.
\end{lem}

In order to state the second lemma, recall that since $X$ is in the Fujiki class $\cC$, $c(\eta)$ is torsion by Proposition~\ref{pro-condK}. So there exist $m \in \bZ_{>0}$ and $\gb \in H^1(B, \cL_{\bH/B})$ such that $\exp(\gb) = m\eta$. Since $G$ is finite, up to replacing $m$ with a larger multiple of it, we may assume that $\gb \in H^1(B, \cL_{\bH/B})^G$.

\begin{lem}\label{lem-preimlisse}
Up to replacing $m$ with a larger multiple of it, there exists $\ga \in (H^1(B,j_*\bH)\otimes \bQ)^G$ such that $ \gb - \phi(\ga) \in V_{\ti{Z}}^G$. 
\end{lem}

Let us admit Lemma~\ref{lem-Rgenlisse} and~\ref{lem-preimlisse} for the moment and finish the proof. As $\phi(K_\bQ)$ is dense in $V_{\ti{Z}}^G$ by Lemma~\ref{lem-Rgenlisse} and $\gb - \phi(\ga) \in V_{\ti{Z}}^G$ by Lemma~\ref{lem-preimlisse}, the subset $V' \cnec \gb - \phi(\ga) - \phi(K_\bQ)$ is dense in $V_{\ti{Z}}^G$. If $t \in V'$, then $\gb - t \in \Ima(\phi)$, so $ m\eta - \exp(t) = \exp(\gb -t) \in H^1(B, \cJ^W_{\bH/B})$ is torsion.  It follows that  $\eta + \exp(t) \in H^1(B, \cJ_{\bH/B})$ is torsion whenever $t$ is in the dense subset  $V_{\tors} \cnec \frac{1}{m}V'$ of $V_{\ti{Z}}^G$, which proves Proposition~\ref{pro-densepair0}. To prove Lemma~\ref{lem-latticeisom}, note that Lemma~\ref{lem-Rgenlisse} implies that $\phi(K)$ contains a lattice $\gL$ of $V_{\ti{Z}}^G$. As $K \subset H^1(B,j_*\bH)$, we have $\cW_t \simeq \cW_{t + \gl}$ over $B$ for all $t \in V_{\ti{Z}}^G$ and $\gl \in \gL$. 
\end{proof}

Now we prove Lemma~\ref{lem-Rgenlisse} and~\ref{lem-preimlisse}. Both proofs use Hodge theory in an essential way.

\begin{proof}[Proof of Lemma~\ref{lem-Rgenlisse} and  Lemma~\ref{lem-preimlisse}]

Since $(\ti{Z},\imath^{-1}(\gD))$ is a log-smooth projective pair and since the local monodromies of $\bH$ around $\gD$ are assumed to be unipotent, by Lemma~\ref{lem-functW} $\cL_{\bH_{\ti{Z}}/\ti{Z}} = \imath^*\cL_{\bH/B}$ and the base change $W_{\ti{Z}} \cnec W \times_B \ti{Z} \to \ti{Z}$ by $\imath : \ti{Z} \to B$ of the minimal Weierstra{\ss} fibration $p : W \to B$ associated to $\bH$ is the minimal Weierstra{\ss} fibration associated to $\bH_{\ti{Z}}$. Also, since both $B$ and $\ti{Z}$ are projective manifolds and $\gD$ and $\imath^{-1}(\gD)$ are normal crossing divisors,
by Lemma~\ref{lem-morStHdgminW}
$K$ is a pure Hodge structure of weight 2 satisfying the Hodge symmetry and the restriction of $\phi$ to $K_\bQ$ is the projection of $K_\bQ$ to its $(0,2)$-part 
$$\ker\(H^1(B, \cL_{\bH/B})^G \to H^1(\ti{Z}, \cL_{\bH_{\ti{Z}}/\ti{Z}})\) = \ker\(H^1(B, \cL_{\bH/B})^G \to H^1(\ti{Z}, \imath^*\cL_{\bH/B})\) = V_{\ti{Z}}^G.$$ 
Hence $\phi(K_\bQ)$ is dense in $V_{\ti{Z}}^G$, which proves Lemma~\ref{lem-Rgenlisse}. 

The short exact sequence~\eqref{exseq-fibjac} together with its pullback by $\imath$ induces the following commutative diagram with exact rows.
$$
\begin{tikzcd}[cramped, row sep = 20, column sep = 25]
 H^1(B,j_*\bH) \ar[r,"\phi"] \ar[d, "\imath^*"]   & H^1(B,\cL_{\bH/B}) \ar[d, "\imath^{0,2}"]\ar[r, "\exp"]     & H^1(B,\cJ^W_{\bH/B}) \ar[d, "\imath^*"] \\
  H^1(\ti{Z},(j_{\ti{Z}})_*\bH_{\ti{Z}}) \ar[r, "\phi_{\ti{Z}}"]  & H^1(\ti{Z},\cL_{\bH_{\ti{Z}}/\ti{Z}}) \ar[r, "\exp"]        & H^1(\ti{Z},\cJ^{W_{\ti{Z}}}_{\bH_{\ti{Z}}/\ti{Z}})  
\end{tikzcd}
$$
Here in this proof, we use $\imath^{0,2}$ to denote the pullback $\imath^* : H^1(B,\cL_{\bH/B}) \to H^1(\ti{Z},\cL_{\bH_{\ti{Z}}/\ti{Z}})$.

Let $g: Y \cnec X \times_B \ti{Z} \to \ti{Z}$, which is isomorphic to the locally Weierstra{\ss} fibration $W_{\ti{Z}}^{\imath^*\eta} \to \ti{Z}$ twisted by $\imath^*\eta \in H^1(\ti{Z},\cJ^{W_{\ti{Z}}}_{\bH_{\ti{Z}}/\ti{Z}})$. Since $g$ has a multi-section by assumption and $Y$ is in the Fujiki class $\cC$, $\imath^*\eta  $ is torsion by Lemma~\ref{lem-multisecTor}. Therefore up to replacing $m$ with a larger multiple, $\exp(\imath^{0,2}(\gb)) = 0$, thus again up to replacing $m$ with a larger multiple of it, there exists $\ga_0 \in  H^1(\ti{Z},(j_{\ti{Z}})_*\bH_{\ti{Z}})^G$ such that $\phi_{\ti{Z}}(\ga_0) = \imath^{0,2}(\gb)$.

Let $H \cnec H^1(B,j_*\bH)$ and $H_{\ti{Z}} \cnec H^1(\ti{Z},(j_{\ti{Z}})_*\bH_{\ti{Z}}) $ for simplicity. As $H_{\ti{Z}}$ is a polarized Hodge structure, there exists a $\bQ$-Hodge structure $H'_{\bQ}$ such that $H_{\ti{Z}} \otimes \bQ = \imath^*H^G_\bQ \oplus H'_{\bQ}$. As $\phi_{\ti{Z}}$ is the projection of the Hodge structure to its $(0,2)$-part by Lemma~\ref{lem-morStHdgminW}, $\phi_{\ti{Z}}  : \imath^*H^G_\bQ \oplus H'_{\bQ} \to \imath^{0,2}(H^{0,2})^G \oplus {H'}^{0,2}$ preserves the two factors in the direct sum where ${H}^{0,2}$ and ${H'}^{0,2}$ are the $(0,2)$-parts of $H_\bQ$ and $H'_\bQ$. So since $\gb \in H^1(B, \cL_{\bH/B})^G =(H^{0,2})^G$ where the equality follows from Lemma~\ref{lem-morStHdgminW}, there exists $\ga \in H^G_\bQ$ such that $\phi_{\ti{Z}}(\imath^*\ga) = \imath^{0,2}(\gb)$. As $\phi_{\ti{Z}}(\imath^*\ga) = \imath^{0,2}(\phi(\ga))$, we have $\imath^{0,2}(\gb - \phi(\ga)) = 0$. Hence $\gb - \phi(\ga) \in V_{\ti{Z}}^G$.
\end{proof}

\begin{rem}
In Lemma~\ref{lem-latticeisom}, the hypotheses that the total space $X$ is in the Fujiki class $\cC$ and $Y \to Z$ has a multi-section are unnecessary.
\end{rem}

\ssec{A stability result}\label{ssec-stabres}
\hfill

Let $f : W^\eta \to B$ be an $\eta$-twisted locally Weierstra{\ss} fibration satisfying Hypotheses~\ref{hyp-KahlerSection} for some $\eta \in H^1(B,\cJ_{\bH/B}^W)$. Let $Z \subset B$ be a subvariety of $B$. In the last paragraph, we showed that if $f^{-1}(\ol{Z\bss \gD}) \to \ol{Z\bss \gD} =: Z_0$ has a multi-section, then the subfamily parameterized by 
$$V_{\ti{Z}_0} = \ker \( \ti{\imath}_0^* :H^1(B,\cL_{\bH/B}) \to H^1(\ti{Z}_0,\ti{\imath}_0^*\cL_{\bH/B})\)$$ 
of the tautological family associated to $f$ is an algebraic approximation of $f$ where $\ti{\imath}_0 : \ti{Z}_0 \to Z_0 \subset B$ is a log-desingularization of $(Z_0,Z_0 \cap \gD)$. Our next application of Lemma~\ref{lem-morStHdgminW} is to show that this subfamily preserves the fibration $f^{-1}(Z)  \to Z$. 

\begin{pro}\label{pro-desingpre}
Let $\eta \in H^1(B,\cJ_{\bH/B}^W)$ and $f : W^\eta \to B$ be the associated locally Weierstra{\ss} fibration.
Assume that $f$ satisfies Hypotheses~\ref{hyp-KahlerSection}. Let $Z \subset B$ be a subvariety of $B$ and $Z_0 \cnec \ol{Z \bss \gD}$. Let $\ti{\imath}_0 : \ti{Z}_0 \xto{\tau_0} Z_0 \hto B$ be a log-desingularization of $(Z_0,Z_0 \cap \gD)$. Then the subfamily of the tautological family associated to $f$ parameterized by 
$$ V_{\ti{Z}_0} \cnec \ker \( \ti{\imath}_0^* :H^1(B,\cL_{\bH/B}) \to H^1(\ti{Z}_0,\ti{\imath}_0^*\cL_{\bH/B})\)$$
 preserves the fibration $f^{-1}(Z) \to Z$. 
\end{pro}

We first prove some auxiliary results before we start the proof of Proposition~\ref{pro-desingpre}. Let $p : W \to B$ be a Weierstra{\ss} fibration over a projective manifold and $\psi : Z \to B$ a map from a projective manifold $Z$.

\begin{lem}\label{lem-pure11}
 Assume that the image of $\psi : Z \to B$ is contained in the discriminant locus $\gD$. Let $q : Y \cnec W \times_B Z \to Z$ denote the base change fibration and $\tau : \ti{Y} \to Y$ a desingularization of $Y$. Let $\ti{q} \cnec q \circ \tau$. Then the quotient $H^2(\ti{Y},\bZ)  / \ti{q}^*H^2(Z,\bZ)$ is a pure Hodge structure of weight two concentrated in bi-degree $(1,1)$.
\end{lem}

\begin{proof}
First we note that since $Y$ and $\tau$ are projective, $\ti{Y}$ is also projective. Therefore $H^2(\ti{Y},\bZ)$ is a pure Hodge structure of weight two and it suffices to show that $\ti{q}^* : H^0(Z, \gO^2_{Z}) \to H^0(\ti{Y}, \gO^2_{\ti{Y}}) $ is an isomorphism. The morphism $\ti{q}^*$ is injective because $\ti{q}$ is a surjective map between  projective manifolds. All we need to prove is that $h^{2,0}(Z) \ge h^{2,0}(\ti{Y})$. Since the zero-section $\gS \subset W$ of $p : W \to B$ is contained in the smooth locus of $f$, its pullback to $q : Y \to Z$ lies in the smooth part of $Y$, so  $\ti{q}$ has a holomorphic section $\gs : Z \to \ti{Y}$. Consider the pushforward 
$$\chi :  \ti{q}_*\gO^2_{\ti{Y}} \to  (\ti{q} \circ \gs)_*\gO^2_{Z} = \gO^2_{Z}$$
 of $\gO^2_{\ti{Y}} \to \gs_*\gO^2_{Z}$ by $\ti{q} $. 
Since $\ti{q}$ is a $\bP^1$-bundle over a dense Zariski open of $Z$, $\chi$ is generically an isomorphism. Moreover since $\ti{q}_*\gO^2_{\ti{Y}}$ is torsion free, $\chi$ is injective. It follows that
 $h^{2,0}(Z) \ge h^{2,0}(\ti{Y})$. 
\end{proof}

Lemma~\ref{lem-pure11} will be used to prove the following result.

\begin{lem}\label{lem-nul}
In the setting of Lemma~\ref{lem-pure11}, assume that $p : W \to B$ is a minimal Weierstra{\ss} fibration. If we define 
$$K_0 \cnec \ker\( H^2(W,\bZ)/p^*H^2(B,\bZ) \xto{\psi^*} H^2(Y,\bZ)  / q^*H^2(Z,\bZ)\),$$ 
then $K_0$ is a pure sub-Hodge structure of $H^2(W,\bZ)/p^*H^2(B,\bZ)$ of weight 2 whose $(0,2)$-part coincides with that of $H^2(W,\bZ)/p^*H^2(B,\bZ)$, which is $H^2(W,\cO_W)/p^*H^2(B,\cO_B)$.
\end{lem}

\begin{proof}
Since $p$ is assumed to be minimal, $W$ has at worst rational singularities. It follows from~\cite[Corollary B.2.8]{KirschnerPeriodSymC} that the underlying Hodge structure of $H^2(W,\bZ)/p^*H^2(B,\bZ)$ is pure of weight 2, and so is $K_0$. Let $H \cnec  H^2(W,\bC)/p^*H^2(B,\bC)$ and $H' \cnec H^2(Y,\bC)/q^*H^2(Z,\bC)$ for simplicity. The map $\psi^*$ is a morphism of mixed Hodge structures and it suffices to show that $\psi^*(F^1H) = \psi^*(H)$ to prove the assertion about the $(0,2)$-part of $K_0$.

To this end, let $\ga \in H$. Let $\bar{\psi} : H \to \Gr^0_\bW H'$ be the composition of $\psi^*$ with the projection $H' \to \Gr^0_\bW H'$ (here, $\bW$ denotes the weight filtration of the mixed Hodge structure $H'$). By Lemma~\ref{lem-pure11}, $\Gr^0_\bW H' \subset H^2(\ti{Y},\bC)  / \ti{q}^*H^2(Z,\bC)$ is a pure Hodge structure of weight two concentrated in bi-degree $(1,1)$, so $\bar{\psi}(H) = F^1\Gr^0_\bW H' \cap \bar{\psi}(H) =  \bar{\psi}(F^1H)$ where the second equality follows from the strictness of morphisms of pure Hodge structures~\cite[Lemma 7.23]{VoisinI}. Thus there exist $\beta \in F^1H$ and $\gamma \in \bW_{-1}H'$ such that $\psi^*(\ga) = \psi^*(\gb) + \gamma$. Since $\Ima\psi^* \cap \bW_{-1}H' = \psi^*(\bW_{-1}H) = 0$ as $H$ is pure, we have $\gamma = \psi^*(\ga - \gb) = 0$. Hence $\psi^*(\ga) = \psi^*(\gb) \in \psi^*(F^1H)$.
\end{proof}

\begin{proof}[Proof of Proposition~\ref{pro-desingpre}]
By Lemma~\ref{lem-presWtauto}, it suffices to prove that the restriction to   
$V_{\ti{Z}_0}$
of the pullback $\imath^* : H^1(B,\cL_{\bH/B}) \to H^1(Z,\imath^*\cL_{\bH/B})$ by the inclusion $\imath : Z \hto B$ is zero. Note that this statement does not depend on $\eta$.

For simplicity, we may assume that $Z$ contains $\gD$. Let $\tau_{\gD} : \ti{\gD} \to \gD$ be a desingularization of $\gD$ and let 
$$\tau \cnec (\tau_0 \sqcup \tau_{\gD}) : \ti{Z} \cnec \ti{Z}_0  \sqcup \ti{\gD}\to Z.$$  We denote by $p : W\to B$  the minimal Weierstra{\ss} fibration associated to $\bH = (R^1f_*\bZ)_{|B \bss \gD}$ and let 
\begin{equation*}
\begin{aligned}
q & :Y \cnec W \times_B Z \to Z, \\
\ti{q}_0 & : \ti{Y}_0 \cnec W \times_B \ti{Z}_0 \to \ti{Z}_0, \\ 
\ti{q}_{\gD} & : \ti{D} \cnec W \times_B \ti{\gD} \to \ti{\gD}, \\
\ti{q} & : \ti{Y} \cnec W \times_B \ti{Z} \to \ti{Z}
\end{aligned}
\end{equation*}
be various base changes of $p : W \to B$. Let $\ti{\imath} \cnec \imath \circ \tau$ and
$$K \cnec \ker\( \ti{\imath}^* : H^2(W,\bC)/p^*H^2(B,\bC) \to H^2(\ti{Y},\bC)/\ti{q}^*H^2(\ti{Z},\bC) \).$$
We have a commutative diagram
\begin{equation}\label{diagfibW}
\begin{tikzcd}[cramped, row sep = 20, column sep = 30]
K \ar[r, hook]  \ar[d, two heads, "\pi"] & \ddfrac{H^2(W,\bC)}{p^*H^2(B,\bC)} \ar[r,"\ga"]  \ar[d] & \ddfrac{H^2(Y,\bC)}{{q}^*H^2(Z,\bC)} \ar[d] \ar[r,"\gb"] & \ddfrac{H^2(\ti{Y},\bC)}{(\tau \circ \ti{q})^*H^2(Z,\bC)} \ar[r,"\gamma"]  & \ddfrac{H^2(\ti{Y},\bC)}{\ti{q}^*H^2(\ti{Z},\bC)} \ar[d] \\  
V_{\ti{Z}_0} \ar[r, hook]  & H^1(B,\cL_{\bH/B})  \ar[r]  & H^1({Z},\imath^*\cL_{\bH/B})  \ar[rr] & & H^1(\ti{Z},\ti{\imath}^*\cL_{\bH/B}) 
\end{tikzcd}
\end{equation}
where the arrows are defined as follows. The horizontal arrows (except the leftmost ones) are induced by pullbacks under various morphisms. The map  $H^2(Y,\bC)/{q}^*H^2(Z,\bC) \to H^1({Z},\imath^*\cL_{\bH/B})$ is defined to be the composition
$$H^2(Y,\bC)/{q}^*H^2(Z,\bC) \to H^2(Y,\cO_Y)/{q}^*H^2(Z,\cO_Z) \simeq H^1(Z,R^1q_*\cO_Y) \simeq H^1({Z},\imath^*R^1p_*\cO_W) \simeq H^1({Z},\imath^*\cL_{\bH/B}),$$
where the first and the third isomorphisms result from Lemma~\ref{lem-specdegE2} and Lemma~\ref{lem-isophi} respectively, and the second from the Grauert base change theorem~\cite[Theorem III.4.6.(1)]{scvVII} since $q : Y \to Z$ is the base change of the flat fibration $p: W \to B$. The other vertical arrows are defined similarly except for $\pi$. As $$H^1(\ti{Z},\ti{\imath}^*\cL_{\bH/B}) = H^1(\ti{\gD},\ti{\imath}_{\gD}^*\cL_{\bH/B}) \oplus H^1(\ti{Z}_0,\ti{\imath}_0^*\cL_{\bH/B})$$  where $\ti{\imath}_i \cnec \imath \circ \tau_i$ and the composition 
\begin{equation*}
\begin{tikzcd}[cramped, row sep = 20, column sep = 30]
K \ar[r] & \ddfrac{H^2(W,\bC)}{p^*H^2(B,\bC)} \ar[r] & H^1(B,\cL_{\bH/B}) \ar[r, swap, "\ti{\imath}^*"] \ar[rr, bend left = 15, "\ti{\imath}_0^*"] & H^1(\ti{Z},\ti{\imath}^*\cL_{\bH/B}) \ar[r, swap, "\pr_2"]  & H^1(\ti{Z}_0,\ti{\imath}_0^*\cL_{\bH/B}) 
\end{tikzcd}
\end{equation*}
is zero, the image of the composition $K \to H^2(W,\bC)/p^*H^2(B,\bC) \to H^1(B,\cL_{\bH/B})$
is contained in $V_{\ti{Z}_0}$, which defines $\pi : K \to V_{\ti{Z}_0}$ in~\eqref{diagfibW}.

\begin{lem}
The map $\pi : K \to V_{\ti{Z}_0} $ is the projection of the (pure) Hodge structure $K$ of weight 2 onto its $(0,2)$-part. In particular, $\pi$ is surjective.
\end{lem}

\begin{proof}
First of all 
$$\ddfrac{H^2(\ti{Y},\bC)}{\ti{q}^*H^2(\ti{Z},\bC)} = \ddfrac{H^2(\ti{Y}_{\gD},\bC)}{\ti{q}_{\gD}^*H^2(\ti{\gD},\bC)}  \oplus \ddfrac{H^2(\ti{Y}_0,\bC)}{\ti{q}_0^*H^2(\ti{Z}_0,\bC)},$$
so 
$$K = \ker \(K_0 \hto H^2(W,\bC)/p^*H^2(B,\bC) \xto{\ti{\imath}_0^*} H^2(\ti{Y}_0,\bC)/\ti{q}_0^*H^2(\ti{Z}_0,\bC)\)$$ 
where
$$K_0 \cnec \ker\( \ti{\imath}_{\gD}^* : H^2(W,\bC)/p^*H^2(B,\bC) \to H^2(\ti{Y}_{\gD},\bC)/\ti{q}_{\gD}^*H^2(\ti{\gD},\bC)\).$$
By Lemma~\ref{lem-nul},  $K_0$ is a pure Hodge structure of weight 2 whose $(0,2)$-part is isomorphic to that of $H^2(W,\bC)/p^*H^2(B,\bC)$, which is $H^2(W,\cO_W)/p^*H^2(B,\cO_B)$, and is further isomorphic to $H^1(B,\cL_{\bH/B})$ by Lemma~\ref{lem-morStHdgminW}. By Lemma~\ref{lem-functW}, $\ti{q}_0 : \ti{Y}_0 \to \ti{Z}_0$ is the minimal Weierstra{\ss} fibration associated to $\bH' \cnec \ti{\imath}_{|\ti{Z}_0 \bss \ti{\imath}^{-1}(\gD)}^{-1}\bH$ and $\cL_{\bH'/\ti{Z}_0} = \ti{\imath}_0^*\cL_{\bH/B}$. So again by Lemma~\ref{lem-morStHdgminW}, $H^2(\ti{Y}_0,\bC)/\ti{q}_0^*H^2(\ti{Z}_0,\bC)$ is a pure Hodge structure of weight 2 whose $(0,2)$-part is isomorphic to $H^1(\ti{Z}_0,\cL_{\bH'/\ti{Z}_0}) = H^1(\ti{Z}_0,\ti{\imath}_0^*\cL_{\bH/B})$. Therefore $V_{\ti{Z}_0} = \ker \( \ti{\imath}_0^* :H^1(B,\cL_{\bH/B}) \to H^1(\ti{Z}_0,\ti{\imath}_0^*\cL_{\bH/B})\)$ is the $(0,2)$-part of the pure Hodge structure $K$ of weight 2.
\end{proof}

\begin{lem}\label{lem-resinj}
In~\eqref{diagfibW}, the restriction of $\gamma$ to $\Ima(\gb)$ is injective.
\end{lem}

\begin{proof}

First of all let $E_{\bullet,Y}^{\bullet,\bullet}$ and $E_{\bullet,\ti{Y}}^{\bullet,\bullet}$ be the Leray spectral sequences of $q : Y \to Z$ and $\ti{q} : \ti{Y} \to \ti{Z}$ computing $H^2(Y,\bQ)$ and $H^2(\ti{Y},\bQ)$ respectively. By Lemma~\ref{lem-specdegE2}, we have $E_{\infty,Y}^{0,2} \simeq H^0(Z,R^2q_*\bQ) \simeq \bQ$  and $E_{\infty,\ti{Y}}^{0,2} \simeq H^0(\ti{Z},R^2\ti{q}_*\bQ) \simeq \bQ$. 

Let $\xi \in H^2(Y,\bQ)$ such that $\tau_Y^*\xi \in \ti{q}^*H^2(\ti{Z},\bQ)$ where $\tau_Y^* : H^2(Y,\bQ) \to H^2(\ti{Y},\bQ)$ is the pullback by the projection $\tau_Y : \ti{Y} = Y \times_Z \ti{Z} \to Y$. Since $\tau_Y^*\xi \in \ti{q}^*H^2(\ti{Z},\bQ)$, the projection of $\tau_Y^*\xi$ in $E_{\infty,\ti{Y}}^{0,2}$ is zero. As the pullback 
$$\tau^* : \bQ \simeq H^0(Z,R^2q_*\bQ)  \to  H^0(\ti{Z},R^2\ti{q}_*\bQ) \simeq \bQ$$ 
is an isomorphism, it follows that the projection of $\xi$ in $E_{\infty,Y}^{0,2}$ is zero. So $\xi \in L^1H^2(Y,\bQ)$.

It remains to show that the projection of $\xi$ in $E_{\infty,Y}^{1,1}$ is zero, so that $\xi \in E_{\infty,Y}^{2,0} =  {q}^*H^2({Z},\bQ)$. Note that since $\tau_Y^*\xi \in \ti{q}^*H^2(\ti{Z},\bQ)$, the image of $\tau_Y^*\xi$ in $E_{\infty,\ti{Y}}^{1,1}$ is zero.
Thus it suffices to show that the pullback $E_{\infty,Y}^{1,1} \to E_{\infty,\ti{Y}}^{1,1}$ is injective. Since  $E_{\infty,Y}^{1,1} \to E_{\infty,\ti{Y}}^{1,1}$ is contained in $E_{2,Y}^{1,1} \to E_{2,\ti{Y}}^{1,1}$, it suffices to show that $H^1(Z,R^1q_*\bQ) \to H^1(\ti{Z},R^1\ti{q}_*\bQ)$ is injective. 

By the base change theorem~\cite[VII.2.6]{IversenBook}, $R^1\ti{q}_*\bQ \simeq \tau^{-1}R^1q_*\bQ$\footnote{In~\cite{IversenBook}, the symbol $f^*$ is used in place of $f^{-1}$.}. So by the projection formula~\cite[VII.2.4]{IversenBook} $H^1(Z,R^1q_*\bQ)$ is in fact the $E_2^{1,0}$-term of the Leray spectral sequence of $\tau$ computing $H^1(\ti{Z},R^1\ti{q}_*\bQ)$. Hence $H^1(Z,R^1q_*\bQ) \to H^1(\ti{Z},R^1\ti{q}_*\bQ)$ is injective.
\end{proof}

Given a  morphism $\phi : M_1 \to M_2$ of mixed Hodge structures, let $\bar{\phi} : \Gr_\bW^0M_1 \to \Gr_\bW^0M_2$ denote the induced morphism on the 0-th graded pieces. Recall that by the strictness of the weight filtration, if $\phi : H \to M$ is a morphism of mixed Hodge structures such that $H$ is pure (in the sense that $\bW_0H = H$ and $\bW_{-1}H = 0$), then $\phi(H) \simeq \bar{\phi}(H)$. This is a property that will be repeatedly used in the next paragraph concluding the proof of Proposition~\ref{pro-desingpre}.

Assume that $\ga(K) \ne 0$, then $\bar{\ga}(K) \ne 0$ because $K$ is a pure Hodge structure. As $\ti{Y} \to Y$ is proper and surjective, $\bar{\gb}$ is injective~\cite[Theorem 5.41]{PetersSteenbrinkMHS}, so $\bar{\gb}(\bar{\ga}(K)) \ne 0$. Accordingly $\gb(\ga(K)) \ne 0$, so $\gamma(\gb(\ga(K))) \ne 0$ by Lemma~\ref{lem-resinj}, which contradicts the definition of $K$. Hence $\ga(K) = 0$. As $K \to V_{\ti{Z}_0}$ is surjective, the map $V_{\ti{Z}_0} \to H^1(Z,\imath^*\cL_{\bH/B})$ is zero, which proves Proposition~\ref{pro-desingpre}.
\end{proof}

Before we turn to the proof of Proposition~\ref{pro-appalgpres}, let us prove the following corollary of Proposition~\ref{pro-densepair0} and Proposition~\ref{pro-desingpre} as an aside.

\begin{cor}
Let $f:X \to B$ be a locally Weierstra{\ss} fibration twisted by $\eta \in H^1(B,\cJ_{\bH/B}^W)$ satisfying Hypotheses~\ref{hyp-KahlerSection} and let $Z \subset B$ be a subvariety.  Assume that $Y \cnec f^{-1}(Z) \to Z$ has a multi-section and $Y \to Z$  is not preserved along any direction in the tautological family associated to $\eta$. Then $X$ is already Moishezon.  
\end{cor}

\begin{proof}
Since $Y \to Z$ is not preserved along any direction in the tautological family, the subspace $V_{\ti{Z}_0}$ defined in Proposition~\ref{pro-desingpre} is zero. Now by Proposition~\ref{pro-densepair0}, the subfamily of the tautological family associated to $f$ parameterized by $V_{\ti{Z}_0}$ is an algebraic approximation of $f$. So $X$ is already Moishezon. 
\end{proof}

\ssec{Proof of Proposition~\ref{pro-appalgpres}}\label{ssec-desingpre}\hfill

The following proposition generalizes Proposition~\ref{pro-desingpre}, in which locally Weierstra{\ss} fibration  is replaced by  $G$-equivariant tautological model (see~\ref{ssec-tautmod}), and the fibration $f^{-1}(Z) \to Z$ is replaced by the formal completion $\rwh{f^{-1}(Z)} \to \hat{Z}$.

\begin{pro}\label{pro-locWeimodHdgiii}
Let $G$ be a finite group and $f:\sX \to B$ a $G$-equivariant tautological model associated to an element $\eta_G \in H^1_G(B,\cJ_{\bH/B})$ satisfying Hypotheses~\ref{hyp-KahlerSection}. Let $Z \subset B$ be a $G$-stable subvariety of $B$ and $Z_0 \cnec \ol{Z \bss \gD}$. Let $\ti{\imath}_0 : \ti{Z}_0 \xto{\tau_0} Z_0 \hto B$ be a $G$-equivariant log-desingularization of $(Z_0,Z_0 \cap \gD)$ such that $\ti{Z}_0$ is projective. Then the subfamily parameterized by 
$$ V_{\ti{Z}_0}^G \cnec \ker \( \ti{\imath}_0^* : H^1(B,\cL_{\bH/B}) \to H^1(\ti{Z}_0,\ti{\imath}_0^*\cL_{\bH/B})\)^G$$ 
of the tautological family 
$$\Pi : \cX \xto{q} B \times  V \to  V \cnec H^1(B,\cL_{\bH/B})$$ 
associated to $f$ preserves $G$-equivariantly the completion $\rwh{f^{-1}(Z)} \to \hat{Z}$ of $f$ along $f^{-1}(Z) \to Z$.
\end{pro}

Let us first prove some general results before we prove Proposition~\ref{pro-locWeimodHdgiii}.

\begin{lem}\label{lem-prescompact}
Let $\Pi : \cX \xto{q} B \times V \to V$ be a deformation of a fibration $f : X \to B$ over a complex vector space $V$ and assume that $q$ is flat. Let $Z$ be a subvariety of $B$ such that $Y \cnec f^{-1}(Z) \to Z$ is preserved by $\Pi$. Assume that there exists a lattice $\gL \subset V$ such that $\cX_t$ is isomorphic to $\cX_{t + \gl}$ over $B$ for all $t \in V$ and $\gl \in \gL$. Then the completion $\hat{Y} \to \hat{Z}$ of $X \to B$ along $Y \to Z$ is also preserved by $\Pi$.

If moreover $f$ is $G$-equivariant, $Z$ is $G$-stable, and $\Pi$ preserves  $G$-equivariantly  $Y \to Z$ for some finite group $G$, then $\Pi$ preserves $G$-equivariantly $\hat{Y} \to \hat{Z}$ as well.
\end{lem}
\begin{proof}
For the first statement, it suffices to prove by induction on $n \in \bZ_{\ge 0}$ that the $n$-th order infinitesimal neighborhood $g_n : Y_n \to Z_n$ of $Y \to Z$ is preserved by $\Pi$. The case where $n=0$ is covered by the assumption. Assume that $g_{n-1} : Y_{n-1} \to Z_{n-1}$ is preserved by $\Pi$. Let $\cI_Z \subset \cO_{B}$, $\cI_Y \subset \cO_{X}$, $\cI \subset \cO_{\cX}$, and $\cI_t \subset \cO_{\cX_t}$ denote the ideal sheaves of $Z \subset B$, $Y \subset X$, $\cY \cnec q^{-1}(Z \times V) \subset \cX$, and $\cY_t \subset \cX_t$ respectively for every $t \in V$. Since $q$ is flat, we have 
$$q^* \pr_1^*\cI_Z = \cI$$ 
where $\pr_1 : B \times V \to B$ is the first projection. Similarly,
$$g_{n-1}^*({\cI_Z}_{|Z_{n-1}}) = {\cI_Y}_{|Y_{n-1}}.$$ 
Let $\cY_i$ denote the $i$-th infinitesimal neighborhood of $\cY$ in $\cX$ for each $i \in \bZ_{\ge 0}$, then we have $q^{-1}(Z_i \times V) = \cY_i$. Therefore as $\cY_{n-1} \to Z_{n-1} \times V \to V$ is a family isomorphic to the constant family $Y_{n-1} \times V \to Z_{n-1} \times V \to V$ by the induction hypothesis, we have
\begin{equation}\label{egal-faisid}
\cI_{|\cY_{n-1}} = (q^* \pr_1^*\cI_Z )_{|\cY_{n-1}} \simeq \pi^*g_{n-1}^*({\cI_Z}_{|Z_{n-1}}) = \pi^*({\cI_Y}_{|Y_{n-1}})
\end{equation} 
where $\pi : Y_{n-1} \times V \to Y_{n-1}$ is the first projection.
Recall that for every morphism of complex spaces $T \to S$ and every sheaf $\cG$ over $T$, there is a natural (\eg functorial under pullback) one-to-one correspondence between the set of square-zero extensions of $T$ by $\cG$  over $S$ with $\Ext^1_{\cO_T}(L^\bullet_{T/S},\cG)$~\cite[Satz 3.16]{FlennerHab} where $L^\bullet_{T/S}$ is the  cotangent complex of $T$ over $S$. Let 
$$\cF \cnec \(\cI^n_Y / \cI^{n+1}_Y \)_{|Y_{n-1}}.$$
According to~\eqref{egal-faisid}, $\cY_n$ is a square-zero extension of $\cY_{n-1} \simeq Y_{n-1} \times V$ by $\pi^*\cF$ over $Z_n \times V$; let
$$\xi \in \Ext^1_{\cO_{Y_{n-1} \times V}}(L^\bullet_{Y_{n-1} \times V/ Z_n \times V},\pi^*\cF) \simeq \Ext^1_{\cO_{Y_{n-1} \times V}}(\pi^*L^\bullet_{Y_{n-1}/Z_n},\pi^*\cF) \simeq \Ext^1_{\cO_{Y_{n-1}}}(L^\bullet_{Y_{n-1}/Z_n},\cF) \otimes H^0(V,\cO_V)$$
be the corresponding element, which we regard as a holomorphic map
$$\xi : V \to \Ext^1_{\cO_{Y_{n-1}}}(L^\bullet_{Y_{n-1}/Z_n},\cF).$$
Again by~\eqref{egal-faisid}, for every $t \in V$, the fiber $\cY_{n,t}$ of $\cY_n \to V$ over $t$ is a square-zero extension of $\cY_{n-1,t} \simeq Y_{n-1}$ by $ \(\cI^n / \cI^{n+1} \)_{|\cY_{n-1,t}} \simeq \cF$ over $Z_n$ and by functoriality, this extension corresponds to the element $\xi(t) \in \Ext^1_{\cO_{Y_{n-1}}}(L^\bullet_{Y_{n-1}/Z_n},\cF)$. 

Since an isomorphism $\cX_t\simeq \cX_{t + \gl}$ over $B$ restricts to an isomorphism $\cY_{i,t} \simeq \cY_{i,t + \gl}$ over $Z_i$ for each $i \in \bZ_{\ge0}$, the fibers  ${\cY_{n,t}}$ and ${\cY_{n,t + \gl}}$ are isomorphic as square-zero extensions of $Y_{n-1}$ by $\cF$ over $Z_n$. It follows that $\xi(t + \gl) = \xi(t)$  for all $t \in V$ and $\gl \in \gL$, so $\xi$ descends to a holomorphic map $V/\gL \to \Ext^1_{\cO_{Y_{n-1}}}(L^\bullet_{Y_{n-1}/Z_n},\cF)$. As $V/\gL$ is a complex torus and $\Ext^1_{\cO_{Y_{n-1}}}(L^\bullet_{Y_{n-1}/Z_n},\cF)$ a complex vector space, $\xi$ is constant whose image represents the square-zero extension $Y_n$ of $Y_{n-1}$ by $\cF$ over $Z_n$. It follows that as square-zero extensions, $\cY_n = q^{-1}(Z_n \times V)$ is isomorphic to $Y_n \times V$ over $Z_n \times V$. So $g_n : Y_n \to Z_n$ is preserved by $\Pi$.

For the last statement, we prove again by induction on $n$ that the isomorphism $\cY_n \simeq Y_n \times V$ is in fact $G$-equivariant. The case $n = 0$ is covered by the assumption. Suppose that the statement is proven for $n-1$. Since we already know that 
 $(\cY_{n-1} \subset \cY_n)$ is isomorphic to $(Y_{n-1} \times V \subset Y_n \times V)$ (in the non-equivariant setting proven previously), the short exact sequences
\begin{equation}\label{eqn-ext0a}
\begin{tikzcd}[cramped, row sep = 20, column sep = 30]
0 \ar[r] & \(\cI^n / \cI^{n+1} \)_{|\cY_{n}}  \ar[r] &  \cO_{\cY_{n}}  \ar[r]  & \cO_{\cY_{n-1}}  \ar[r]  & 0 
\end{tikzcd}
\end{equation}
\begin{equation}\label{eqn-ext0b}
\begin{tikzcd}[cramped, row sep = 20, column sep = 30]
0 \ar[r] & \pi^*\(\cI_Y^n / \cI_Y^{n+1} \)_{|Y_n \times V}  \ar[r] &  \cO_{Y_n  \times V}  \ar[r]  & \cO_{Y_{n-1}  \times V}  \ar[r]  & 0. 
\end{tikzcd}
\end{equation}
are isomorphic. As $Y \subset X$ and $\cY \subset \cX$ are $G$-stable, their corresponding ideal sheaves $\cI_Y \subset \cO_X$ and $\cI \subset \cO_{\cX}$ are also $G$-stable. So the extensions~\eqref{eqn-ext0a} and~\eqref{eqn-ext0b} are $G$-equivariant for the $G$-actions induced by the $G$-actions on $\cX$ and $X \times V$. On the one hand, since the isomorphism $\cY_{n-1} \simeq Y_{n-1} \times V$ is $G$-equivariant by the induction hypothesis, the isomorphisms $\cO_{\cY_{n-1}} \simeq \cO_{Y_{n-1} \times V}$ and~\eqref{egal-faisid} (with $n-1$ replaced by $n$) are $G$-equivariant. So we can pullback the $G$-action on~\eqref{eqn-ext0b} to a $G$-action on~\eqref{eqn-ext0a}, which might induce a different $G$-structure on $\cO_{\cY_n}$. But on the other hand, since $G$ is a finite group, we have 
$$\Ext^1_G(\cO_{\cY_{n-1}}, (\cI^n / \cI^{n+1} )_{|\cY_{n}} ) = \Ext^1(\cO_{\cY_{n-1}}, (\cI^n / \cI^{n+1} )_{|\cY_{n}})^G ,$$ 
so the $G$-actions on $\cO_{\cY_{n-1}}$ and on $(\cI^n / \cI^{n+1})_{|\cY_{n}} $ determine uniquely the $G$-action on $\cO_{\cY_n}$. Therefore the two $G$-actions on~\eqref{eqn-ext0a} are in fact the same. In other words, $\cY_n \simeq Y_n \times V$ identifies the $G$-action on $\cY_n$ and the  $G$-action on $Y_n \times V$.
\end{proof}

We can apply Lemma~\ref{lem-prescompact} to the tautological families of twisted locally Weierstra{\ss} fibrations.

\begin{lem}\label{lem-presssfiblocWform}
In the setting of Proposition~\ref{pro-desingpre}, let $G$ be a finite group acting on $B$ and on $\bH$ in a compatible way. Assume that $f$ and $\ti{\imath}_0 : \ti{Z}_0 \to B$ are $G$-equivariant. Then the subfamily 
$$\Pi^G_{\ti{Z}_0} : \cW^G_{\ti{Z}_0} \xto{q'} B \times  V^G_{\ti{Z}_0} \to  V^G_{\ti{Z}_0} $$ 
of the tautological family associated to $f$ parameterized by $V^G_{\ti{Z}_0}$ preserves $G$-equivariantly the completion $\hat{Y} \to \hat{Z}$ of $f$ along $Y \cnec f^{-1}(Z) \to Z$.
\end{lem}

\begin{proof}
By Lemma~\ref{lem-latticeisom}, there exists a lattice $\gL \subset V_{\ti{Z}_0}$ such that $t$ and $t + \gl$ parameterize isomorphic elliptic fibrations in $\Pi^G_{\ti{Z}_0}$ for every $t \in V_{\ti{Z}_0}$ and $\gl \in \gL$. By Proposition~\ref{pro-desingpre}, the family $\Pi^G_{\ti{Z}_0}$ preserves $G$-equivariantly $Y\to Z$. As $q'$ is a locally Weierstra{ss} fibration, $q'$ is flat. Applying Lemma~\ref{lem-prescompact} to the family $\Pi^G_{\ti{Z}_0}$ and the fibration $Y \to Z$ contained in $f$ yields Lemma~\ref{lem-presssfiblocWform}.
\end{proof}

The next general result that we prove is the following, which will be used to pass from locally Weierstra{\ss} models to tautological models in the proof of Proposition~\ref{pro-locWeimodHdgiii}.

\begin{lem}\label{lem-finisurjtriv}
Let $f : X \to B$ and $g : Y \to B$ be two fibrations where $X$, $Y$, and $B$ are completions of  reduced complex spaces and assume that $Y$ has only finitely many irreducible components. Let $\bm : X \to Y$ be a finite surjective map over $B$. Let $\Pi : \cX \xto{q} B \times \gD \to \gD$ and $\Pi' : \cY \xto{q'} B \times \gD \to \gD$ be morphisms flat over a connected base $\gD$ (which could be regarded as deformations of $f$ and $g$ over $\gD$ respectively). 

Suppose that $\bm$ can be extended to a finite surjective map $\mu : \cX \to \cY$ over $B \times \gD$ and that for some open cover $\{U_i\}$ of $B$,  the restriction $\mu_i = \mu_{|\cX_i} :\cX_i \to \cY_i$  is isomorphic to $\bm_i \times \Id : X_i \times \gD \to Y_i \times \gD$ over $B \times \gD$ where $\cX_i  \cnec q^{-1}(U_i \times \gD)$, $\cY_i  \cnec q'^{-1}(U_i \times \gD)$, $X_i \cnec \cX_i \cap X$, and $Y_i \cnec \cY_i \cap Y$. Furthermore, assume that the restriction of the isomorphism $\mu_i \simeq \bm_i \times \Id$ to the central fiber $X_i \to Y_i$ is the identity. 

If  $\Pi'$ is a trivial deformation of $g$, then $\mu$ is isomorphic to $\bm \times \Id : X \times \gD \to Y \times \gD$ over $B \times \gD$. In particular, $\cX \to B \times \gD \to \gD$ is a trivial deformation. Moreover, let $G$ be a group acting on $X$, $Y$, and $B$ such that $f$, $g$, and $\bm$ are $G$-equivariant and assume that $\Pi$, $\Pi'$, and $\mu$ preserve the $G$-action. Assume that $\{U_i\}$ is $G$-invariant and the isomorphisms $\mu_i \simeq \bm_i \times \Id$ are $G$-equivariant as well. Then $\mu$ is $G$-equivariantly isomorphic to $\bm \times \Id$. 
\end{lem}

\begin{proof}
It suffices to show that for every $i$ and $j$, the isomorphisms $\phi_i : X_i \times \gD \xto{\sim} \cX_i$ and $\phi_j :   X_j \times \gD  \xto{\sim} \cX_j$ induced by $\mu_i \simeq \bm_i \times \Id$ agree on the intersection $X_{ij} \times \gD = (X_i \cap X_j) \times \gD$. Let $Y_{ij} = Y_i \cap Y_j$. Consider the map $\gD \to \Aut(X_{ij}/Y_{ij})$ which  associates $t \in \gD$ to the unique automorphism $\psi_t : X_{ij} \to X_{ij}$ satisfying $\phi_i(x,t) = \phi_j(\psi_t(x),t)$. As $\bm : X \to Y$ is finite and $X_{ij}$ is reduced,  $\Aut(X_{ij}/Y_{ij})$ is finite by the following lemma.
\begin{lem}\label{lem-finiredaut}
Let $f : S \to T$ be a finite morphism of formal complex spaces. Assume that $S$ is a completion of a reduced complex space and $T$ has only finitely many irreducible components, then $\Aut(S/T)$ is finite.
\end{lem}

Accordingly, $t \mapsto \psi_t$ is constant. Since the restriction of the isomorphism $\mu_i \simeq \bm_i \times \Id$ to the central fiber $X_i \to Y_i$ is the identity, we have $\psi_t = \psi_o = \Id$. Hence $\phi_i$ and $\phi_j$ agree on $X_{ij} \times \gD$.
\end{proof}

\begin{proof}[Proof of Lemma~\ref{lem-finiredaut}]
First we prove the following.
\begin{claim}
$S$ is a reduced formal complex space.
\end{claim}
\begin{proof}
Assume that $S$ is the completion of the reduced complex space $\gS$ along a complex subspace $Z \subset \gS$, defined by the ideal sheaf $\cI$. It suffices to show that for any small open subset $U \subset \gS$, the $\cI(U)$-adic completion $\wh{\cO_{\gS}(U)}$ of the ring $\cO_{\gS}(U)$ is reduced. 

Let $\nu: \ti{\gS} \to \gS$ be a desingularization of $\gS$ and let $\ti{Z} \cnec \nu^{-1}(Z)$. Let $\cJ \subset \cO_{\ti{\gS}}$ be the ideal sheaf of $\ti{Z} \subset \ti{\gS}$. As $\nu$ is surjective and $\gS$ is reduced, the natural map $\cO_{\gS} \to \nu_*\cO_{\ti{\gS}}$ is injective. Consider the ideal sheaf $\cI' \cnec (\nu_*\cJ) \cap \cO_{\gS} \subset \cO_\gS$. Since $Z$ and the complex subspace defined by $\cI'$ have the same underlying reduced complex subspace, $\wh{\cO_{\gS}(U)}$ is also the $\cI'(U)$-adic completion of $\cO_{\gS}(U)$.

Let $\ti{U} = \nu^{-1}(U)$. As $\cO_{\gS}(U)/\cI'(U)^k \hto \cO_{\ti{\gS}}(\ti{U})/\cJ(\ti{U})^k$ for every $k \in \bZ_{>0}$, we have $\wh{\cO_{\gS}(U)} \hto \wh{\cO_{\ti{\gS}}(\ti{U})}$ where $\wh{\cO_{\ti{\gS}}(\ti{U})}$ denotes the $\cJ(\ti{U})$-adic completion of ${\cO_{\ti{\gS}}(\ti{U})}$. Since $\ti{\gS}$ is smooth, $\wh{\cO_{\ti{\gS}}(\ti{U})}$ is a regular ring~\cite[2.1]{BingenerForC}. In particular, $\wh{\cO_{\ti{\gS}}(\ti{U})}$ is reduced, hence $\wh{\cO_{\gS}(U)}$ is reduced as well.
\end{proof}

Let $S = \bigcup_{i,j} S_{ij}$ be the decomposition of $S$ into its irreducible components such that $f(S_{ij}) = f(S_{i'j'})$ if and only if $i = i'$. Since 
$$ \Aut(S/T) \subset \prod_{i} \Aut(\cup_j S_{ij}/f(S_{ij})),$$ we can assume that $T$ is irreducible and  the image of each irreducible component of $S$ is $T$. By~\cite[(3.5)]{BingenerForC}, the map $(f: S \to T) \mapsto f_*\cO_S$
defines an equivalence of categories between the category of finite formal complex spaces over $T$ and the category of coherent $\cO_T$-algebras. By the above claim, $(f_*\cO_S)_t$ is a reduced $\cO_{T,t}$-algebra for every $t \in T$. Moreover, $(f_*\cO_S)_t$ is finitely generated as an $\cO_{T,t}$-module.  It follows from the following claim that the automorphism group $\Aut_{\cO_{T,t}}((f_*\cO_S)_t)$ of the $\cO_{T,t}$-algebra $(f_*\cO_S)_t$ is finite. 

\begin{claim}
Let $A \to B$ be a finite ring extension such that $B$ is reduced, $\Spec(A)$ is irreducible, and each irreducible component $\Spec(B_i)$ of $\Spec(B)$ surjects onto $A$. Then the automorphism group $\Aut_{A}(B)$ of the $A$-algebra $B$ is finite.
\end{claim}

\begin{proof}
Note that since $B$ is reduced and  $\Spec(A)$ is irreducible, $A$ is an integral domain. Let $\{ \Spec(B_i)\}_{i \in I}$ be the set of irreducible components of $\Spec(B)$, which is finite because $B$ is finite over $A$. As $B$ is reduced, each $B_i$ is an integral domain. The induced maps $A \to B_i$ are finite ring extensions for all $i \in I$. If $K$ (resp. $L_i$) denotes the fractional field of $A$ (resp. $B_i$) and $L \cnec \prod_i L_i$, then each element of $\Aut_{A}(B)$ extends uniquely to an automorphism of $K$-algebra $ L \to L$. Therefore we have an injective map $\Aut_{A}(B) \hto \Aut_{K}(L)$ and it suffices to show that $\Aut_{K}(L)$ is finite. Note that $\Aut_{K}(L)$ is in bijection with $G \times \prod_{i}\Aut_{K}(L_i)$ where $G$ is the set of bijections $\gs : I \to I$ such that $L_i$ and $L_{\gs(i)}$ are isomorphic as field extensions of $K$ for all $i \in I$. As $I$ is finite, $G$ is also finite. Finally, since $L_i$ is a finite field extension of $K$, $\Aut_{K}(L_i)$ is finite as well for each $i \in I$.  Hence  $\Aut_{K}(L)$ is finite.
\end{proof}

Therefore, it suffices to show that the group homomorphism
$$\Psi_t : \Aut_{\cO_T}(f_*\cO_S) \to \Aut_{\cO_{T,t}}((f_*\cO_S)_t)$$
defined by the restriction is injective. To this end, let $g : S \to S$ be an automorphism over $T$ such that $\Psi_t(g) = \Id$. The maps $g$ and $\Id_S : S \to S$ give rise to two elements 
$$g, \Id_S \in \Hom_{\cO_T}(f_*\cO_S, f_*\cO_S) \simeq \Hom_{\cO_S}(f^*f_*\cO_S, \cO_S)$$
where $f_*\cO_S$ is considered as a \emph{coherent sheaf over $\cO_T$} (instead of a coherent $\cO_T$-algebra). Let $$\cI \cnec (g - \Id_S)(f^*f_*\cO_S) \subset \cO_S.$$ 
As $\Psi_t(g) = \Id$, there exists a neighborhood $U \subset T$ of $t$ such that $g_{|f^{-1}(U)}$ is the identity. So $\cI_{| f^{-1}(U)} = 0$. For every irreducible component $S_i$ of $S$, since $f(S_i) = T$ by assumption, the intersection $S_i \cap f^{-1}(U)$ is a nonempty open subset of $S_i$. Because $\cI$ is an ideal sheaf and $S_i$ is reduced, it follows from $\cI_{| S_i \cap f^{-1}(U)} = 0$ that $\cI_{| S_i} = 0$. So $\cI = 0$, hence  $g = \Id_S$.
\end{proof}

\begin{proof}[Proof of Proposition~\ref{pro-locWeimodHdgiii}]
 Let $Y \cnec f^{-1}(Z)$. It suffices to prove that 
$$\hat{\cY} \cnec q^{-1}(\hat{Z} \times  V_{\ti{Z}_0}^G) \to \hat{Z} \times  V_{\ti{Z}_0}^G$$ 
 is $G$-equivariantly isomorphic to the product $\hat{Y} \times  V_{\ti{Z}_0}^G \to \hat{Z} \times  V_{\ti{Z}_0}^G$ of $\hat{Y} \to \hat{Z}$ with the identity $\Id : V_{\ti{Z}_0}^G \to V_{\ti{Z}_0}^G$. Assume that $\sX = \sX\(m; (\eta_{ij}),(\eta^g_{i});\eta'_G\)$ where $\left\{(\eta_{ij}),(\eta^g_{i})\right\}$ is a 1-cocycle representing $\eta_G$ and $m \in \bZ \bss\{0\}$ and $\eta'_G \in H^1_G(B,\cJ^W_{\bH/B})$ are such that
 $m\eta_G$ can be lifted to $\eta'_G$. We will use the construction of the tautological family described in the proof of Proposition-Definition~\ref{pro-locWeimodHdg} and the notations therein.

 Let $\Psi : \hat{Z} \times V_{\ti{Z}_0}^G \hto B \times V$ be the product of $\hat{Z} \hto B$ with  $V_{\ti{Z}_0}^G \hto V$ and let
\begin{equation}\label{comm-recollement}
\begin{tikzcd}[cramped, row sep = 20, column sep = 60]
\hat{\sX}_{Z,ij} \times V_{\ti{Z}_0}^G \arrow[r, "H'_{ij} "] \ar[d, swap, "\mu_{i} \times \Id"] & \hat{\sX}_{Z,ji} \times V_{\ti{Z}_0}^G \ar[d, "\mu_{j} \times \Id "] \\ 
\hat{W}_{Z,ij} \times V_{\ti{Z}_0}^G \ar[r, "\gT_{ij}'"] &  \hat{W}_{Z,ji} \times V_{\ti{Z}_0}^G  
\end{tikzcd}
\end{equation}
be the pullback of the left side of~\eqref{diag-recollementW'} by $\Psi$.
Let  $\hat{\sX}_{Z,i} = \sX_i \times_B \hat{Z}$ and $\hat{W}_{Z,i} = W_i \times_B \hat{Z}$. Let 
$$\cW  \xto{q'} B \times V \to V$$ 
be the tautological family associated to the locally Weierstra{\ss} fibration $f' :  W' \cnec W^{\eta'}  \to B$ twisted by $\eta'$ where we recall that $\eta'$ is the image of $\eta'_G$ in $H^1(B,\cJ^W_{\bH/B})$. By Lemma~\ref{lem-presssfiblocWform}, the complex space obtained by gluing the $\hat{W}_{Z,i} \times V_{\ti{Z}_0}^G$ using $\gT_{ij}'$, which is $\hat{\cW}'_Z \cnec q'^{-1}(\hat{Z} \times V_{\ti{Z}_0}^G)$ by construction, is $G$-equivariantly isomorphic to $\hat{W}'_Z \times V_{\ti{Z}_0}^G$ over $\hat{Z} \times V_{\ti{Z}_0}^G$, where $\hat{W}'_Z \to \hat{Z}$ is the completion of $f' : W' \to B$ along $f'^{-1}(Z) \to Z$. Since $\hat{\sX}_Z \cnec \sX \times_B \hat{Z}$ is a completion of $\sX$ which is reduced and the finite map $\bm_{|\hat{\cY}} : \hat{\cY} \to \hat{\cW}'_Z$ extending $\bm_{|\hat{\sX}_Z} : \hat{\sX}_Z \to \hat{W}'_Z$ is obtained by gluing the 
$$\mu_{i} \times \Id : \hat{\sX}_{Z,i} \times V_{\ti{Z}_0}^G \to \hat{W}_{Z,i} \times V_{\ti{Z}_0}^G$$ 
using~\eqref{comm-recollement}, by Lemma~\ref{lem-finisurjtriv}  $\hat{\cY}  \to \hat{Z} \times  V_{\ti{Z}_0}^G$ 
 is $G$-equivariantly isomorphic to $\hat{Y} \times  V_{\ti{Z}_0}^G \to \hat{Z} \times  V_{\ti{Z}_0}^G$.
\end{proof}

\begin{rem}
In both Proposition~\ref{pro-desingpre} and Proposition~\ref{pro-locWeimodHdgiii}, the hypothesis that the total space of $f$ is in the Fujiki class $\cC$ is unnecessary.
\end{rem}

Combining Proposition~\ref{pro-densepair0} and Proposition~\ref{pro-locWeimodHdgiii}, now we can prove Proposition~\ref{pro-appalgpres}.

\begin{proof}[Proof of Proposition~\ref{pro-appalgpres}]
Let $Z_0 = \ol{Z\bss \gD}$ and $\ti{\imath}_0 : \ti{Z}_0 \to Z_0 \hto B$ be a $G$-equivariant log-desingularization of $(Z_0, Z_0 \cap \gD)$. Then the subfamily $\Pi'$ of the tautological family associated to $f$ parameterized by 
 $$V' = \ker\(\ti{\imath}_0^* : H^1(B, \cL_{\bH/B}) \to H^1(\ti{Z}_0,\ti{\imath}_0^*\cL_{\bH/B})\)^G$$
  is a deformation of $f$ which is $G$-equivariantly locally trivial over $B$ (Proposition-Definition~\ref{pro-locWeimodHdg}\comm{.}\ref{pro-locWeimodHdgii}) and preserves $G$-equivariantly $\rwh{f^{-1}(Z)} \to \hat{Z}$ (Proposition~\ref{pro-locWeimodHdgiii}). 

By Proposition-Definition~\ref{pro-locWeimodHdg}, there exists a multiplication-by-$m$ map $\bm : \cX \to \cW$ over $B \times V$ where $V \cnec H^1(B, \cL_{\bH/B})$ and $\cX \to B \times V \to V$ (resp. $\Pi : \cW \to B \times V \to V$) denotes the tautological family associated to $f$ (resp. the minimal locally Weierstra{\ss} fibration $p^\eta$ twisted by some element $\eta \in H^1(B,\cJ^W_{\bH/B})$). Since $f^{-1}(Z_0) \to Z_0$ has a multi-section, $(p^{\eta})^{-1}(Z_0) \to Z_0$ has a multi-section as well. By Proposition~\ref{pro-densepair0}, the subfamily of  $\Pi$ parameterized by $V'$ is an algebraic approximation of $W^\eta$. Since $\bm$ is finite and surjective, it follows that $\Pi'$ is an algebraic approximation of $\sX$.
 \end{proof}

\section{Algebraic approximations of elliptic fibrations over a projective variety}\label{sec-preuve}

We will prove Theorem~\ref{thm-mainfibellip} in this section. The following proposition provides a general approach that we will follow to proving the existence of algebraic approximations.
\begin{pro}\label{pro-Formalgapp}
Let $X$ be a compact complex variety and $\nu : X \dto X'$ a bimeromorphic map. Suppose that there exist a subvariety $Y' \subset X'$  and an algebraic approximation $\cX' \to \gD$ of $X'$ preserving the formal completion $\hat{Y}'$ of $X'$ along $Y'$ (cf. Definition~\ref{def-preserver}) such that $\nu^{-1}_{|X' \bss Y'}$ is biholomorphic onto its image. Then $X$ has an algebraic approximation.
\end{pro}

\begin{proof}[Proof of Proposition~\ref{pro-Formalgapp}]
Let
\begin{equation}\label{res-nu}
\begin{tikzcd}[cramped, row sep = 20, column sep = 25]
X    & \ti{X} \ar[l, swap, "\mu' "] \ar[r, "\nu' "]     & X'
\end{tikzcd}
\end{equation}
be a minimal resolution of $\nu$ (in the sense that  if $U \cnec \nu^{-1}(X' \bss Y')$, then $\mu'_{|{\mu'}^{-1}(U)}$ is biholomorphic onto $U$). Let $Y \cnec  X \bss \(\nu^{-1}(X' \bss Y')\)$, namely $Y$ is the subvariety of $X$ such that the restriction of $\nu$ to $X \bss Y$ is biholomorphic onto $X' \bss Y'$. We have 
$$E \cnec {\mu'}^{-1}(Y) =  {\nu'}^{-1}(Y').$$ 
If $\hat{E}$ (resp. $\hat{Y}$ and $\hat{Y}'$) denotes the completion of $\ti{X}$ (resp. $X$ and $X'$) along $E$ (resp. $Y$ and $Y'$), then $\hat{E} \to \hat{Y}$ and $\hat{E} \to \hat{Y}'$ are formal modifications. So their products $\hat{E} \times \gD \to \hat{Y} \times \gD $ and $\hat{E} \times \gD \to \hat{Y}' \times \gD $ with the identity $\Id : \gD \to \gD$ are also formal modifications (Proposition~\ref{pro-modforex}). By assumption, there exists an algebraic approximation $\cY' \subset \cX' \xto{\pi'} \gD$ of the pair $(Y' \subset X')$ such that the completion $\hat{\cY}'$ of $\cX'$ along $\cY'$ is isomorphic to $\hat{Y}' \times \gD$ over $\gD$. Therefore by Theorem~\ref{thm-ATB}.\emph{ii)}, there exists a modification $\ti{\nu} : \ti{\cX} \to \cX'$ of $\cX'$ along $\cY'$ such that if $\cE \cnec \ti{\nu}^{-1}(\cY')$, then the completion of $\ti{\nu}$ along $\cE \to \cY'$ is isomorphic to $\hat{E} \times \gD \to \hat{Y}' \times \gD$. Similarly by Theorem~\ref{thm-ATB}.\emph{i)}, there exists a modification $\ti{\mu} : \ti{\cX} \to \cX$ along a subvariety $\cY \subset \cX$ such that $\cE = \ti{\mu}^{-1}(\cY)$ and the completion of $\ti{\mu}$ along $\cE \to \cY$ is isomorphic to $\hat{E} \times \gD \to \hat{Y} \times \gD$. Let $\ti{\pi} \cnec \pi' \circ \ti{\nu} : \ti{\cX} \to \cX' \to \gD$. As $\cE \to \cY$, being isomorphic to $E \times \gD \to Y \times \gD$, is a morphism over $\gD$, the map $\ti{\pi}$ induces a map $\pi : \cX \to \gD$. To sum up, we have a diagram
\begin{equation}
\begin{tikzcd}[cramped, row sep = 30, column sep = 30]
{\cX}  \ar[rd, swap, "\pi"]   & \ti{\cX} \ar[l, swap, "\ti{\mu} "] \ar[d, "\ti{\pi} "] \ar[r, "\ti{\nu} "]     & \cX' \ar[ld, "\pi'"] \\
& \gD &
\end{tikzcd}
\end{equation}
of morphisms over $\gD$ containing~\eqref{res-nu} as a fiber over the point $o \in \gD$ parameterizing $X'$ in $\pi' : \cX' \to \gD$.

Now we show that $\ti{\pi}$ and $\pi$ are deformations of $\ti{X}$ and $X$ respectively. First we show that $\ti{\pi}$ and $\pi$ are flat. On the one hand since $\hat{\cE} \to \gD$ is isomorphic to the projection $\hat{E} \times \gD \to \gD$, which is flat, the composition $\ti{\pi} : \ti{\cX} \to \cX \to \gD$ is flat at every point of $\cE \subset \ti{\cX}$  by the infinitesimal criterion of flatness. On the other hand as the restriction $\mu_{|\ti{\cX}\bss \cE} : \ti{\cX}\bss \cE \to {\cX}\bss \cY$ is an isomorphism and ${\cX}\bss \cY \to \gD$ is flat, the restriction of $\ti{\pi}$ to the Zariski open $\ti{\cX}\bss \cE$ is also flat. Hence $\ti{\pi}$ is flat everywhere. A similar argument shows that $\pi$ is flat as well.

Let $\ti{\cX}_o$ (resp. ${\cX}_o$) denote the central fiber of $\ti{\pi}$ (resp. $\pi$). The restriction $\ti{\cX}_o \to X'$ of $\ti{\nu} : \ti{\cX} \to \cX'$ to $\ti{\cX}_o$ is a modification of $X'$, whose completion along the exceptional locus is isomorphic to $\hat{E} \to \hat{Y}'$. So by the uniqueness of extensions of formal modifications (Theorem~\ref{thm-ATB}), the restriction of $\ti{\nu}$ to $\ti{\cX}_o$ is isomorphic to  $\nu : \ti{X} \to X'$. Thus $\ti{\pi}' : \ti{\cX}' \to \gD$ is a deformation of $\ti{X}'$. A similar argument shows that the restriction of $\ti{\mu}$ to $\ti{\cX}_o$ is isomorphic to  $\mu : \ti{X} \to X$. Thus $\pi : \cX \to \gD$ is a deformation of $X$. Finally since $\pi' : \cX' \to \gD$ is an algebraic approximation of $X'$, we conclude that $\pi$ is an algebraic approximation of $X$.
\end{proof}

\begin{proof}[Proof of Theorem~\ref{thm-mainfibellip}]

Let $X$ be a compact K\"ahler manifold bimeromorphic to the total space of an elliptic fibration over a projective variety. We will first construct bimeromorphic maps 
$$
\begin{tikzcd}[cramped, row sep = 20, column sep = 25]
X    & X'/G \ar[l, swap, "\mu"] \ar[r, dashed, "\bar{\tau}"]     & \sX/G
\end{tikzcd}
$$
such that $\bar{\tau}$ satisfies the hypothesis of Proposition~\ref{pro-Formalgapp}.

\begin{lem}\label{lem-red}
There exist a compact complex manifold $X'$, a finite group $G$ acting on it, a $G$-equivariant tautological model $f : \sX \to B$ satisfying Hypotheses~\ref{hyp-KahlerSection}, and bimeromorphic maps
$$
\begin{tikzcd}[cramped, row sep = 20, column sep = 25]
X    & X'/G \ar[l, swap, "\mu"] \ar[r, dashed, "\bar{\tau}"]     & \sX/G
\end{tikzcd}
$$
such that $\mu$ is holomorphic and $\bar{\tau}$ is the quotient of a $G$-equivariant bimeromorphic map $\tau : X' \dto \sX$ satisfying the following property: There exists a $G$-stable subvariety $Z \subset B$ containing $\gD$ such that $f^{-1}(\ol{Z \bss \gD}) \to \ol{Z \bss \gD}$ has a multi-section and $\tau^{-1}_{|\sX \bss f^{-1}(Z)}$ is biholomorphic onto its image.
\end{lem}

\begin{proof}
The proof of Lemma~\ref{lem-red} consists of several steps.

\noindent \textbf{Step 1:} \textit{$X$ is bimeromorphic to $X_0/G$ where $G$ is a finite group and $X_0$ is the total space of a $G$-equivariant tautological model $f_0 : X_0 \to B_0$ satisfying Hypotheses~\ref{hyp-KahlerSection}.
}

Let $\bar{g} : \bar{X}' \to \bar{B}_0$ be an elliptic fibration over a projective variety $\bar{B}_0$ such that $\bar{X}'$ is a compact K\"ahler manifold bimeromorphic to $X$. Let $\bar{\gD}_0 \subset \bar{B}_0$ be the discriminant locus of $\bar{g}$. Up to base changing $\bar{g}$ with a log-desingularization $\bar{B}'_0 \to \bar{B}_0$ of $(\bar{B}_0, \bar{\gD}_0)$ and taking a desingularization of $\bar{X}' \times_{\bar{B}_0} \bar{B}'_0 $, we can assume that $\bar{B}_0$ is smooth and $\bar{\gD}_0$ is a simple normal crossing (SNC) divisor. 

Since $\bar{X}'$ is a compact K\"ahler manifold and $\bar{B}_0$ is smooth, $\bar{g}$ is locally projective~\cite[Theorem 3.3.3]{NakayamaLoc}. So there exists a Galois cover $r_0 : B_0 \to \bar{B}_0$ such that the base change $g : X' \cnec \bar{X}' \times_{\bar{B}_0} {B}_0 \to {B}_0$ of $\bar{g}$ is an elliptic fibration over a projective manifold $B_0$ whose  discriminant locus $r_0^{-1}(\bar{\gD}_0)$ is a simple normal crossing divisor and $g$ has local meromorphic sections at every point of $B_0$~\cite[Proposition 3.11]{ClaudonHorpi1}. Let $\gD_0  \subset r_0^{-1}(\bar{\gD}_0) \subset B_0$ be the discriminant divisor of the minimal Weierstra{\ss} fibration associated to $g$. As $g$ is locally projective, the local monodromies of the local system $\bH_0 \cnec (R^1g_*\bZ)_{|{B}_0 \bss {\gD}_0}$ around ${\gD}_0$ are quasi-unipotent. By Kawamata's unipotent reduction~\cite[Corollary 18]{KawamataAb}, we can further assume that the Galois cover $r_0 : {B}_0 \to \bar{B}_0$ we take is such that the local monodromies of $\bH_0$ around $\gD_0$ are unipotent and $({B}_0,{\gD}_0)$ remains a log-smooth projective pair. 

Since $g  : X' \to B_0$ has local meromorphic sections at every point of $B_0$, we can apply Lemma-Definition~\ref{lem-locWeimodHdg} and obtain a $G$-equivariant tautological model $f_0 : {X}_0 \to  {B}_0$ of  $g$ where $G \cnec \Gal (B_0 / \bar{B}_0)$. So ${X}_0/G$ is bimeromorphic to $X'/G$, which is further bimeromorphic to $X$. As the discriminant locus of $f_0 : {X}_0 \to  {B}_0$ is $\gD_0$ by Lemma~\ref{lem-dicstaut} and $({B}_0,{\gD}_0)$ is a log-smooth projective pair, $f_0$ satisfies $i)$ and $iii)$ in Hypotheses~\ref{hyp-KahlerSection}. Since $f_0$ is bimeromorphic to $g$ and $g$ satisfies $ii)$ and $iv)$ in Hypotheses~\ref{hyp-KahlerSection}, so does $f_0$. \qed

\noindent \textbf{Step 2:} \textit{There exist a sequence of blow-ups  
\begin{equation}\label{suite-eclatementX0}
 {X}_n \to \cdots \to {X}_0
 \end{equation}
along $G$-stable subvarieties $C_i \subset X_i$ (hence all maps $X_{i+1} \to X_i$ are $G$-equivariant) which do not dominate $B_0$ and a bimeromorphic morphism $\mu_n : X_n/G \to X$.}

By Step 1, we have a bimeromorphic map $\phi : \bar{X}_0 \cnec X_0/G \dto X$. Let
\begin{equation}\label{suite-eclatementX0quot}
 \bar{X}_n \to \cdots \to \bar{X}_0
 \end{equation}
 be a sequence of blow-ups and $\phi' : \bar{X}_n \to X$ a (bimeromorphic) morphism resolving $\phi$. As ${X}_0$ is normal by Lemma~\ref{lem-dicstaut}, so is the finite quotient $\bar{X}_0$. Thus the indeterminacy locus of $\phi$ is of codimension at least $2$. As $\dim \bar{B}_0 =  \dim \bar{X}_0 - 1$, we may assume that the blow-up centers $\bar{C}_i \subset \bar{X}_i$ of~\eqref{suite-eclatementX0quot} do not dominate $\bar{B}_0$.

We construct by induction the blow-up sequence~\eqref{suite-eclatementX0}  together with $G$-invariant morphisms $q_i: X_i \to \bar{X}_i$ such that the quotients $\bar{q}_i : X_i/G \to \bar{X}_i$ are bimeromorphic. For $i= 0$, we define $q_0 : X_0 \to \bar{X}_0$ to be the quotient map, so $\bar{q}_0$ is the identity. Suppose that $X_i$ and $q_i$ are constructed and $\bar{q}_i$ is known to be bimeromorphic. Let $C_i \cnec q_i^{-1}(\bar{C}_i)$. We define $X_{i + 1} \to X_i$ to be the blow-up of $X_i$ along $C_i$. As $q_i$ is $G$-equivariant, $C_i$ is $G$-stable. Since $\bar{C}_i$ does not dominate $\bar{B}_0$, ${C}_i$ does not dominate ${B}_0$. The universal property of blowing up provides a $G$-invariant map $q_{i+1} : X_{i+1} \to \bar{X}_{i+1}$ and a commutative diagram
\begin{equation}\label{diag-eclatcomm}
\begin{tikzcd}[cramped, row sep = 20, column sep = 25]
X_{i+1} \ar[r] \ar[d,"q_{i+1}"]  & X_i  \ar[d, "q_i"] \\
  \bar{X}_{i+1} \ar[r]  &\bar{X}_{i}
\end{tikzcd}
\end{equation}
where the horizontal arrows are the blow-ups along $C_i$ and $\bar{C}_i$ respectively. As $\bar{q}_i$ is bimeromorphic, $\bar{q}_{i+1}$ is also bimeromorphic. It follows that  the map $\mu_n$ defined by
$$\mu_n \cnec \phi' \circ \bar{q}_n : X_n/G \to \bar{X}_n \to X$$
 is  bimeromorphic as well. \qed

\noindent \textbf{Step 3:} \textit{We construct by induction starting from $i = 0$ a sequence of $G$-equivariant blow-ups
$$\nu :  {B}_n \xto{\nu_{n-1}} \cdots \xto{\nu_0} {B}_0$$
 and $G$-stable subvarieties $Z_i \subset B_i$ containing the pre-images $\gD_i \subset B_i$ of $\gD_0 \subset B_0$ and satisfying the following properties. 
\begin{enumerate}[label = \roman{enumi})]
\item The induced  $G$-equivariant meromorphic map  $f_i : X_i \dto B_i$ is almost holomorphic and well-defined over $B_i \bss Z_i$, which is nonempty.
\item Let $f'_i : X'_i \cnec X_0 \times_{B_0} B_i \to B_i$ be the base change of the fibration $f_0 : X_0 \to B_0$. Then $f'_i$ is a smooth elliptic fibration over $B_i \bss Z_i$ and ${f'_i}^{-1}(\ol{Z_i \bss \gD_i}) \to \ol{Z_i \bss \gD_i}$ has a multi-section.
\item  The map $f_i$ is $G$-equivariantly biholomorphic to $f'_i$ over $B_i \bss Z_i$. 
\end{enumerate}
}

For $i = 0$, we set $Z_0 = \gD_0$ and the three properties above hold obviously. Assume that $B_{i}$ and $Z_i$ are constructed and satisfy the properties listed above. Let $C'_i$ be the union of the irreducible components of the blow-up center $C_i \subset  X_i$ of $X_{i+1} \to X_i$ that are of the form $\ol{f_i^{-1}(\gS \bss Z_i)}$ for some subvariety $\gS \subset B_i$; we write $C_i = C'_i \cup C''_i$ where $C''_i$ is the union of the irreducible components of $C_i$ complementary to $C'_i$. We define $\nu_{i} : B_{i+1} \to B_{i}$ to be the blow-up of $B_i$ along $f_i(C'_i)$\footnote{For any subvariety $Y \subset X_i$, $f_i(Y)$ is defined as $\ol{f_i\(Y \cap {f_i}^{-1}(B_i \bss Z_i)\)}$, which is well-defined since $f_i$ is holomorphic over $B_i \bss Z_i$. In particular $ f_i(Y)$ is empty if and only if $Y \subset {f_i}^{-1}(Z_i)$.} and
$$Z_{i+1} \cnec \nu_i^{-1}(Z_{i} \cup f_i(C''_i)).$$
Obviously, $Z_{i+1}$ contains $\gD_{i+1}$.
As $C_i$ and $Z_i$ are $G$-stable, $C'_i$ and $C''_i$ are $G$-stable as well. Thus the blow-up center $f_i(C'_i)$ of $\nu_i$ and $Z_{i+1}$ are also $G$-stable.  
 
As $f_i$ is biholomorphic to $f'_i$ over $B_i \bss Z_{i} $ and $f'_i$ is a smooth elliptic fibration over $B_i \bss Z_{i}$, the map $f_i : X_i \dto B_i$ is flat over $B_i \bss (Z_{i} \cup f_i(C''_i))$. Since blowing up commutes with flat base change, we have the following cartesian diagram
\begin{equation}\label{diag-eclatplatcart}
\begin{tikzcd}[cramped, row sep = 20, column sep = 25]
X_{i+1}^\st \cnec \wt{ {f_i}^{-1}\(B_i \bss  (Z_{i} \cup f_i(C''_i))\) } \ar[r] \ar[d]  &  {f_i}^{-1}\(B_i \bss  (Z_{i} \cup f_i(C''_i))\)  \ar[d, "f_i"] \\
  B_{i+1} \bss Z_{i+1} \ar[r, "\nu_{i}"]  & B_i \bss  (Z_{i} \cup f_i(C''_i))
\end{tikzcd}
\end{equation}
where the upper horizontal arrow is the blow-up of the Zariski open ${f_i}^{-1}\(B_i \bss  (Z_{i} \cup f_i(C''_i))\)\subset X_i$ along 
\begin{equation}\label{eq-centreeclat}
{f_i}^{-1}\({f_i}(C'_i) \bss  (Z_{i} \cup f_i(C''_i))\) 
= C_i' \bss {f_i}^{-1}(Z_{i} \cup f_i(C''_i)) =  C_i \bss {f_i}^{-1}(Z_{i} \cup f_i(C''_i)) =
 C_i \cap {f_i}^{-1}\(B_i \bss (Z_i \cup f_i(C''_i))\).
\end{equation}
Here, the first equality of~\eqref{eq-centreeclat} follows from the assumption that irreducible components of $C'_i$ are of the form $\ol{f_i^{-1}(\gS \bss Z_i)}$.
Recall that $X_{i+1}$ is the blow-up of $X_i$ along $C_i$, so $X_{i+1}^\st$ is a Zariski open of $X_{i+1}$ and the left vertical arrow is the restriction of $f_{i+1}$ to $X_{i+1}^\st$. As $X_{i+1}^\st \to B_{i+1} \bss Z_{i+1}$ is proper, $f_{i+1}$ is almost holomorphic and well-defined over $B_{i+1} \bss Z_{i+1}$. As $C_i \subset X_i$  does not dominate $B_i$ by Step 2, we have $B_{i+1} \bss Z_{i+1} \ne \emptyset$, which shows $i)$. Since~\eqref{diag-eclatplatcart} is cartesian and by the induction hypothesis, the right vertical arrow of~\eqref{diag-eclatplatcart} is $G$-equivariantly isomorphic to the base change
$$X_0 \times_{B_0} \(B_i \bss (Z_i \cup f_i(C''_i))\) \to B_i \bss (Z_i \cup f_i(C''_i))$$
of $f_0 : X_0 \to B_0$, $iii)$ holds.

As for $ii)$, since $f_0$ is smooth over $B_0 \bss Z_0$ and $Z_{i+1}$ contains the pre-image of $Z_0$, the elliptic fibration $f'_{i+1}$ is smooth over $B_{i+1} \bss Z_{i+1}$. Since $Z_{i+1} = \nu_i^{-1}(Z_{i} \cup f_i(C''_i))$ and $f'_{i+1}$ is the base change of $f'_i$ by $\nu_i : B_{i+1} \to B_i$, it suffices to show that $f'^{-1}_{i}(\ol{Z_i \bss \gD_i}) \to \ol{Z_i \bss \gD_i}$ and $f'^{-1}_{i}(f_i(C''_i)) \to  f_i(C''_i)$ have multi-sections. The former holds by the induction hypothesis. As for the latter, we have to show that for every irreducible component $Y''$ of $C''_i$, if $Z'' \cnec f_i(Y'')$ (which we can assume to be nonempty, namely $Z'' \nsubset Z_i$), then the fibration $f'^{-1}_{i}(Z'') \to Z''$ has a multi-section. Over the dense Zariski open $Z'' \bss Z_i$ of $Z''$, the map ${f'_{i}}^{-1}(Z'') \to Z''$ is isomorphic to the almost holomorphic map $f^{-1}_{i}(Z'') \cnec \ol{f^{-1}_{i}(Z'' \bss Z_i)} \dto Z''$. In particular, fibers of  $f^{-1}_{i}(Z'') \dto Z''$ over $Z'' \bss Z_i$ are elliptic curves. So if $Y''$ is not a multi-section of $f^{-1}_{i}(Z'') \dto Z''$, then $Y'' = \ol{f^{-1}_{i}(Z'' \bss Z_i)}$, which would be an irreducible component of $C'_i$ and yield a contradiction. Accordingly, the image of $Y''$ in ${f'_{i}}^{-1}(Z'')$ is a multi-section of ${f'_{i}}^{-1}(Z'') \to Z''$. \qed

\noindent \textbf{Step 4:} \textit{Construction of the $G$-equivariant tautological model $f : \sX \to B$.}

Let $\nu' : ({B},{\gD}) \to ({B}_n,\gD_n)$ be a $G$-equivariant log-resolution of singularities of the pair $ ({B}_n,\gD_n)$, where we recall that  $\gD_n \subset B_n$ is the pre-image of  $\gD_0 \subset B_0$, the discriminant locus of $f_0 : X_0 \to B_0 $. Let $f : \sX \to B$ be the $G$-equivariant tautological model associated to the $G$-equivariant elliptic fibration $X_0 \times_{B_0} B = X'_n \times_{B_n} B \to B$ by virtue of Lemma-Definition~\ref{lem-locWeimodHdg}. Since $\gD = (\nu \circ \nu')^{-1}(\gD_0)$ is a normal crossing divisor and the local monodromies of $\bH_0$ around $\gD_0$ are unipotent, by Lemma~\ref{lem-functW} the pullback of the minimal Weierstra{\ss} fibration associated to $\bH_0$ by $\nu \circ \nu'$ is still a minimal Weierstra{\ss} fibration. Therefore by Lemma~\ref{lem-dicstaut}, the discriminant locus of $f$ is $\gD$. As $({B},{\gD})$ is log-smooth and projective, $f$ satisfies $i)$ and $iii)$ in Hypotheses~\ref{hyp-KahlerSection}. Since $f$ is the pullback of $f_0$ by $\nu \circ \nu'$ and $f_0$ satisfies $ii)$ and $iv)$ in Hypotheses~\ref{hyp-KahlerSection}, so does $f$.  \qed

 \noindent \textbf{Step 5:}  \textit{Construction of the $G$-stable subvariety $Z \subset B$ and the end of the proof of Lemma~\ref{lem-red}.}
 
Let $Z \cnec \nu'^{-1}(Z_n)$. As $\gD_n \subset Z_n$, we have $\gD \subset Z$. Since $f_n'^{-1}(\ol{Z_n \bss \gD_n}) \to \ol{Z_n \bss \gD_n}$ has a multi-section by Step 3 and $f$ is bimeromorphic to $X'_n \times_{B_n} B \to B$ over $B \bss \gD$, the fibration $f^{-1}(\ol{Z \bss \gD}) \to \ol{Z \bss \gD}$ has a multi-section.
 
Let $f''_n : X_n'' \to B_n$ be a $G$-equivariant minimal resolution of the $G$-equivariant meromorphic map $f_n : X_n \dto B_n$ and let $f'' : X'' \cnec X''_n \times_{B_n} B \to B$ be the base change of $f''_n$. Here "minimal" resolution means that over the complementary of the indeterminacy locus of  $f_n : X_n \dto B_n$, the bimeromorphic map $X_n'' \to X_n$ is biholomorphic. As the resolution is minimal and $X_n \dto B_n$ is almost holomorphic and well-defined over $B_n \bss Z_n$, the map $f''_n$ is ($G$-equivariantly) biholomorphic to $f_n$ over $B_n \bss Z_n$, which is further ($G$-equivariantly) biholomorphic to $f'_n$ over $B_n \bss Z_n$ by Step 3.   It follows that the induced  bimeromorphic map $\ga : X'' \dto X_0 \times_{B_0} B$ is biholomorphic over $B \bss Z$.  In particular, ${f''}^{-1}(B \bss Z) \to B \bss Z$ is a smooth elliptic fibration because $f_0 : X_0 \to B_0$ is smooth over $B_0 \bss \gD_0$ and $(\nu \circ \nu')^{-1}(\gD_0) = \gD \subset Z$. Therefore the singular locus of $X''$ is contained in ${f''}^{-1}(Z)$, so if $\mu' : X' \to X''$ is a $G$-equivariant desingularization of $X''$, then $\mu'$ is biholomorphic over $B\bss Z$. By Lemma~\ref{lem-isomlisse}, there exists a $G$-equivariant bimeromorphic map $\gb : X_0 \times_{B_0} B \dto \sX$ over $B$ that is biholomorphic over $B \bss \gD \supset B \bss Z$.  Thus if we define the  bimeromorphic map $\tau$ in the statement of Lemma~\ref{lem-red} to be the composition of the bimeromorphic maps
$$
\begin{tikzcd}[cramped, row sep = 20, column sep = 25]
\tau : X'   \ar[r, "\mu' "]  & X''  \ar[r, dashed, "\ga "] & X_0 \times_{B_0} B \ar[r, dashed, "\gb"] & \sX,
\end{tikzcd}
$$
then as $\ga$, $\gb$, and $\mu'$ are G-equivariant and biholomorphic over $B \bss Z$, so is $\tau$.  Hence $\tau$ satisfies the desired property in Lemma~\ref{lem-red}.

Finally we define the holomorphic map $\mu$ in the statement of Lemma~\ref{lem-red} to be the composition of the bimeromorphic morphisms  
$$\mu: X'/G \xto{\mu'/G} X''/G \to X_n''/G \to X_n/G \xto{\mu_n} X,$$
where the second arrow is the quotient by $G$ of the projection $X'' = X''_n \times_{B_n} B \to X''_n$ and the third arrow the quotient of the $G$-equivariant resolution $X''_n \to X_n$ of the map $X_n \dto B_n$ introduced in Step 4.
\end{proof}

By Lemma~\ref{lem-red} we have a diagram 
$$
\begin{tikzcd}[cramped, row sep = 20, column sep = 25]
X    & X'/G \ar[l, swap, "\mu"] \ar[r, dashed, "\bar{\tau}"]     & \sX/G
\end{tikzcd}
$$
where both arrows are bimeromorphic. Since $X'/G$ is a finite quotient of a complex manifold, it has at worst rational singularities~\cite[Proposition 5.15]{KollarMori}. As $X$ is smooth, we have $\mu_*\cO_{X'/G} =\cO_X$ and $R^1\mu_*\cO_{X'/G} = 0$. It follows from~\cite[Theorem 2.1]{RanStabMap} that if $X'/G$ has an algebraic approximation, then it induces an algebraic approximation of $X$. Since $Z$ is $G$-stable and $f$ is $G$-equivariant, $Y \cnec f^{-1}(Z)$ is also $G$-stable. As $\tau^{-1}_{|\sX \bss Y}$ is biholomorphic onto its image, so is $\bar{\tau}^{-1}_{|(\sX/G) \bss (Y/G)}$.  Thus by Proposition~\ref{pro-Formalgapp}, it suffices to prove that there exists an algebraic approximation of $\sX/G$ preserving the completion $\rwh{Y/G}$ of $\sX/G$ along $Y/G$.

Since $f : \sX \to B$ is a $G$-equivariant tautological model satisfying Hypotheses~\ref{hyp-KahlerSection} and since $f^{-1}(\ol{Z \bss \gD}) \to \ol{Z \bss \gD}$ is a fibration over a $G$-stable subvariety of $B$ which has a multi-section by Lemma~\ref{lem-red}, according to Proposition~\ref{pro-appalgpres} there exists an algebraic approximation 
$$\Pi : \cX  \xto{q} B \times V' \to V'$$ 
of $f : \sX \to B$ which is $G$-equivariantly locally trivial over $B$ and preserves $G$-equivariantly the completion $\hat{Y} \to \hat{Z}$ of $f$ along $Y \cnec f^{-1}(Z) \to Z$. By Lemma~\ref{lem-Gquotloctriv}, the quotient $\cX/G \to V'$ of $\Pi$ is an algebraic approximation of $\sX/G$. As $\hat{Y}$ is $G$-equivariantly preserved by $\Pi$, we have a $G$-stable subvariety $\cY \subset \cX$ along which the completion $\hat{\cY}$ of $\cX$ is $G$-equivariantly isomorphic to $\hat{Y} \times V'$ over $V'$. So 
$$\rwh{\cY/G} \simeq \hat{\cY}/G  \simeq (\hat{Y}/G) \times V' \simeq \rwh{Y/G} \times V'$$
over $V'$, where $\rwh{\cY/G}$ is the completion of $\cX/G$ along $\cY/G$. It follows that the quotient $\cX/G \to V'$ of $\cX \to V'$ is an algebraic approximation of $\sX/G$ preserving $\rwh{Y/G}$, which finishes the proof of Theorem~\ref{thm-mainfibellip}. 
\end{proof}

\section*{Acknowledgement}

The theory of elliptic fibrations developed by Nakayama plays a crucial role in this work, and I had a more thorough understanding of this subject when collaborating with Beno\^it Claudon and Andreas H\"oring. I am grateful to them for the e-mail correspondence related to this subject and their comments on the preliminary version of the text. I am also in debt to Noboru Nakayama for explaining to me several details of his work~\cite{NakayamaGlob}. Finally, I thank the referees for their careful reading, remarks, and suggestions.

\bibliographystyle{plain}
\bibliography{Kod3kappa2v2}

\end{document}